\documentclass[10pt, a4paper, oneside, reqno]{amsart}

\usepackage{float,graphicx}
\usepackage{bbm}  
\RequirePackage[reqno]{amsmath}
\usepackage{cases}
\usepackage[ruled,vlined]{algorithm2e} 
\SetKwRepeat{Do}{do}{while}

\usepackage[usenames, dvipsnames]{color}
\definecolor{darkblue}{rgb}{0.0, 0.0, 0.45}

\usepackage{dsfont,amssymb,amsmath,subfigure, graphicx,enumitem} 
\usepackage{amsfonts,dsfont,mathtools, mathrsfs,amsthm,fancyhdr} 
\usepackage[amssymb, thickqspace]{SIunits}
\usepackage{url}

\makeatletter
\def\@settitle{\begin{center}%
		\baselineskip14\p@\relax
		\normalfont\LARGE\scshape\bfseries
		\@title
	\end{center}%
}
\makeatother
\makeatletter

\def\subsection{\@startsection{subsection}{2}%
	\z@{.5\linespacing\@plus.7\linespacing}{.5\linespacing}%
	{\normalfont\large\bfseries}}

\def\subsubsection{\@startsection{subsubsection}{3}%
	\z@{.5\linespacing\@plus.7\linespacing}{.5\linespacing}%
	{\normalfont\itshape}}

\usepackage{multirow}
\usepackage{bigstrut}

\makeatletter
\def\munderbar#1{\underline{\sbox\tw@{$#1$}\dp\tw@\z@\box\tw@}}
\makeatother

\newtheorem{theorem}{Theorem}[section]
\newtheorem{definition}[theorem]{Definition}
\newtheorem{lemma}[theorem]{Lemma}
\newtheorem{remark}[theorem]{Remark}
\newtheorem{corollary}[theorem]{Corollary}
\newtheorem{proposition}[theorem]{Proposition}

\newcommand{\be}{\begin{equation}}
\newcommand{\ee}{\end{equation}}
\newcommand{\bea}{\begin{equation*}\begin{aligned}}
\newcommand{\eea}{\end{aligned}\end{equation*}}
\newcommand{\ds}{\displaystyle}

\newcommand{\R}{\mathbb{R}}

\newcommand{\Min}{\min\limits_}
\newcommand{\Sup}{\sup\limits_}
\newcommand{\Inf}{\inf\limits_}
\newcommand{\Tr}[1]{\Trace \left[ #1 \right]}

\newcommand{\wh}{\widehat}
\newcommand{\mc}{\mathcal}
\newcommand{\mbb}{\mathbb}
\newcommand{\inner}[2]{\big \langle #1, #2 \big \rangle }
\newcommand{\norm}[1]{\big\| #1\big\| }
\newcommand{\cov}{\Sigma} 
\newcommand{\covsa}{\wh{\Sigma}} 

\newcommand{\trueP}{\mbb P}
\newcommand{\empiP}{\wh{\mbb P}_n}

\newcommand{\Ambi}{\mc P_\rho} 

\newcommand{\est}{X\opt} 
\newcommand{\estdual}{\dualvar\opt} 
\newcommand{\estx}{x\opt}

\DeclareMathOperator{\Trace}{Tr}
\DeclareMathOperator{\diag}{diag}

\DeclareMathOperator{\dd}{d}
\DeclareMathOperator{\st}{s.t.}

\DeclareMathOperator{\vect}{vec}

\newcommand{\PSD}{\mathbb{S}_{+}} 
\newcommand{\PD}{\mathbb{S}_{++}} 
\newcommand{\Let}{\coloneqq}
\newcommand{\opt}{^\star}
\newcommand{\eps}{\varepsilon}
\newcommand{\ra}{\rightarrow}
\newcommand{\X}{\mathcal{X}}
\newcommand{\Wass}{\mathds{W}}
\newcommand{\V}{\mathds{W}_S}
\newcommand{\Q}{\mbb{Q}}
\newcommand{\EE}{\mathds{E}}

\newcommand{\Dg}{\Delta_\gamma}
\newcommand{\DX}{\Delta_X}
\newcommand{\half}{\frac{1}{2}}

\newcommand{\eigval}{\lambda} 
\newcommand{\dualvar}{\gamma}

\newcommand{\covsaeig}{\lambda}
\newcommand{\covsaeigvect}{v}

\newcommand{\g}{g}
\newcommand{\J}{\mc J}

\title[Distributionally Robust Inverse Covariance Estimation]{Distributionally Robust Inverse Covariance Estimation:\\ The Wasserstein Shrinkage Estimator}

\author{Viet Anh Nguyen, Daniel Kuhn, Peyman Mohajerin Esfahani}

\thanks{The authors are with the Risk Analytics and Optimization Chair, EPFL, Switzerland ({\tt viet-anh.nguyen@epfl.ch}, {\tt daniel.kuhn@epfl.ch}) and the Delft Center for Systems and Control, Delft University of Technology, The Netherlands ({\tt P.MohajerinEsfahani@tudelft.nl}).}

\date{\today}

\begin{document}

\begin{abstract}
We introduce a distributionally robust maximum likelihood estimation model with a Wasserstein ambiguity set to infer the inverse covariance matrix of a $p$-dimensional Gaussian random vector from $n$ independent samples. The proposed model minimizes the worst case (maximum) of Stein's loss across all normal reference distributions within a prescribed Wasserstein distance from the normal distribution characterized by the sample mean and the sample covariance matrix. We prove that this estimation problem is equivalent to a semidefinite program that is tractable in theory but beyond the reach of general purpose solvers for practically relevant problem dimensions~$p$. In the absence of any prior structural information, the estimation problem has an analytical solution that is naturally interpreted as a nonlinear shrinkage estimator. Besides being invertible and well-conditioned even for $p>n$, the new shrinkage estimator is rotation-equivariant and preserves the order of the eigenvalues of the sample covariance matrix. These desirable properties are not imposed {\em ad hoc} but emerge naturally from the underlying distributionally robust optimization model. Finally, we develop a sequential quadratic approximation algorithm for efficiently solving the general estimation problem subject to conditional independence constraints typically encountered in Gaussian graphical models.
\end{abstract}
	
\maketitle

\section{Introduction}
The covariance matrix $\cov \Let \EE^{\trueP}[(\xi-\EE^{\trueP}[\xi])(\xi-\EE^{\trueP}[\xi])^\top]$ of a random vector $\xi \in \R^p$ governed by a distribution~$\trueP$ collects basic information about the spreads of all individual components and the linear dependencies among all pairs of components of $\xi$. The inverse $\cov^{-1}$ of the covariance matrix is called the {\em precision matrix}. This terminology captures the intuition that a large spread reflects a low precision and vice versa. While the covariance matrix appears in the {\em formulations} of many problems in engineering, science and economics, it is often the precision matrix that emerges in their {\em solutions}. For example, the optimal classification rule in linear discriminant analysis~\cite{ref:Fisher-1936}, the optimal investment portfolio in Markowitz' celebrated mean-variance model~\cite{ref:Markowitz-1952} or the optimal array vector of the beamforming problem in signal processing~\cite{ref:Du-2010} all depend on the precision matrix. Moreover, the optimal fingerprint method used to detect a multivariate climate change signal blurred by weather noise requires knowledge of the climate vector's precision matrix~\cite{ref:Ribes-2009}. 

If the distribution $\trueP$ of $\xi$ is known, then the covariance matrix~$\cov$ and the precision matrix~$\cov^{-1}$ can at least principally be calculated in closed form. In practice, however, $\trueP$ is never known and only indirectly observable through $n$ independent training samples $\wh{\xi}_1, \ldots, \wh{\xi}_n$ from $\trueP$. In this setting, $\cov$ and $\cov^{-1}$ need to be estimated from the training data. Arguably the simplest estimator for $\cov$ is the sample covariance matrix $\covsa \Let \frac{1}{n} \sum_{i=1}^n (\wh{\xi}_i-\wh \mu) (\wh{\xi}_i-\wh \mu)^{\top}$, where $\wh \mu \Let \frac{1}{n} \sum_{i=1}^n \wh{\xi}_i$ stands for the sample mean. Note that $\wh \mu$ and $\covsa$ simply represent the actual mean and covariance matrix of the uniform distribution on the training samples. For later convenience, $\covsa$ is defined here without Bessel's correction and thus constitutes a biased estimator.\footnote{An elementary calculation shows that $\EE^{\trueP^n}[\covsa] =\frac{n-1}{n}\cov$.} Moreover, as a sum of $n$ rank-$1$ matrices, $\covsa$ is rank deficient in the big data regime ($p> n$). In this case, $\covsa$ cannot be inverted to obtain a precision matrix estimator, which is often the actual quantity of interest. 

If $\xi$ follows a normal distribution with unknown mean $\mu$ and precision matrix $X\succ 0$, which we will assume throughout the rest of the paper, then the log-likelihood function of the training data can be expressed as
\begin{align}
	\wh{\mc L}(\mu,X) &\Let -\frac{np}{2} \log(2\pi) + \frac{n}{2} \log \det X - \frac{1}{2} \sum_{i=1}^n (\wh\xi_i-\mu)^\top X(\wh\xi_i-\mu) \notag \\
	&= -\frac{np}{2} \log(2\pi) + \frac{n}{2} \log \det X - \frac{n}{2} \Tr{\covsa X}- \frac{n}{2} (\wh{\mu}- \mu)^{\top}X (\wh{\mu}- \mu). \label{eq:log-likelihood}
\end{align}
Note that $\wh{\mc L}(\mu,X)$ is strictly concave in $\mu$ and $X$ \cite[Chapter~7]{ref:Boyd-04} and depends on the training samples only through the sample mean and the sample covariance matrix. It is clear from the last expression that $\wh{\mc L}(\mu,X)$ is maximized by $\mu\opt=\wh\mu$ for any fixed $X$. The maximum likelihood estimator $\est$ for the precision matrix is thus obtained by maximizing $\wh{\mc L}(\wh \mu,X)$ over all $X\succ 0$, which is tantamount to solving the convex program
\begin{equation}
	\label{eq:nominal-problem}
	\inf_{X\succ 0} - \log \det X + \Tr{\covsa X}.
\end{equation}
If $\covsa$ is rank deficient, which necessarily happens for $p>n$, then problem~\eqref{eq:nominal-problem} is unbounded. Indeed, expressing the sample covariance matrix as $\covsa = R\Lambda R^\top$ with $R$ orthogonal and $\Lambda\succeq 0$ diagonal, we may set $X_k=R \Lambda_k R^\top$ for any $k\in\mathbb N$, where $\Lambda_k\succ 0$ is the diagonal matrix with $(\Lambda_k)_{ii}=1$ if $\covsaeig_i>0$ and $(\Lambda_k)_{ii}=k$ if $\covsaeig_i=0$. By construction, the objective value of $X_k$ in~\eqref{eq:nominal-problem} tends to $-\infty$ as $k$ grows. If $\covsa$ is invertible, on the other hand, then the first-order optimality conditions can be solved analytically, showing that the minimum of problem~\eqref{eq:nominal-problem} is attained at $\est=\covsa^{-1}$. This implies that maximum likelihood estimation under normality simply recovers the sample covariance matrix but fails to yield a precision matrix estimator for $p>n$.

Adding an $\ell_1$-regularization term to its objective function guarantees that problem~\eqref{eq:nominal-problem} has a unique minimizer $\est\succ 0$, which constitutes a proper (invertible) precision matrix estimator~\cite{ref:Hsieh-2014}. Moreover, as the $\ell_1$-norm represents the convex envelope of the cardinality function on the unit hypercube, the $\ell_1$-norm regularized maximum likelihood estimation problem promotes sparse precision matrices that encode interpretable Gaussian graphical models~\cite{ref:Banerjee-2008, ref:Friedman-2008}. Indeed, under the given normality assumption one can show that $X_{ij}=0$ if and only if the random variables $\xi_i$ and $\xi_j$ are conditionally independent given $\{\xi_k\}_{k\notin\{i,j\}}$ \cite{ref:Lauritzen-1996}. The sparsity pattern of the precision matrix $X$ thus captures the conditional independence structure of $\xi$. 

In theory, the $\ell_1$-norm regularized maximum likelihood estimation problem can be solved in polynomial time via modern interior point algorithms. In practice, however, scalability to high dimensions remains challenging due to the problem's semidefinite nature, and larger problem instances must be addressed with special-purpose methods such as the Newton-type QUIC algorithm \cite{ref:Hsieh-2014}. 

Instead of penalizing the $\ell_1$-norm of the precision matrix, one may alternatively penalize its inverse $X^{-1}$ with the goal of promoting sparsity in the covariance matrix and thus controlling the {\em marginal} independence structure of $\xi$ \cite{ref:BienTibshirani-2011}. Despite its attractive statistical properties, this alternative model leads to a hard non-convex and non-smooth optimization problem, which can only be solved approximately. 


By the Fisher-Neyman factorization theorem, $\covsa$ is a sufficient statistic for the true covariance matrix $\cov$ of a normally distributed random vector, that is, $\covsa$ contains the same information about $\cov$ as the entire training dataset. Without any loss of generality, we may thus focus on estimators that depend on the data only through $\covsa$. If neither the covariance matrix $\cov$ nor the precision matrix $\cov^{-1}$ are known to be sparse and if there is no prior information about the orientation of their eigenvectors, it is reasonable to restrict attention to {\em rotation equivariant} estimators. 
A precision matrix estimator $\wh X(\covsa)$ 
is called rotation equivariant if 
$\wh X(R\covsa R^\top) = R \wh X(\covsa) R^{\top}$ for any rotation matrix~$R$. This definition requires that the estimator for the rotated data coincides with the rotated estimator for the original data. One can show that rotation equivariant estimators have the same eigenvectors as the sample covariance matrix (see, {\em e.g.}, \cite[Lemma~5.3]{ref:Perlman-07} for a simple proof) and are thus uniquely determined by their eigenvalues. Hence, imposing rotation equivariance reduces the degrees of freedom from $p(p+1)/2$ to $p$. Using an entropy loss function introduced in~\cite{ref:Stein-1961}, Stein was the first to demonstrate that superior covariance estimators in the sense of statistical decision theory can be constructed by shrinking the eigenvalues of the sample covariance matrix~\cite{ref:Stein-1975, ref:Stein-1986}. Unfortunately, his proposed shrinkage transformation may alter the order of the eigenvalues and even undermine the positive semidefiniteness of the resulting estimator when $p>n$, which necessitates an {\em ad hoc} correction step involving an isotonic regression. Various refinements of this approach are reported in \cite{ref:Dey-1985, ref:Haff-91, ref:Yang-94} and the references therein, but most of these works focus on the low-dimensional case when~$n \ge p$.

Jensen's inequality suggests that the largest (smallest) eigenvalue of the sample covariance matrix $\covsa$ is biased upwards (downwards), which implies that $\covsa$ tends to be ill-conditioned \cite{ref:vanderVaart-61}. This effect is most pronounced for $\cov\approx I$. A promising shrinkage estimator for the covariance matrix is thus obtained by forming a convex combination of the sample covariance matrix and the identity matrix scaled by the average of the sample eigenvalues \cite{ref:Ledoit-2004}. If its convex weights are chosen optimally in view of the Frobenius risk, the resulting shrinkage estimator can be shown to be both well-conditioned and more accurate than $\covsa$. Alternative shrinkage targets 
include the constant correlation model, which preserves the sample variances but equalizes all pairwise correlations \cite{ref:Ledoit-2004:honey}, the single index model, which assumes that each random variable is explained by one systematic and one idiosyncratic risk factor \cite{ref:Ledoit-03}, or the diagonal matrix of the sample eigenvalues~\cite{ref:Touloumis-15}~etc. 

The {\em linear} shrinkage estimators described above are computationally attractive because evaluating convex combinations is cheap. Computing the corresponding precision matrix estimators requires a matrix inversion and is therefore more expensive. We emphasize that linear shrinkage estimators for the precision matrix itself, obtained by forming a cheap convex combination of the inverse sample covariance matrix and a shrinkage target, are not available in the big data regime when $p>n$ and $\covsa$ fails to be invertible.

More recently, insights from random matrix theory have motivated a new rotation equivariant shrinkage estimator that applies an individualized shrinkage intensity to every sample eigenvalue \cite{ref:Ledoit-2012}. While this {\em nonlinear} shrinkage estimator offers significant improvements over linear shrinkage, its evaluation necessitates the solution of a hard nonconvex optimization problem, which becomes cumbersome for large values of~$p$. Alternative nonlinear shrinkage estimators can be obtained by imposing an upper bound on the condition number of the covariance matrix in the underlying maximum likelihood estimation problem \cite{ref:Won-2013}.

Alternatively, multi-factor models familiar from the arbitrage pricing theory can be used to approximate the covariance matrix by a sum of a low-rank and a diagonal component, both of which have only few free parameters and are thus easier to estimate. Such a dimensionality reduction leads to stable estimators~\cite{ref:Chun-18, ref:Fan-2008}.

This paper endeavors to develop a principled approach to precision matrix estimation, which is inspired by recent advances in distributionally robust optimization \cite{ref:DelYe-10, ref:GohSim-10, ref:WieKuhSim-14}. For the sake of argument, assume that the true distribution of $\xi$ is given by $\trueP= \mc N(\mu_0,\cov_0)$, where $\cov_0\succ 0$. If $\mu_0$ and $\cov_0$ were known, the quality of some estimators $\mu$ and $X$ for $\mu_0$ and $\cov_0^{-1}$, respectively, could conveniently be measured by Stein's loss~\cite{ref:Stein-1961}
\begin{align}
	L(X,\mu) & \Let -\log \det (\cov_0 X) + \Tr{\cov_0 X}+(\mu_0 - \mu)^\top X(\mu_0-\mu)-p \notag \\
	& = -\log \det X + \EE^{\trueP} \left[(\xi-\mu)^\top X (\xi-\mu) \right] -\log \det \cov_0-p, \label{eq:stein-loss}
\end{align}
which is reminiscent of the log-likelihood function~\eqref{eq:log-likelihood}. It is easy to verify that Stein's loss is nonnegative for all $\mu\in\R^p$ and $X\in\PSD^p$ and vanishes only at the true mean $\mu=\mu_0$ and the true precision matrix~$X=\cov^{-1}_0$. Of course, we cannot minimize Stein's loss directly because $\trueP$ is unknown. As a na\"ive remedy, one could instead minimize an approximation of Stein's loss obtained by removing the (unknown but irrelevant) normalization constant $-\log\det\cov_0-p$ and replacing $\trueP$ in~\eqref{eq:stein-loss} with the empirical distribution $\empiP= \mc N(\wh \mu, \covsa)$. However, in doing so we simply recover the standard maximum likelihood estimation problem, which is unbounded for $p>n$ and outputs the sample mean and the inverse sample covariance matrix for $p\le n$. This motivates us to robustify the empirical loss minimization problem by exploiting that $\empiP$ is close to $\trueP$ in Wasserstein distance.

\begin{definition}[Wasserstein distance]
	The type-2 Wasserstein distance between two arbitrary distributions $\trueP_1$ and $\trueP_2$ on $\R^p$ with finite second moments is defined as
	\[
	\Wass(\trueP_1, \trueP_2) \Let \Inf{\Pi} \left\{ 
	\left(
	\int_{\R^p\times \R^p} \norm{\xi_1 - \xi_2}^2\, \Pi(\rm{d}\xi_1, \rm{d} \xi_2) 
	\right)^{\frac{1}{2}}
	: \begin{array}{l}
	\Pi \text{ is a joint distribution of $\xi_1$ and $\xi_2$} \\
	\text{with marginals $\trueP_1$ and $\trueP_2$, respectively}
	\end{array}
	\right\}.
	\]
\end{definition}

The squared Wasserstein distance between $\trueP_1$ and $\trueP_2$ can be interpreted as the cost of moving the distribution $\trueP_1$ to the distribution $\trueP_2$, where $\|\xi_1 - \xi_2\|^2$ quantifies the cost of moving unit mass from~$\xi_1$ to~$\xi_2$. 

A central limit type theorem for the Wasserstein distance between empirical normal distributions implies that $n\cdot\Wass(\empiP, \trueP)^2$ converges weakly to a quadratic functional of independent normal random variables as the number $n$ of training samples tends to infinity \cite[Theorem~2.3]{ref:Rippl-16}. We may thus conclude that for every $\eta\in(0,1)$ there exists $q(\eta)> 0$ such that $\trueP^n[\Wass(\empiP, \trueP)\le q(\eta)n^{-\frac{1}{2}}]\ge 1-\eta$ for all $n$ large enough. In the following we denote by $\mc N^p$ the family of all normal distributions on $\R^p$ and by
\[
	\Ambi = \{ \Q \in \mc N^p : \Wass(\empiP, \Q) \leq \rho\}
\]
the ambiguity set of all normal distributions whose Wasserstein distance to $\empiP$ is at most $\rho\ge 0$. Note that~$\Ambi$ depends on the unknown true distribution $\trueP$ only through the training data and, for $\rho\ge q(\eta) n^{-\frac{1}{2}}$, contains~$\trueP$ with confidence $1-\eta$ asymptotically as $n$ tends to infinity. It is thus natural to formulate a {\em distributionally robust} estimation problem for the precision matrix that minimizes Stein's loss---modulo an irrelevant normalization constant---in the worst case across all reference distributions $\Q\in\Ambi$.
\begin{align}
	\label{eq:dro}
	\J(\wh \mu, \covsa) \Let \inf_{\mu\in\R^p,\,X\in\mc X} \left\{ -\log \det X + \Sup{\Q \in \Ambi}\EE^{\Q} \left[(\xi-\mu)^\top X (\xi-\mu) \right] \right\}
\end{align}
Here, $\mc X\subseteq \PD^p$ denotes the set of admissible precision matrices. In the absence of any prior structural information, the only requirement is that $X$ be positive semidefinite and invertible, in which case~$\mc X=\PD^p$. Known conditional independence relationships impose a sparsity pattern on $X$, which is easily enforced through linear equality constraints in $\X$. By adopting a worst-case perspective, we hope that the minimizers of~\eqref{eq:dro} will have low Stein's loss with respect to all distributions in~$\Ambi$ including the unknown true distribution~$\trueP$. As Stein's loss with respect to the empirical distribution is proportional to the log-likelihood function~\eqref{eq:log-likelihood}, problem~\eqref{eq:dro} can also be interpreted as a robust maximum likelihood estimation problem that hedges against perturbations in the training samples. As we will show below, this robustification is tractable and has a regularizing effect.

Recently it has been discovered that distributionally robust optimization models with Wasserstein ambiguity sets centered at {\em discrete} distributions on $\R^p$ (and {\em without} any normality restrictions) are often equivalent to tractable convex programs \cite{ref:MohajerinEsfahani-2017, ref:zhao-15}. Extensions of these results to general Polish spaces are reported in~\cite{ref:blanchet-16, ref:Gao-16}. The explicit convex reformulations of Wasserstein distributionally robust models have not only facilitated efficient solution procedures but have also revealed insightful connections between distributional robustness and regularization in machine learning. Indeed, many classical regularization schemes of supervised learning such as the Lasso method can be explained by a Wasserstein distributionally robust model. This link was first discovered in the context of logistic regression \cite{ref:shafieezadeh-15} and later extended to other popular regression and classification models \cite{ref:blanchet-16, ref:shafieezadeh-17} and even to generative adversarial networks in deep learning~\cite{ref:gao-17}.

Model~\eqref{eq:dro} differs fundamentally from all existing distributionally robust optimization models in that the ambiguity set contains only normal distributions. As the family of normal distributions fails to be closed under mixtures, the ambiguity set is thus nonconvex. In the remainder of the paper we devise efficient solution methods for problem~\eqref{eq:dro}, and we investigate the properties of the resulting precision matrix estimator. 

The main contributions of this paper can be summarized as follows.
\begin{itemize}
\item Leveraging an analytical formula for the Wasserstein distance between two normal distributions derived in~\cite{ref:Givens-84}, we prove that the distributionally robust estimation problem~\eqref{eq:dro} is equivalent to a tractable semidefinite program---despite the nonconvex nature of the underlying ambiguity set. 
\item We prove that problem~\eqref{eq:dro} and its unique minimizer depend on the training data only through $\covsa$ (but not through $\wh\mu$), which is reassuring because $\covsa$ is a sufficient statistic for the precision matrix.
\item In the absence of any structural information, we demonstrate that problem~\eqref{eq:dro} has an analytical solution that is naturally interpreted as a nonlinear shrinkage estimator. Indeed, the optimal precision matrix estimator shares the eigenvectors of the sample covariance matrix, and as the radius $\rho$ of the Wasserstein ambiguity set grows, its eigenvalues are shrunk towards 0 while preserving their order. At the same time, the condition number of the optimal estimator steadily improves and eventually converges to~1 even for~$p>n$. These desirable properties are not enforced {\em ex ante} but emerge naturally from the underlying distributionally robust optimization model.
\item In the presence of conditional independence constraints, the semidefinite program equivalent to~\eqref{eq:dro} is beyond the reach of general purpose solvers for practically relevant problem dimensions~$p$. We thus devise an efficient sequential quadratic approximation method reminiscent of the QUIC algorithm~\cite{ref:Hsieh-2014}, which can solve instances of problem~\eqref{eq:dro} with $p\lesssim 10^4$ on a standard PC.
\item We derive an analytical formula for the extremal distribution that attains the supremum in~\eqref{eq:dro}.
\end{itemize}

The paper is structured as follows. Section~\ref{sec:tractability} demonstrates that the distributionally robust estimation problem~\eqref{eq:dro} admits an exact reformulation as a tractable semidefinite program. Section~\ref{sec:analytical} derives an analytical solution of this semidefinite program in the absence of any structural information, while Section~\ref{sec:numerical} develops an efficient sequential quadratic approximation algorithm for the problem with conditional independence constraints. The extremal distribution that attains the worst-case expectation in~\eqref{eq:dro} is characterized in Section~\ref{sec:wc-dist}, and numerical experiments based on synthetic and real data are reported in Section~\ref{sec:num-res}.

\textbf{Notation.} For any $A\in \R^{p\times p}$ we use $\Tr{A}$ to denote the trace and $\|A\|=\sqrt{\Tr{A^\top A}}$ to denote the Frobenius norm of $A$. By slight abuse of notation, the Euclidean norm of $v\in\R^p$ is also denoted by $\|v\|$. Moreover, $I$ stands for the identity matrix. Its dimension is usually evident from the context. For any $A,B\in\R^{p\times p}$, we use $\inner{A}{B} = \Tr{A^\top B}$ to denote the inner product and $A\otimes B\in\R^{p^2\times p^2}$ to denote the Kronecker product of $A$ and $B$. The space of all symmetric matrices in $\R^{p\times p}$ is denoted by $\mathbb S^p$. We use $\PSD^p$ ($\mathbb S_{++}^p$) to represent the cone of symmetric positive semidefinite (positive definite) matrices in $\mathbb S^p$. For any $A,B\in\mathbb S^p$, the relation $A\succeq B$ ($A\succ B$) means that $A-B\in\PSD^p$ ($A-B\in\mathbb S^p_{++}$).


\section{Tractable Reformulation}
\label{sec:tractability}
Throughout this paper we assume that the random vector $\xi\in \R^p$ is normally distributed. This is in line with the common practice in statistics and in the natural and social sciences, whereby normal distributions are routinely used to model random vectors whose distributions are unknown. The normality assumption is often justified by the central limit theorem, which suggests that random vectors influenced by many small and unrelated disturbances are approximately normally distributed. Moreover, the normal distribution maximizes entropy across all distributions with given first- and second-order moments, and as such it constitutes the least prejudiced distribution compatible with a given mean vector and covariance matrix.

In order to facilitate rigorous statements, we first provide a formal definition of normal distributions.

\begin{definition}[Normal distributions]
\label{def:normal-dist} 
	We say that $\mbb P$ is a normal distribution on $\R^p$ with mean $\mu\in\R^p$ and covariance matrix $\cov\in\PSD^p$, that is, $\mbb P=\mc N(\mu,\cov)$, if $\mbb P$ is supported on $\text{\rm supp}(\mbb P)=\{\mu+Ev: v\in\R^k\}$, and if the density function of $\mbb P$ with respect to the Lebesgue measure on $\text{\rm supp}(\mbb P)$ is given by
	\[
	\varrho_{\mbb P}(\xi) \Let \frac{1}{\sqrt{(2\pi)^k\det(D)}} e^{-(\xi-\mu)^\top ED^{-1}E^\top(\xi-\mu)},
	\]
	where $k=\text{\rm rank}(\cov)$, $D\in \PD^k$ is the diagonal matrix of the positive eigenvalues of $\cov$, and $E\in\R^{p\times k}$ is the matrix whose columns correspond to the orthonormal eigenvectors of the positive eigenvalues of $\cov$. The family of all normal distributions on $\R^p$ is denoted by $\mc N^p$, while the subfamily of all distributions in $\mc N^p$ with zero means and arbitrary covariance matrices is denoted by~$\mc N^p_0$.
\end{definition}

Definition~\ref{def:normal-dist} explicitly allows for degenerate normal distributions with rank deficient covariance matrices.


The normality assumption also has distinct computational advantages. In fact, while the Wasserstein distance between two generic distributions is only given implicitly as the solution of a mass transportation problem, the Wasserstein distance between two normal distributions is known in closed form. It can be expressed explicitly as a function of the mean vectors and covariance matrices of the two distributions.

\begin{proposition}[Givens and Shortt {\cite[Proposition~7]{ref:Givens-84}}]
	\label{prop:Wass}
	 The type-2 Wasserstein distance between two normal distributions $\mbb P_1=\mc N(\mu_1,\cov_1)$ and $\mbb P_2= \mc N(\mu_2,\cov_2)$ with $\mu_1,\mu_2\in\R^p$ and $\cov_1,\cov_2\in\PSD^p$ amounts to
	\[ 
	\Wass(\mbb P_1, \mbb P_2) = \sqrt{\norm{\mu_1 - \mu_2}^2 + \Tr{\cov_1} + \Tr{\cov_2} - 2\Tr{ \sqrt{\sqrt{\cov_2} \cov_1 \sqrt{\cov_2}}}}.
	\]
\end{proposition}

If $\mbb P_1$ and $\mbb P_2$ share the same mean vector ({\em e.g.}, if $\mu_1=\mu_2=0$), then the Wasserstein distance $\Wass(\mbb P_1, \mbb P_2)$ reduces to a function of the covariance matrices $\cov_1$ and $\cov_2$ only, thereby inducing a metric on the cone~$\PSD^p$.

\begin{definition}[Induced metric on $\PSD^p$]
	\label{def:Wass-for-S}
	 Let $\V:\PSD^p\times \PSD^p\rightarrow \R_+$ be the metric on $\PSD^p$ induced by the type-2 Wasserstein metric on the family of normal distributions with equal means. Thus, for all $\cov_1,\cov_2\in\PSD^p$ we set
	\[ 
	\V(\cov_1, \cov_2) \Let \sqrt{\Tr{\cov_1} + \Tr{\cov_2} - 2\Tr{ \sqrt{\sqrt{\cov_2} \cov_1 \sqrt{\cov_2}}}}.
	\]
\end{definition}

The definition of $\V$ implies via Proposition~\ref{prop:Wass} that $\Wass(\mbb P_1, \mbb P_2)=\V(\cov_1, \cov_2)$ for all $\mbb P_1=\mc N(\mu_1,\cov_1)$ and $\mbb P_2= \mc N(\mu_2,\cov_2)$ with $\mu_1=\mu_2$. Thanks to its interpretation as the restriction of $\Wass$ to the space of normal distributions with a fixed mean, it is easy to verify that $\V$ is symmetric and positive definite and satisfies the triangle inequality. In other words, $\V$ inherits the property of being a metric from $\Wass$.

\begin{corollary}[Commuting covariance matrices]
	\label{cor:commuting}
	If $\cov_1, \cov_2\in\PSD^p$ commute ($\cov_1 \cov_2= \cov_2 \cov_1$), then the induced Wasserstein distance $\V$ simplifies to the trace norm between the square roots of $\cov_1$ and $\cov_2$, that is, 
	\[
	\V(\cov_1, \cov_2) = \norm{\sqrt{\cov_1} - \sqrt{\cov_2}}.
	\]
\end{corollary}
\begin{proof}
The commutativity of $\cov_1$ and $\cov_2$ implies that $\sqrt{\cov_2} \cov_1 \sqrt{\cov_2}=\cov_1\cov_2$, whereby
\begin{align*}
	\V(\cov_1, \cov_2) &= \sqrt{\Tr{\cov_1} + \Tr{\cov_2} - 2\Tr{ \sqrt{\cov_1 \cov_2}}} =\sqrt{\Tr{\left (\sqrt{\cov_1}- \sqrt{\cov_2} \right)^2}}= \norm{\sqrt{\cov_1} - \sqrt{\cov_2}}.
\end{align*}
Thus, the claim follows.
\end{proof}

Proposition~\ref{prop:Wass} reveals that the Wasserstein distance between any two (possibly degenerate) normal distributions is finite. In contrast, the Kullback-Leibler divergence between degenerate and non-degenerate normal distributions is infinite.

\begin{remark}[Kullback-Leibler divergence between normal distributions] A simple calculation shows that the Kullback-Leibler divergence from $\mbb P_2= \mc N(\mu_2,\cov_2)$ to $\mbb P_1=\mc N(\mu_1,\cov_1)$ amounts to
		\[
		D_{\rm KL}(\mbb P_1 \| \mbb P_2) = \frac{1}{2} \left[ (\mu_2 - \mu_1)^\top \cov_2^{-1} (\mu_2 - \mu_1) + \Tr{\cov_1 \cov_2^{-1}}  - p - \log \det \cov_1 + \log \det \cov_2 \right]
		\]
		whenever $\mu_1,\mu_2\in\R^p$ and $\cov_1,\cov_2\in\PD^p$. If either $\mbb P_1$ or $\mbb P_2$ is degenerate (that is, if $\cov_1$ is singular and $\cov_2$ invertible or vice versa), then $\mbb P_1$ fails to be absolutely continuous with respect to $\mbb P_2$, which implies that $D_{\rm KL}(\mbb P_1 \| \mbb P_2)=\infty$. Moreover, from the above formula it is easy to verify that $D_{\rm KL}(\mbb P_1 \| \mbb P_2)$ diverges if either $\cov_1$ or $\cov_2$ tends to a singular matrix. 
\end{remark}

In the big data regime ($p>n$) the sample covariance matrix $\covsa$ is singular even if the samples are drawn from a non-degernerate normal distribution $\trueP= \mc N(\mu,\cov)$ with $\cov\in\PD^p$. In this case, the Kullback-Leibler distance between the empirical distribution $\wh{\mbb P}= \mc N(\wh \mu,\covsa)$ and $\trueP$ is infinite, and thus $\wh{\mbb P}$ and $\trueP$ are perceived as maximally dissimilar despite their intimate relation. In contrast, their Wasserstein distance is finite.

In the remainder of this section we develop a tractable reformulation for the distributionally robust estimation problem~\eqref{eq:dro}. Before investigating the general problem, we first address a simpler problem variant where the true mean $\mu_0$ of $\xi$ is known to vanish. Thus, we temporarily assume that $\xi$ follows $\mc N(0,\cov_0)$. In this setting, it makes sense to focus on the modified ambiguity set $\Ambi^0 \Let \{ \Q \in \mc N^p_0:\Wass(\Q, \wh{\mbb P}) \leq \rho\}$, which contains all normal distributions with zero mean that have a Wasserstein distance of at most $\rho\ge 0$ from the empirical distribution $\wh{\mbb P}= \mc N(0,\covsa)$. Under these assumptions, the estimation problem~\eqref{eq:dro} thus simplifies to
\be
\label{eq:DROSimplified}
\begin{array}{clcl}
	\J(\covsa) \Let 
	\Inf{X \in \X} \left\{ -\log \det X + \Sup{ \Q \in \Ambi^0} \EE^\Q[ \inner{\xi \xi^\top}{X}] \right\}. 
\end{array}
\ee

We are now ready to state the first main result of this section.

\begin{theorem}[Convex reformulation]
	\label{thm:refor}
	For any fixed $\rho> 0$ and $\covsa\succeq 0$, the simplified distributionally robust estimation problem~\eqref{eq:DROSimplified} is equivalent to 
	\begin{equation}
	\label{eq:Reformulation2} 
	\J (\covsa) = 
	\left\{
	\begin{array}{cl}
	\Inf{X, \dualvar} & -\log \det X  + \dualvar \left( \rho^2 - \Tr{\covsa} \right) + \dualvar^2 \inner{(\dualvar I - X)^{-1}}{\covsa} \\
	\st & \dualvar I \succ X \succ 0, \quad X \in \X \,.
	\end{array}
	\right.
	\end{equation}
	Moreover, the optimal value function~$\J(\covsa)$ is continuous in $\covsa \in \PSD$.
\end{theorem}
The proof of Theorem~\ref{thm:refor} relies on several auxiliary results. A main ingredient to derive the convex program~\eqref{eq:Reformulation2} is a reformulation of the worst-case expectation function $\g:\PSD \times \PSD \ra \R$ defined through
\be
\label{eq:g:def}
\g (\covsa, X) \Let \Sup{ \Q \in \Ambi^0} \EE^\Q[ \inner{\xi \xi^\top}{X}]\,.
\ee
In Proposition~\ref{prop:g-refor} below we will demonstrate that $\g (\covsa, X)$ is continuous and coincides with the optimal value of an explicit semidefinite program, a result  which depends on the following preparatory lemma. 

\begin{lemma}[Continuity properties of partial infima]
	\label{lem:cont}
	Consider a function $\varphi:\mc E \times \Gamma \ra \R$ on two normed spaces $\mc E$ and $\Gamma$, and define the partial infimum with respect to $\gamma$ as~$\Phi(\eps) \Let \inf_{\dualvar \in \Gamma} \varphi(\eps,\dualvar)$ for every $\eps\in\mc E$.
	\begin{itemize}
		\item[(i)] If $\varphi (\eps,\dualvar)$ is continuous in $\eps$ at $\eps_0 \in \mc E$ for every $\dualvar \in \Gamma$, then~$\Phi(\eps)$ is upper-semicontinuous at $\eps_0$.
		
		\item[(ii)] If $\varphi (\eps,\dualvar)$ is calm from below at $\eps_0 \in \mc E$ uniformly in $\dualvar \in \Gamma$, that is, if there exists a constant $L\ge 0$ such that  $\varphi(\eps,\dualvar)- \varphi(\eps_0,\dualvar) \ge -L\|\eps_0 - \eps\|$ for all $\dualvar \in \Gamma$, then $\Phi(\eps)$ is lower-semicontinuous at~$\eps_0$.
	\end{itemize}
\end{lemma}

\begin{proof}
	As for assertion~(i), we have
	\begin{align*}
		\limsup_{\eps \ra \eps_0} \Phi(\eps) & = \Inf{\delta >0} \Sup{\|\eps-\eps_0\|\le \delta} \Phi(\eps) =  \Inf{\delta >0} \Sup{\|\eps-\eps_0\|\le \delta} \Inf{\dualvar \in\Gamma} \varphi(\eps,\gamma) \\ 
		&\le \Inf{\dualvar \in\Gamma} \Inf{\delta >0} \Sup{\|\eps-\eps_0\|\le \delta} \varphi(\eps,\dualvar) = \Inf{\dualvar \in \Gamma}\limsup_{\eps \ra \eps_0} \varphi (\eps,\dualvar) = \Inf{\dualvar \in\Gamma} \varphi (\eps_0,\gamma) = \Phi(\eps_0),
	\end{align*}
	where the inequality follows from interchanging the infimum and supremum operators, while the penultimate equality in the last line relies on the continuity assumption. As for assertion~(ii), note that 
	\begin{align*}
		\liminf_{\eps \ra \eps_0} \Phi (\eps) & = \Sup{\delta >0} \Inf{\|\eps-\eps_0\|\le \delta} \Phi (\eps) =  \Sup{\delta >0} \Inf{\|\eps-\eps_0\|\le \delta} \Inf{\dualvar \in\Gamma} \varphi (\eps,\dualvar) = \Sup{\delta >0}\Inf{\dualvar \in\Gamma} \inf_{\|\eps-\eps_0\|\le \delta}  \varphi (\eps,\dualvar)\\
		& \ge \Sup{\delta >0} \Inf{\dualvar \in\Gamma} \inf_{\|\eps-\eps_0\|\le \delta} \big( \varphi(\eps_0,\gamma) - L\|\eps_0 - \eps\|\big) = \Sup{\delta >0} \Inf{\dualvar \in\Gamma} \big(\varphi(\eps_0,\gamma) - L\delta\big) \\
		& = \Inf{\dualvar \in\Gamma} \varphi (\eps_0,\gamma) = \Phi (\eps_0),
	\end{align*}
	where the inequality in the second line holds due to the calmness assumption.
\end{proof}

\begin{proposition}[Worst-case expectation function]
	\label{prop:g-refor}
	For any fixed $\rho > 0$, $\covsa\succeq 0$ and $X\succ 0$, the worst-case expectation $\g(\covsa,X)$ defined in~\eqref{eq:g:def} coincides with the optimal value of the tractable semidefinite program
	\be
	\label{eq:g:refor}  
	\begin{array}{cl}
		\Inf{\dualvar} & \dualvar \left(\rho^2 - \Tr{\covsa} \right) + \dualvar^2 \inner{(\dualvar I - X)^{-1}}{\covsa}\\
		\st &  \dualvar I \succ X .
	\end{array}
	\ee
	Moreover, the optimal value function $\g (\covsa, X)$ is continuous in $(\covsa, X)\in\PSD^p\times \PD^p$.
	
\end{proposition}

\begin{proof}
	Using the definitions of the worst-case expectation $\g (\covsa, X)$ and the ambiguity set $\Ambi^0$, we find
	\begin{align*}
	\g (\covsa, X) = \Sup{ \Q \in \Ambi^0} \inner{\EE^\Q[ \xi \xi^\top]}{X} = \Sup{S\in\PSD^p} \left\{ \inner{S}{X} : \V(S, \covsa) \le \rho \right\},
	\end{align*}
	where the second equality holds because the metric $\V$ on $\PSD^p$ is induced by the type-2 Wasserstein metric $\Wass$ on $\mc N^p_0$, meaning that there is a one-to-one correspondence between distributions $\mbb Q\in \mc N^p_0$ with $\Wass(\mbb Q,\wh{\mbb P})\le \rho$ and covariance matrices $S\in\PSD^p$ with $\V(S, \covsa) \le \rho$. The continuity of~$\g (\covsa, X)$ thus follows from Berge's maximum theorem \cite[pp.~115--116]{ref:Berge-63}, which applies because $\inner{S}{X}$ and $\V(S, \covsa)$ are continuous in ${(S,\covsa, X)\in\PSD^p \times \PSD^p\times \PD^p}$, while $\{ S\in\PSD^p : \V(S, \covsa) \le \rho\}$ is nonempty and compact for every $\covsa \in \PSD^p$ and $\rho> 0$.
	
	By the definition of the induced metric~$\V$ we then obtain
	\begin{align}
	\label{eq:g:induced-metric}
	\g (\covsa, X) =  \Sup{S \in \PSD^p } \left\{ \inner{S}{X} :  \Tr{\covsa} + \Tr{S} - 2\Tr {\sqrt{\covsa^{\frac{1}{2}} S \covsa^{\frac{1}{2}}}} \leq \rho^2 \right\}.
	\end{align}
	To establish the equivalence between~\eqref{eq:g:refor} and~\eqref{eq:g:induced-metric}, we first assume that $\covsa \succ 0$. The generalization to rank deficient sample covariance matrices will be addressed later. By dualizing the explicit constraint in~\eqref{eq:g:induced-metric} and introducing the constant matrix $M = \covsa^{\frac{1}{2}}$, which inherits invertibility from $\covsa$, we find
	\begin{align}
		\g(\covsa, X) =&\Sup{S \in \PSD^p} \Inf{\dualvar \geq 0} \inner{S}{X - \dualvar I} + 2 \dualvar \inner{\sqrt{M S M}}{I} + \dualvar \left(\rho^2 - \Tr{\covsa} \right) \notag \\
		= & \Inf{\dualvar \geq 0} \Sup{S \in \PSD^p} \inner{S}{X - \dualvar I} + 2 \dualvar \inner{\sqrt{M S M}}{I} + \dualvar \left( \rho^2 - \Tr{\covsa} \right) \notag \\
		=& \Inf{\dualvar \geq 0} \left\{ \dualvar \left(\rho^2 - \Tr{\covsa} \right) +  \Sup{B \in \PSD^p} \left\{ \inner{ B^2}{ M^{-1} (X - \dualvar I) M^{-1}} + 2 \dualvar \inner{B}{I} \right\} \right\} \label{eq:g:b}.
	\end{align}
	Here, the first equality exploits the identity $\Tr{A}=\langle A , I\rangle$ for any $A \in\R^{p\times p}$, the second equality follows from strong duality, which holds because $\covsa$ constitutes a Slater point for problem~\eqref{eq:g:induced-metric} when $\rho>0$, and the third equality relies on the substitution $B \leftarrow \sqrt{M S M}$, which implies that $S = M^{-1} B^2 M^{-1}$. Introducing the shorthand $\Delta = M^{-1}(X-\dualvar I) M^{-1}$ allows us to simplify the inner maximization problem over $B$ in~\eqref{eq:g:b} to
	\begin{align}
		\label{eq:subproblem-B}
		\Sup{B \in \PSD^p} \left\{ \inner{B^2}{\Delta} + 2 \dualvar \inner{B}{I} \right\} .
	\end{align}
	
	If $\Delta \not \prec 0$, then~\eqref{eq:subproblem-B} is unbounded. To see this, denote by $\overline \lambda(\Delta)$ the largest eigenvalue of $\Delta$ and by $\overline v$ a corresponding eigenvector. If $\overline\lambda(\Delta)>0$, then the objective value of $B_k= k\cdot \overline v \, \overline v^\top\succeq 0$ in~\eqref{eq:subproblem-B} grows quadratically with $k$. If $\overline\lambda(\Delta)=0$, then $\gamma>0$ for otherwise $X\preceq 0$ contrary to our assumption, and thus the objective value of $B_k$ in~\eqref{eq:subproblem-B} grows linearly with $k$. In both cases~\eqref{eq:subproblem-B} is indeed unbounded.
	
	If $\Delta \prec 0$, then~\eqref{eq:subproblem-B} becomes a convex optimization problem that can be solved analytically. Indeed, the objective function of~\eqref{eq:subproblem-B} is minimized by $B\opt = -\dualvar \Delta^{-1}$, which satisfies the first-order optimality condition 
	\be
	\label{eq:OptimalB}
	B\Delta + \Delta B + 2 \dualvar I = 0
	\ee
	and is strictly feasible in~\eqref{eq:subproblem-B} because $\Delta\prec 0$. Moreover, as \eqref{eq:OptimalB} is naturally interpreted as a continuous Lyapunov equation, its solution $B\opt$ can be shown to be unique; see, {\em e.g.},~\cite[Theorem~12.5]{ref:Hespanha-09}. We may thus conclude that $B\opt$ is the unique maximizer of~\eqref{eq:subproblem-B} and that the maximum of~\eqref{eq:subproblem-B} amounts to~$-\gamma^2\Tr{\Delta}$.
	
	Adding the constraint $\gamma I\succ X$ to the outer minimization problem in~\eqref{eq:g:b}, thus excluding all values of $\gamma$ for which $\Delta\not \prec 0$ and the inner supremum is infinite, and replacing the optimal value of the inner maximization problem with $-\gamma^2\Tr{\Delta}=\dualvar^2 \inner{(\dualvar I - X)^{-1}}{\covsa}$ yields~\eqref{eq:g:refor}. This establishes the claim for $\covsa\succ 0$.

	In the second part of the proof, we show that the claim remains valid for rank deficient sample covariance matrices. To this end, we denote the optimal value of problem~\eqref{eq:g:refor} by $\g' (\covsa,X)$. From the first part of the proof we know that $\g' (\covsa,X)=\g (\covsa,X)$ for all $\covsa,X\in \PD^d$. We also know that $\g (\covsa,X)$ is continuous in $(\covsa, X)\in\PSD^p\times \PD^p$. It remains to be shown that $\g' (\covsa,X)=\g (\covsa,X)$ for all $\covsa\in\PSD^d$ and $X\in \PD^d$.
	
	Fix any $\covsa\in\PSD^p$ and $X\in\PD^d$, and note that $\covsa+\eps I\succ 0$ for every $\eps>0$. Defining the intervals $\mc E=\R_+$ and $\Gamma=\{\gamma\in\R: \gamma I\succ X\}$ as well as the auxiliary functions 
	\[
		\Phi (\eps)=\g'(\covsa+\eps I,X) \quad \text{and}\quad \varphi (\eps, \gamma) = \dualvar \left(\rho^2 - \Tr{\covsa+\eps I} \right) + \dualvar^2 \inner{(\dualvar I - X)^{-1}}{\covsa+\eps I} ,
	\]
	it follows from~\eqref{eq:g:refor} that 
	\[
		\Phi(\eps) = \inf_{\dualvar \in\Gamma} \varphi(\eps,\dualvar) \quad \forall \eps\in\mc E.
	\]
	One can show via Lemma~\ref{lem:cont} that $\Phi(\eps)$ is continuous at $\eps=0$. Indeed, $\varphi(\eps,\dualvar)$ is linear and thus continuous in $\eps$ for every $\dualvar\in\Gamma$, which implies via Lemma~\ref{lem:cont}(a) that $\Phi(\eps)$ is upper-semicontinuous at $\eps=0$. Moreover, $\varphi (\eps,\dualvar)$ is calm from below at $\eps=0$ with $L=0$ uniformly in $\dualvar \in \Gamma$ because
	\begin{align*}
		\varphi (\eps,\dualvar) - \varphi (0,\dualvar) = \dualvar \Tr{(I - \dualvar^{-1}X)^{-1} - I}\eps \ge 0 \quad \forall \gamma\in\Gamma.
	\end{align*}
	Here, the inequality holds for all $\gamma\in\Gamma$ due to the conditions $I\succ \dualvar^{-1} X\succ 0$, which are equivalent to $0\prec I-\dualvar^{-1} X\prec I$ and imply $(I-\dualvar^{-1} X)^{-1}\succ I$. Lemma~\ref{lem:cont}(b) thus ensures that $\Phi(\eps)$ is lower-semicontinuous at $\eps=0$. In summary, we conclude that~$\Phi(\eps)$ is indeed continuous at $\eps = 0$.	

	Combining the above results, we find 	
	\begin{align*}
	\g (\covsa,X) = \lim_{\eps\ra 0^+} \g (\covsa+\eps I,X)= \lim_{\eps\ra 0^+} \g' (\covsa+\eps I,X)= \lim_{\eps\ra 0^+} \Phi(\eps) = \Phi(0) = \g' (\covsa,X),
	\end{align*}
	where the five equalities hold due to the continuity of $\g(\covsa, X)$ in $\covsa$, the fact that $\g(\covsa, X)=\g'(\covsa, X)$ for all $\covsa\succ 0$, the definition of $\Phi(\eps)$, the continuity of $\Phi(\eps)$ at $\eps=0$ and once again from the definition of $\Phi(\eps)$, respectively. The claim now follows because $\covsa\in\PSD^p$ and $X\in\PD^d$ were chosen arbitrarily.
\end{proof}

We have now collected all necessary ingredients for the proof of Theorem~\ref{thm:refor}.

\begin{proof}[Proof of Theorem~\ref{thm:refor}]
	By Proposition~\ref{prop:g-refor}, the worst-case expectation in~\eqref{eq:DROSimplified} coincides with the optimal value of the semidefinite program~\eqref{eq:g:refor}. Substituting this semidefinite program into~\eqref{eq:DROSimplified} yields~\eqref{eq:Reformulation2}. Note that the condition $X\succ 0$, which ensures that $\log\det X$ is well-defined, is actually redundant because it is implied by the constraint $X\in\mc X$. Nevertheless, we make it explicit in~\eqref{eq:Reformulation2} for the sake of clarity.
		
	It remains to show that $\J(\covsa)$ is continuous. To this end, we first construct bounds on the minimizers of~\eqref{eq:Reformulation2} that vary continuously with $\covsa$. Such bounds can be constructed from any feasible decision $(X_0, \dualvar_0)$. Assume without loss of generality that $\gamma_0 > p/\rho^2$, and denote by $f_0(\covsa)$ the objective value of  $(X_0, \dualvar_0)$ in~\eqref{eq:Reformulation2}, which constitutes a linear function of~$\covsa$. Moreover, define two continuous auxiliary functions
	\begin{align}
	\label{eq:bound:def}
	\overline x(\covsa) \Let  { f_0(\covsa) - p(1-\log \dualvar_0) \over \rho^2 - p\dualvar_0^{-1}} \qquad \text{and} \qquad 
	\underline x(\covsa) \Let \frac{e^{-f_0(\covsa)}}{\overline x(\covsa)^{p-1}} ,
	\end{align}	
which are strictly positive because $\gamma_0 > p/\rho^2$. Clearly, the infimum of problem~\eqref{eq:Reformulation2} is determined only by feasible decisions $(X, \dualvar)$ with an objective value of at most $f_0(\covsa)$. All such decisions satisfy
	\begin{align}
	\label{eq:estimator-ub}
	f_0 (\covsa) & \ge -\log\det X + \dualvar \rho^2 + \dualvar \inner{  (I - \dualvar^{-1} X)^{-1} - I}{\covsa}
	\geq -\log\det X + \dualvar \rho^2 \\
	& \ge -p\log \dualvar + \dualvar \rho^2 \ge (\rho^2 - \dualvar_0^{-1} p) \dualvar + p(1 - \log \dualvar_0), \notag
	\end{align}
	where the second and third inequalites exploit the estimates $(I-\dualvar^{-1} X)^{-1}\succ I$ and $\det X \leq \det (\dualvar I)=\dualvar^p$, respectively, which are both implied by the constraint $\dualvar I \succ X \succ 0$, and the last inequality holds because $\log \dualvar\le \log\dualvar_0+\dualvar_0^{-1}(\dualvar-\dualvar_0)$ for all~$\dualvar>0$. By rearranging the above inequality and recalling the definition of $\overline x(\covsa)$, we thus find $\gamma\le \overline x(\covsa)$, which in turn implies that $X\prec \lambda I\preceq \overline x(\covsa) I$.

	Denoting by $\{x_i\}_{i\le p}$ the eigenvalues of the matrix $X$ and setting $x_{\min}=\min_{i\le p}x_i$, we further find
	\begin{align*}
	f_0(\covsa) \geq - \log \det X = -\log\bigg( \prod_{i=1}^{p} x_{i} \bigg) \geq - \log \big(x_{\min}\, \overline x(\covsa)^{p-1}\big) =  - \log x_{\min} - (p-1) \log \overline x(\covsa),
	\end{align*}
	where the first inequality follows from~\eqref{eq:estimator-ub}, while the second inequality is based on overestimating all but the smallest eigenvalue of $X$ by $\overline x(\covsa)$. By rearranging the above inequality and recalling the definition of $\underline x(\covsa)$, we thus find $x_{\min}\ge \underline x(\covsa)$, which in turn implies that $X\succeq \underline x(\covsa) I$.
	
	The above reasoning shows that the extra constraint $\underline x(\covsa) I\preceq X\preceq \overline x(\covsa) I$ has no impact on~\eqref{eq:Reformulation2}, that~is,
	\begin{align*}
	\J (\covsa) & = \left\{ \begin{array}{cl} \Inf{X} & -\log \det X  + \displaystyle \inf_{\dualvar} \left\{ \dualvar \left( \rho^2 - \Tr{\covsa} \right) + \dualvar^2 \inner{(\dualvar I - X)^{-1}}{\covsa} :  \dualvar I \succ X \right\} \\
	\st & X \in \X,\quad \underline x(\covsa) I\preceq X\preceq \overline x(\covsa) I.
	\end{array}
	\right. \\
	& = \left\{ \begin{array}{cl} \Inf{X} & -\log \det X  + \g(\covsa,X)\\
	\st & X \in \X,\quad \underline x(\covsa) I\preceq X\preceq \overline x(\covsa) I,
	\end{array}
	\right.
	\end{align*}
	where the second equality follows from Proposition~\ref{prop:g-refor}. The continuity of $\J(\covsa)$ now follows directly from Berge's maximum theorem~\cite[pp.~115--116]{ref:Berge-63}, which applies due to the continuity of~$\g(\covsa,X)$ established in Proposition~\ref{prop:g-refor}, the compactness of the feasible set and the continuity of $\underline x(\covsa)$ and $\overline x(\covsa)$.
\end{proof}

An immediate consequence of Theorem~\ref{thm:refor} is that the simplified estimation problem~\eqref{eq:DROSimplified} is equivalent to an explicit semidefinite program and is therefore in principle computationally tractable.

\begin{corollary}[Tractability]
	\label{thm:refor-sdp}
	For any fixed $\rho> 0$ and $\covsa\succeq 0$, the simplified distributionally robust estimation problem~\eqref{eq:DROSimplified} is equivalent to the tractable semidefinite program
	\begin{equation}
		\label{eq:Reformulation1} 
		\J (\covsa) = 
		\left\{
		\begin{array}{cl}
			\Inf{X, Y, \dualvar} &-\log \det X  + \dualvar \left(\rho^2 - \Tr{\covsa} \right)  + \Tr{Y}  \\
			\st & 
			\begin{bmatrix} Y & \dualvar \covsa^{\frac{1}{2}} \\ \dualvar \covsa^{\frac{1}{2}} & \dualvar I - X \end{bmatrix} \succeq 0 \vspace{1mm}\\
			& \dualvar I \succ X \succ 0, \quad Y \succeq 0, \quad X \in \X.
		\end{array}
		\right.
	\end{equation}
\end{corollary}
\begin{proof}
We know from Theorem~\ref{thm:refor} that the estimation problem~\eqref{eq:DROSimplified} is equivalent to the convex program~\eqref{eq:Reformulation2}. As $X$ represents a decision variable instead of a parameter, however, problem~\eqref{eq:Reformulation2} fails to be a semidefinite program per se. Indeed, its objective function involves the nonlinear term $h(X, \dualvar) \Let \dualvar^2 \inner{(\dualvar I - X)^{-1}}{\covsa}$, which is interpreted as $\infty$ outside of its domain $\big\{(X, \dualvar) \in \PSD \times \mathbb{R} : \dualvar I \succ X \big\}$. However, $h(X,\dualvar)$ constitutes a matrix fractional function as described in \cite[Example 3.4]{ref:Boyd-04} and thus admits the semidefinite reformulation			
	\begin{align*}
			h(X,\dualvar) 
			&= \inf_{t}\left\{t ~:~ \dualvar I \succ X, \quad \dualvar^2 \inner{(\dualvar I - X)^{-1}}{\covsa} \leq t \right\} \\
			&= \inf_{Y, t} \left\{t ~:~ \dualvar I \succ X, \quad Y\succeq \dualvar^2 \covsa^{\half} (\dualvar I - X )^{-1} \covsa^{\half}, \quad \Tr{Y} \leq t \right\} \\
			&= \inf_{Y}  \left\{ \Tr{Y} ~:~ \dualvar I \succ X, \quad \begin{bmatrix} Y & \dualvar \covsa^{\half} \\ \dualvar \covsa^{\half} & \dualvar I - X \end{bmatrix} \succeq 0 
			\right\},
		\end{align*}
	where the second equality holds because $A \succeq B$ implies $\Tr{A} \geq \Tr{B}$, while the third equality follows from a standard Schur complement argument; see, {\em e.g.}, \cite[Appendix~A.5.5]{ref:Boyd-04}. Thus, $h(X,\dualvar)$ is representable as the optimal value of a parametric semidefinite program whose objective and constraint functions are jointly convex in the auxiliary decision variable $Y$ and the parameters $X$ and $\dualvar$. The postulated reformulation~\eqref{eq:Reformulation1} is then obtained by substituting the last expression into~\eqref{eq:Reformulation2}.
\end{proof}

Now that we have derived a tractable semidefinite reformulation for the simplified estimation problem~\eqref{eq:DROSimplified}, we are ready to address the generic estimation problem~\eqref{eq:dro}, which does {\em not} assume knowledge of the mean and is robustified against all distributions in the ambiguity set $\Ambi$ {\em without} mean constraints.

\begin{theorem}[Sufficiency of $\covsa$]
	\label{thm:refor2}
	For any fixed $\rho> 0$, $\wh\mu\in\R^p$ and $\covsa\in \PSD^p$, the general distributionally robust estimation problem~\eqref{eq:dro} is equivalent to the optimization problem~\eqref{eq:Reformulation2} and the tractable semidefinite program~\eqref{eq:Reformulation1}. Moreover, the optimal value function~$\J(\wh\mu, \covsa)$ is constant in $\wh \mu$ and continuous in $\covsa$.
\end{theorem}
\begin{proof}
By Proposition~\ref{prop:Wass}, the optimal value of the estimation problem~\eqref{eq:dro} can be expressed as
\begin{align*}
	\J(\wh \mu, \covsa)& =\inf_{\mu,\,X\in\mc X} -\log \det X +  \left\{ \begin{array}{cl} \displaystyle \sup_{\mu',\,S\succeq 0} & (\mu'-\mu)^\top X (\mu'-\mu) + \inner{S}{X} \\
	\st & \Tr{S} + \Tr{\covsa} - 2\Tr{ \sqrt{\covsa^{\frac{1}{2}} S \covsa^{\frac{1}{2}} }}\le \rho^2 -  \norm{\mu' - \wh \mu}^2
	\end{array}\right.\\
	& =\inf_{\mu,\,X\in\mc X} -\log \det X + \sup_{\mu':\|\mu'- \wh \mu\|\le\rho} (\mu'-\mu)^\top X (\mu'-\mu) \\
	& \hspace{4cm}+ \inf_{\dualvar:\dualvar I\succ X} \dualvar \left( \rho^2 -  \norm{\mu' - \wh \mu}^2 - \Tr{\covsa} \right) + \dualvar^2 \inner{(\dualvar I - X)^{-1}}{\covsa} .
\end{align*}
Here, the second equality holds because the Wasserstein constraint is infeasible unless $\|\mu'-\wh \mu\|\le \rho$ and because the maximization problem over $S$, which constitutes an instance of \eqref{eq:g:induced-metric} with $\rho^2-\|\mu'-\wh \mu\|^2$ instead of~$\rho^2$, can be reformulated as a minimization problem over $\dualvar$ thanks to Proposition~\ref{prop:g-refor}. By the minimax theorem \cite[Proposition~5.5.4]{ref:Bert-09}, which applies because $\mu'$ ranges over a compact ball and because $X-\dualvar I\prec 0$, we may then interchange the maximization over $\mu'$ with the minimization over $\dualvar$ to obtain
\begin{align*}
	\J(\wh \mu, \covsa)& = \inf_{\scriptsize \begin{array}{c }\mu,\,X\in\mc X,\\[-0.5ex] \gamma:\gamma I\succ X\end{array}} -\log \det X + \sup_{\mu':\|\mu'-\wh \mu\|\le\rho} (\mu'-\mu)^\top X (\mu'-\mu) \\[-2ex]
	& \hspace{4.8cm} + \dualvar \left( \rho^2 - \norm{\mu' - \wh \mu}^2 - \Tr{\covsa} \right) + \dualvar^2 \inner{(\dualvar I - X)^{-1}}{\covsa}.
\end{align*}
Using the minimax theorem \cite[Proposition~5.5.4]{ref:Bert-09} once again to interchange the minimization over $\mu$ with the maximization over $\mu'$ yields
\begin{align*}
	\J(\wh \mu, \covsa)& = \inf_{\scriptsize \begin{array}{c } X\in\mc X,\\[-0.5ex] \gamma:\gamma I\succ X\end{array}} -\log \det X + \sup_{\mu':\|\mu'-\wh \mu\|\le\rho} \inf_{\mu}\;(\mu'-\mu)^\top X (\mu'-\mu) \\[-2ex]
	& \hspace{4.8cm} + \dualvar \left( \rho^2 - \norm{\mu' - \wh \mu}^2 - \Tr{\covsa} \right) + \dualvar^2 \inner{(\dualvar I - X)^{-1}}{\covsa}\\
	& = \inf_{\scriptsize \begin{array}{c } X\in\mc X,\\[-0.5ex] \gamma:\gamma I\succ X\end{array}} -\log \det X + \dualvar \left( \rho^2 - \Tr{\covsa} \right) + \dualvar^2 \inner{(\dualvar I - X)^{-1}}{\covsa},
\end{align*}
where the second equality holds because $\mu'$ is the unique optimal solution of the innermost minimization problem over $\mu$, while $\wh \mu$ is the unique optimal solution of the maximization problem over $\mu'$. Thus, the general estimation problem \eqref{eq:dro} is equivalent to~\eqref{eq:Reformulation2}, and $\J(\wh \mu, \covsa)$ is manifestly constant in $\wh \mu$. Theorem~\ref{thm:refor} further implies that $\J(\wh \mu, \covsa)$ is continuous in $\covsa$, while Corollary~\ref{thm:refor-sdp} implies that \eqref{eq:dro} is equivalent to the tractable semidefinite program~\eqref{eq:Reformulation1}. These observations complete the proof.
\end{proof}

Theorem~\ref{thm:refor2} asserts that the general estimation problem \eqref{eq:dro} is equivalent to the simplified estimation problem~\eqref{eq:DROSimplified}, which is based on the hypothesis that the mean of $\xi$ is known to vanish. Theorem~\ref{thm:refor2} further reveals that the general estimation problem \eqref{eq:dro} as well as its (unique) optimal solution depend on the training data only through the sample covariance matrix $\covsa$. This is reassuring because $\covsa$ is known to be a sufficient statistic for the precision matrix. As solving~\eqref{eq:dro} is tantamount to solving~\eqref{eq:DROSimplified}, it suffices to devise solution procedures for the simplified estimation problem~\eqref{eq:DROSimplified} or its equivalent reformulations~\eqref{eq:Reformulation2} and~\eqref{eq:Reformulation1}.


We emphasize that the strictly convex log-determinant term in the objective of~\eqref{eq:Reformulation1} is supported by state-of-the-art interior point solvers for semidefinite programs such as SDPT3~\cite{ref:Tutuncu-03}. In principle, problem~\eqref{eq:Reformulation1} can therefore be implemented directly in MATLAB using the YALMIP interface~\cite{ref:Lofberg-04}, for instance. In spite of its theoretical tractability, however, the semidefinite program~\eqref{eq:Reformulation1} quickly becomes excruciatingly large, and direct solution with a general purpose solver becomes impracticable already for moderate values of~$p$. This motivates us to investigate practically relevant special cases in which the estimation problem~\eqref{eq:DROSimplified} can be solved either analytically (Section~\ref{sec:analytical}) or numerically using a dedicated fast Newton-type algorithm (Section~\ref{sec:numerical}).

\section{Analytical Solution without Sparsity Information}
\label{sec:analytical}
If we have no prior information about the precision matrix, it is natural to set $\X = \PD^p$. In this case, the distributionally robust estimation problem~\eqref{eq:DROSimplified} can be solved in quasi-closed form. 

\begin{theorem}[Analytical solution without sparsity information]
\label{thm:main:theorem}
If $\rho > 0$, $\X = \PD^p$ and $\covsa\in\PSD^p$ admits the spectral decomposition $\covsa = \sum_{i=1}^p \covsaeig_i \covsaeigvect_i \covsaeigvect_i^\top$ with eigenvalues $\covsaeig_i$ and corresponding orthonormal eigenvectors~$\covsaeigvect_i$, $i\le p$, then the unique minimizer of~\eqref{eq:DROSimplified} is given by $\est = \sum_{i=1}^p \estx_i \covsaeigvect_i \covsaeigvect_i^\top$, where 
\begin{subequations}
\label{eq:analytical}
	\be
	\label{eq:x:value}
	\estx_i = \estdual \left[ 1 - \frac{1}{2} \left( \sqrt{\covsaeig_i^2 (\estdual)^2  + 4 \covsaeig_i \estdual } - \covsaeig_i \estdual \right) \right] \qquad \forall i \le p
	\ee
	and $\estdual>0$ is the unique positive solution of the algebraic equation
	\be
	\label{eq:gamma:value}
	\bigg( \rho^2 - \frac{1}{2} \sum_{i=1}^p \covsaeig_i \bigg)  \dualvar - p + \frac{1}{2}  \sum_{i=1}^p   \sqrt{\covsaeig_i^2 \dualvar^2  + 4 \covsaeig_i \dualvar } = 0. 
	\ee
\end{subequations}
\end{theorem}
\begin{proof}
	We first demonstrate that the algebraic equation~\eqref{eq:gamma:value} admits a unique solution in $\R_+$. For ease of exposition, we define $\varphi(\dualvar)$ as the left-hand side of~\eqref{eq:gamma:value}. It is easy to see that $\varphi (0) = -p < 0$ and $ \lim_{\dualvar \to \infty} \varphi(\dualvar)/\dualvar = \rho^2$, which implies that $\varphi(\dualvar)$ grows asymptotically linearly with $\dualvar$ at slope $\rho^2>0$. By the intermediate value theorem, we may thus conclude that the equation~\eqref{eq:gamma:value} has a solution~$\estdual>0$.

	As $\covsaeig_i \dualvar + 2 > \sqrt{ \covsaeig_i^2 \dualvar^2 + 4 \covsaeig_i \dualvar}$, the derivative of $\varphi(\dualvar)$ satisfies
	\[
	\frac{\dd}{\dd\!\dualvar} \varphi (\dualvar) = \rho^2 + \frac{1}{2} \sum_{i=1}^p \covsaeig_i \left( \frac{\covsaeig_i \dualvar  + 2}{\sqrt{\covsaeig_i^2 \dualvar^2  + 4 \covsaeig_i \dualvar }} -1 \right) > 0\,,
	\]
	whereby $\varphi (\dualvar)$ is strictly increasing in $\dualvar\in\R_+$. Thus, the solution $\estdual$ is unique. The positive slope of~$\varphi(\dualvar)$ further implies via the implicit function theorem that $\estdual$ changes continuously with $\lambda_i\in\R_+$,~$i\le p$.
	
	In analogy to Proposition~\ref{prop:g-refor}, we prove the claim first under the assumption that $\covsa \succ 0$ and postpone the generalization to rank deficient sample covariance matrices. 
	Focussing on $\covsa\succ 0$, we will show that $(\est, \estdual)$ is feasible and optimal in~\eqref{eq:Reformulation2}. By Theorem~\ref{thm:refor}, this will imply that $\est$ is feasible and optimal in~\eqref{eq:DROSimplified}. 
	
	As $\estdual > 0$ and $\covsa\succ 0$, which means that $\covsaeig_i > 0$ for all $i\le p$, an elementary calculation shows that
	\[
	2>  \sqrt{\covsaeig_i^2 (\estdual)^2 + 4 \covsaeig_i \estdual} - \covsaeig_i \estdual>0 \quad \iff \quad 1>  1-\frac{1}{2}\left(\sqrt{\covsaeig_i^2 (\estdual)^2 + 4 \covsaeig_i \estdual} - \covsaeig_i \estdual\right)>0.
	\]
	Multiplying the last inequality by $\estdual$ proves that $\estdual>  \estx_i >0$ for all $i\le p$, which in turn implies that $\estdual I\succ \est \succ 0$. Thus, $(\est,\estdual )$ is feasible in~\eqref{eq:Reformulation2}, and $\est$ is feasible in~\eqref{eq:DROSimplified}.

	To prove optimality, we denote by $f(X,\dualvar)$ the objective function of problem~\eqref{eq:Reformulation2} and note that its gradient with respect to $X$ vanishes at $(\est, \estdual)$. Indeed, we have
	\be
		\notag
		\begin{aligned}
		\nabla_X f(\est, \estdual) &= -(\est)^{-1} + (\estdual)^2 (\estdual I - \est)^{-1} \covsa (\estdual I - \est)^{-1} \\
		&=  \ds \sum_{i=1}^p \left((\estdual)^2 (\estdual - \estx_i)^{-2} \covsaeig_i - (\estx_i)^{-1} \right) v_i v_i^\top \\
	 	&= \ds \sum_{i=1}^p \frac{(\estdual)^2 \estx_i \covsaeig_i - (\estdual - \estx_i)^2}{(\estdual - \estx_i)^2 x_i} v_i v_i^\top = 0,
	 \end{aligned}
	 \ee
	where the first equality exploits the basic rules of matrix calculus (see, {\em e.g.}, \cite[p.~631]{ref:Bernstein-2009}), the second equality holds because $\covsa$ and $X$ share the same eigenvectors $v_i$, $i\le p$, and the last equation follows from the identity
	\be
	\label{eq:analytical:nice:equation:for:x} 
	(\estdual)^2 \estx_i \covsaeig_i = (\estdual - \estx_i)^2 \quad \forall i\le p,
	\ee	
	which is a direct consequence of the definitions of $\estdual$ and $\estx_i$, $i\le p$, in~\eqref{eq:analytical}. Similarly, the partial derivative of $f(X,\dualvar)$ with respect to $\dualvar$ vanishes at $(\est, \estdual)$, too. In fact, we have
	\be
		\begin{aligned}
			\frac{\partial}{\partial\dualvar} f(\est, \estdual) &= \rho^2 - \Tr{\covsa} + 2 \estdual \Tr {(\estdual I - \est)^{-1} \covsa}  - (\estdual)^2 \Tr {(\estdual I - \est)^{-1} \covsa (\estdual I - \est)^{-1}} \notag \\
			&=  \rho^2 - \ds \sum_{i=1}^p \covsaeig_i \left(1 - \frac{2 \estdual }{\estdual - \estx_i} + \frac{(\estdual)^2 }{(\estdual - \estx_i)^2} \right) = \rho^2 - \ds \sum_{i=1}^p \frac{(\estx_i)^2}{(\estdual - \estx_i)^2} \covsaeig_i  \notag \\
			&= \frac{1}{(\estdual)^2} \left(\rho^2 (\estdual)^2 - \sum_{i = 1}^p \estx_i \right) = 0,
		\end{aligned}
	\ee
	where the second equality expresses $\covsa$ and $X$ in terms of their respective spectral decompositions, the fourth equality holds due to~\eqref{eq:analytical:nice:equation:for:x}, and the last equality follows from the observation that $\rho^2 (\estdual)^2 = \sum_{i=1}^p \estx_i$. In summary, we have shown that $(\est, \estdual)$ satisfies the first-order optimality conditions of the convex optimization problem~\eqref{eq:Reformulation2}, which ensures that $\est$ is optimal in~\eqref{eq:DROSimplified}.
	
	Consider now any (possibly singular) sample covariance matrix $\covsa\in\PSD^p$. As $\estdual>0$, similar arguments as in the first part of the proof show that $\estdual\ge \estx_i>0$ for all $i\le p$, which in turn implies that $\estdual I\succeq \est\succ 0$. Moreover, if $\covsa$ has at least one zero eigenvalue, it is easy to see that $\estdual I\not \succ \est$, in which case $(\est,\estdual)$ fails to be feasible in~\eqref{eq:Reformulation2}. However, $\est$ remains feasible and optimal in~\eqref{eq:DROSimplified}. To see this, consider the invertible sample covariance matrix $\covsa+\eps I\succ 0$ for some $\eps>0$, and denote by $(\est(\eps),\estdual(\eps))$ the corresponding minimizer of problem~\eqref{eq:Reformulation2} as constructed in~\eqref{eq:analytical}. As the solution of the algebraic equation~\eqref{eq:gamma:value} depends continuously on the eigenvalues of the sample covariance matrix, we conclude that the auxiliary variable $\estdual(\eps)$ and---by virtue of~\eqref{eq:x:value}---the estimator $\est(\eps)$ are both continuous in $\eps\in\R_+$. Thus, we find
	\begin{align*}
	\J(\covsa)= \lim_{\eps\ra0^+} \J(\covsa+\eps I) = \lim_{\eps\ra0^+} -\log \det \est(\eps)  + \g(\covsa+\eps I,\est(\eps)) = - \log\det \est + \g(\covsa,\est),
	\end{align*}
	where the first equality follows from the continuity of $\J(\covsa)$ established in Theorem~\ref{thm:refor}, the second equality holds because $\est(\eps)$ is the optimal estimator corresponding to the sample covariance matrix $\covsa+\eps I\succ 0$ in problem~\eqref{eq:DROSimplified}, and the third equality follows from the continuity of $ \g(\covsa,X)$ established in Proposition~\ref{prop:g-refor} and the fact that $\lim_{\eps\ra0^+} \est(\eps)=\est\succ 0$. Thus, $\est$ is indeed optimal in~\eqref{eq:DROSimplified}. The strict convexity of $-\log\det X$ further implies that $\est$ is unique. This observation completes the proof. 
\end{proof}


\begin{remark}[Properties of $\est$] 
The optimal distributionally robust estimator $\est$ identified in Theorem~\ref{thm:main:theorem} commutes with the sample covariance matrix $\covsa$ because both matrices share the same eigenbasis. Moreover, the eigenvalues of $\est$ are obtained from those of $\covsa$ via a nonlinear transformation that depends on the size~$\rho$ of the ambiguity set. We emphasize that all eigenvalues of $\est$ are positive for every~$\rho>0$, which implies that $\est$ is invertible. These insights suggest that $\est$ constitutes a nonlinear shrinkage estimator, which enjoys the rotation equivariance property (when all data points are rotated by $R\in \R^{p\times p}$, then $\est$ changes to $R\est R^\top$). 
\end{remark}


Theorem~\ref{thm:main:theorem} characterizes the optimal solution of problem~\eqref{eq:DROSimplified} in quasi-closed form up to the spectral decomposition of $\covsa$ and the numerical solution of equation~\eqref{eq:gamma:value}. By~\cite[Theorem~1.1]{ref:Pan-1999}, the eigenvalues of~$\covsa$ can be computed to within an absolute error~$\eps$ in $\mc O(p^3)$ arithmetic operations. Moreover, as its left-hand side is increasing in~$\estdual$, equation~\eqref{eq:gamma:value} can be solved reliably via bisection or by the Newton-Raphson method. The following lemma provides a priori bounds on $\estdual$ that can be used to initialize the bisection interval.

\begin{lemma}[Bisection interval]
	\label{lemma:BoundForLambda}
	For $\rho>0$, the unique solution of~\eqref{eq:gamma:value} satisfies $\estdual \in [\dualvar_{\min}, \dualvar_{\max}]$, where
	\be
	\label{eq:LowerUpperDefinition}
	\dualvar_{\min} = \frac{p^2 \covsaeig_{\max} + 2 p\rho^2 - p \sqrt{p^2 \covsaeig_{\max}^2 + 4 p \rho^2 \covsaeig_{\max}}}{2\rho^4} > 0, \qquad 
	\dualvar_{\max} = \min\left\{\frac{p}{\rho^2}, \frac{1}{\rho}\sqrt{\sum_{i=1}^p \frac{1}{\covsaeig_i}} \right\},
	\ee
	and $\covsaeig_{\max}$ denotes the maximum eigenvalue of $\covsa$. 
\end{lemma}
\begin{proof}
	By the definitions of $\estdual$ and $\estx_i$ in~\eqref{eq:analytical} we have $\covsaeig_i \estx_i = (\estdual-\estx_i)^2/(\estdual)^2 <1$, which implies that $\estx_i \leq \frac{1}{\covsaeig_i}$. Using~\eqref{eq:analytical} one can further show that $(\estdual)^2 = \frac{1}{\rho^2}\sum_{i=1}^p \estx_i \leq \frac{1}{\rho^2}\sum_{i=1}^p \frac{1}{\covsaeig_i}$, which is equivalent to $\estdual\le \frac{1}{\rho} (\sum_{i=1}^p \frac{1}{\covsaeig_i})^{\frac{1}{2}}$. Note that this upper bound on $\estdual$ is finite only if $\covsaeig_i>0$ for all $i\le p$. To derive an upper bound that is universally meaningful, we denote the left-hand side of~\eqref{eq:gamma:value} by $\varphi(\dualvar)$ and note that $\rho^2 \dualvar - p \leq \varphi(\dualvar)$ for all $\dualvar\ge 0$. This estimate implies that $\estdual \leq \frac{p}{\rho^2}$. Thus, we find $\estdual\le\min\{\frac{p}{\rho^2}, \frac{1}{\rho} (\sum_{i=1}^p \frac{1}{\covsaeig_i})^{\frac{1}{2}}\} = \dualvar_{\max}$. 
	
	To derive a lower bound on $\estdual$, we set $\covsaeig_{\max} = \max_{i\le p}\covsaeig_i$ and observe that
	\[
	\varphi(\dualvar) \leq \rho^2 \dualvar - p + \sum_{i=1}^p \sqrt{\covsaeig_i \dualvar} \leq \rho^2 \dualvar - p + p \sqrt{\covsaeig_{\max} \dualvar}\,,
	\]
	where the first inequality holds because $\sqrt{a+b} \leq \sqrt{a} + \sqrt{b}$ for all $a,b\ge 0$. As the unique positive zero of the right-hand side, $\dualvar_{\min}$ provides a nontrivial lower bound on $\estdual$. Thus, the claim follows.
\end{proof}

Lemma~\ref{lemma:BoundForLambda} implies that $\estdual$ can be computed via the standard bisection algorithm to within an absolute error of~$\eps$ in $\log_2((\dualvar_{\max}-\dualvar_{\min})/\eps)=\mc O(\log_2 p)$ iterations. As evaluating the left-hand side of~\eqref{eq:gamma:value} requires only $\mc O(p)$ arithmetic operations, the computational effort for constructing $\est$ is largely dominated by the cost of the spectral decomposition of the sample covariance matrix.


\begin{remark}[Numerical stability]
	If both $\estdual$ and $\covsaeig_i$ are large numbers, then formula~\eqref{eq:x:value} for $\estx_i$ becomes numerically instable. A mathematically equivalent but numerically more robust reformulation of~\eqref{eq:x:value} is
	\[
	\estx_i = \estdual \left( 1 - \frac{2}{1+\sqrt{1+\frac{4}{\covsaeig_i \estdual}}} \right) \,.
	\]
\end{remark}

In the following we investigate the impact of the Wasserstein radius $\rho$ 
on the optimal Lagrange multiplier~$\estdual$ and the corresponding optimal estimator $\est$. 

\begin{proposition}[Sensitivity analysis]
	\label{prop:sens}
	Assume that the eigenvalues of $\covsa$ are sorted in ascending order, that is, $\covsaeig_1\le \cdots \le \covsaeig_p$. If $\estdual(\rho)$ denotes the solution of \eqref{eq:gamma:value}, and $\estx_i(\rho)$, $i\le p$, represent the eigenvalues of $\est$ defined in \eqref{eq:x:value}, which makes the dependence on $\rho>0$ explicit, then the following assertions hold:
	\begin{itemize}
	\item[(i)] $\estdual(\rho)$ decreases with $\rho$, and $\lim_{\rho \to \infty} \estdual(\rho) = 0$;
	\item[(ii)] $\estx_i(\rho)$ decreases with $\rho$, and $ \lim_{\rho \to \infty} \estx_i(\rho) = 0$ for all $i\le p$;
	\item[(iii)] the eigenvalues of $\est$ are sorted in descending order, that is, $\estx_1(\rho)\ge \cdots \ge \estx_p(\rho)$ for every $\rho>0$;
	\item[(vi)] the condition number $\estx_1(\rho)/\estx_p(\rho)$ of $\est$ decreases with $\rho$, and $\lim_{\rho \to \infty} \estx_1(\rho)/\estx_p(\rho)=1$.
	\end{itemize}
\end{proposition}

\begin{proof}
	As the left-hand side of~\eqref{eq:gamma:value} is strictly increasing in $\rho$, it is clear that $\estdual(\rho)$ decreases with $\rho$. Moreover, the a priori bounds on $\estdual(\rho)$ derived in Lemma~\ref{lemma:BoundForLambda} imply that
	\[
		0\le \lim_{\rho \to \infty} \estdual(\rho) \le \lim_{\rho \to \infty} \frac{p}{\rho^2}=0.
	\]
	Thus, assertion~(i) follows. Next, by the definition of the eigenvalue~$\estx_i$ in~\eqref{eq:x:value}, we have
	\[
	\frac{\partial \estx_i}{\partial \estdual} 
	= 1 + \covsaeig_i \estdual - \frac{1}{2} \left( \sqrt{\covsaeig_i^2 (\estdual)^2 + 4 \covsaeig_i \estdual} + \frac{\covsaeig_i^2 (\estdual)^2 + 2 \covsaeig_i \estdual}{\sqrt{\covsaeig_i^2 (\estdual)^2 + 4 \covsaeig_i \estdual}}\right)
	= 1 +  \covsaeig_i \estdual - \frac{\covsaeig_i^2 (\estdual)^2 + 3 \covsaeig_i \estdual}{\sqrt{\covsaeig_i^2 (\estdual)^2 + 4 \covsaeig_i \estdual}}.
	\]
	Elementary algebra indicates that $(1+z)\sqrt{z^2 + 4z} \geq z^2 + 3z$ for all $z\ge 0$, whereby the right-hand side of the above expression is strictly positive for every $\covsaeig_i\ge 0$ and $\estdual\ge 0$. We conclude that $\estx_i$ grows with $\estdual$ and, by the monotonicity of $\estdual(\rho)$ established in assertion~(i), that $\estx_i(\rho)$ decreases with $\rho$. As $\estdual(\rho)$ drops to $0$ for large $\rho$ and as the continuous function~\eqref{eq:x:value} evaluates to $0$ at $\estdual=0$, we thus find that $\estx_i(\rho)$ converges to $0$ as $\rho$ grows. These observations establish assertion~(ii). As for assertion~(iii), use~\eqref{eq:x:value} to express the $i$-th eigenvalue of $\est$ as $\estx_i=1-\frac{1}{2}\psi(\covsaeig_i)$, where the auxiliary function $\psi (\covsaeig) = \sqrt{\covsaeig^2 (\estdual)^2 + 4 \covsaeig \estdual} - \covsaeig \estdual$ is defined for all $\covsaeig\ge 0$. Note that $\psi(\covsaeig)$ is monotonically increasing because
	\[
	\frac{\dd}{\dd\!\covsaeig} \psi(\covsaeig) = \frac{\covsaeig (\estdual)^2 + 2 \estdual}{\sqrt{\covsaeig^2 (\estdual)^2 + 4 \covsaeig \estdual}} - \estdual
	= \estdual \left( \frac{\covsaeig \estdual + 2}{\sqrt{\covsaeig^2 (\estdual)^2 + 4 \covsaeig \estdual}} - 1 \right) > 0\,.
	\]
	As $\covsaeig_{i+1}\ge \covsaeig_i$ for all $i<p$, we thus have $\psi(\covsaeig_{i+1})\ge \psi(\covsaeig_i)$, which in turn implies that $\estx_{i+1}\le\estx_j$. Hence, assertion~(iii) follows. As for assertion~(iv), note that by~\eqref{eq:x:value} the condition number of $\est$ is given by
	\[
	{\estx_1(\rho) \over \estx_p(\rho)} = 
	\frac{1 - \frac{1}{2} \left( \sqrt{\covsaeig_1^2 \estdual(\rho)^2 + 4 \covsaeig_1 \estdual(\rho)} - \covsaeig_1 \estdual(\rho) \right)}{1 - \frac{1}{2} \left( \sqrt{\covsaeig_p^2 \estdual(\rho)^2 + 4 \covsaeig_p \estdual(\rho)} - \covsaeig_p \estdual(\rho) \right)}.
	\]
	The last expression converges to $1$ as $\rho$ tends to infinity because $\estdual(\rho)$ vanishes asymptotically due to assertion~(i). A tedious but straightforward calculation using~\eqref{eq:x:value} shows that $\frac{\partial}{\partial \estdual} \log (\estx_1/\estx_p)>0$, which implies via the monotonicity of the logarithm that $\estx_1/\estx_p$ increases with $\estdual$. As $\estdual(\rho)$ decreases with $\rho$ by virtue of assertion~(i), we may then conclude that the condition number $\estx_1(\rho)/\estx_p(\rho)$ decreases with~$\rho$. 
	\end{proof}

%

Figure~\ref{figure:eig_gamma_rho} visualizes the dependence of $\estdual$ and $\est$ on the Wasserstein radius $\rho$ in an example where $p=5$ and the eigenvalues of $\covsa$ are given by $\covsaeig_i=10^{i-3}$ for $i \le 5$. Figure~\ref{figure:eig_gamma_rho:1} displays $\estdual$ as well as its a priori bounds $\dualvar_{\min}$ and $\dualvar_{\max}$ derived in Lemma~\ref{lemma:BoundForLambda}. Note first that $\estdual$ drops monotonically to $0$ for large $\rho$, which is in line with Proposition~\ref{prop:sens}(i). As $\estdual$ represents the Lagrange multiplier of the Wasserstein constraint, which limits the size of the ambiguity set to $\rho$, this observation indicates that the worst-case expectation~\eqref{eq:g:def} displays a decreasing marginal increase in $\rho$. Figure~\ref{figure:eig_gamma_rho:2} visualizes the eigenvalues $\estx_i$, $i\le 5$, as well as the condition number of~$\est$. Note that all eigenvalues are monotonically shrunk towards $0$ and that their order is preserved as $\rho$ grows, which provides empirical support for Propositions~\ref{prop:sens}(ii) and \ref{prop:sens}(iii), while the condition number decreases monotonically to~$1$, which corroborates Proposition~\ref{prop:sens}(iv). 

In summary, we have shown that $\est$ constitutes a nonlinear shrinkage estimator that is rotation equivariant, positive definite and well-conditioned. Moreover, $(\est)^{-1}$ preserves the order of the eigenvalues of~$\covsa$. We emphasize that neither the interpretation of $\est$ as a shrinkage estimator nor any of its desirable properties---most notably the improvement of its condition number with $\rho$---were dictated {\em ex ante}. Instead, these properties arose naturally from an intuitively appealing distributionally robust estimation scheme. In contrast, existing estimation schemes sometimes impose {\em ad hoc} constraints on condition numbers; see,~{\em e.g.},~\cite{ref:Won-2013}. On the downside, as $\est$ shares the same eigenbasis as the sample covariance matrix $\covsa$, it does not prompt a new robust principal component analysis. We henceforth refer to $\est$ as the {\em Wasserstein shrinkage estimator}.


\begin{figure*} [t]
	\centering
	\subfigure[Lagrange multiplier $\estdual$ and its a priori bounds $\dualvar_{\min}$ and $\dualvar_{\max}$ from Lemma~\ref{lemma:BoundForLambda}.]{\label{figure:eig_gamma_rho:1}
		\includegraphics[width=0.4\columnwidth]{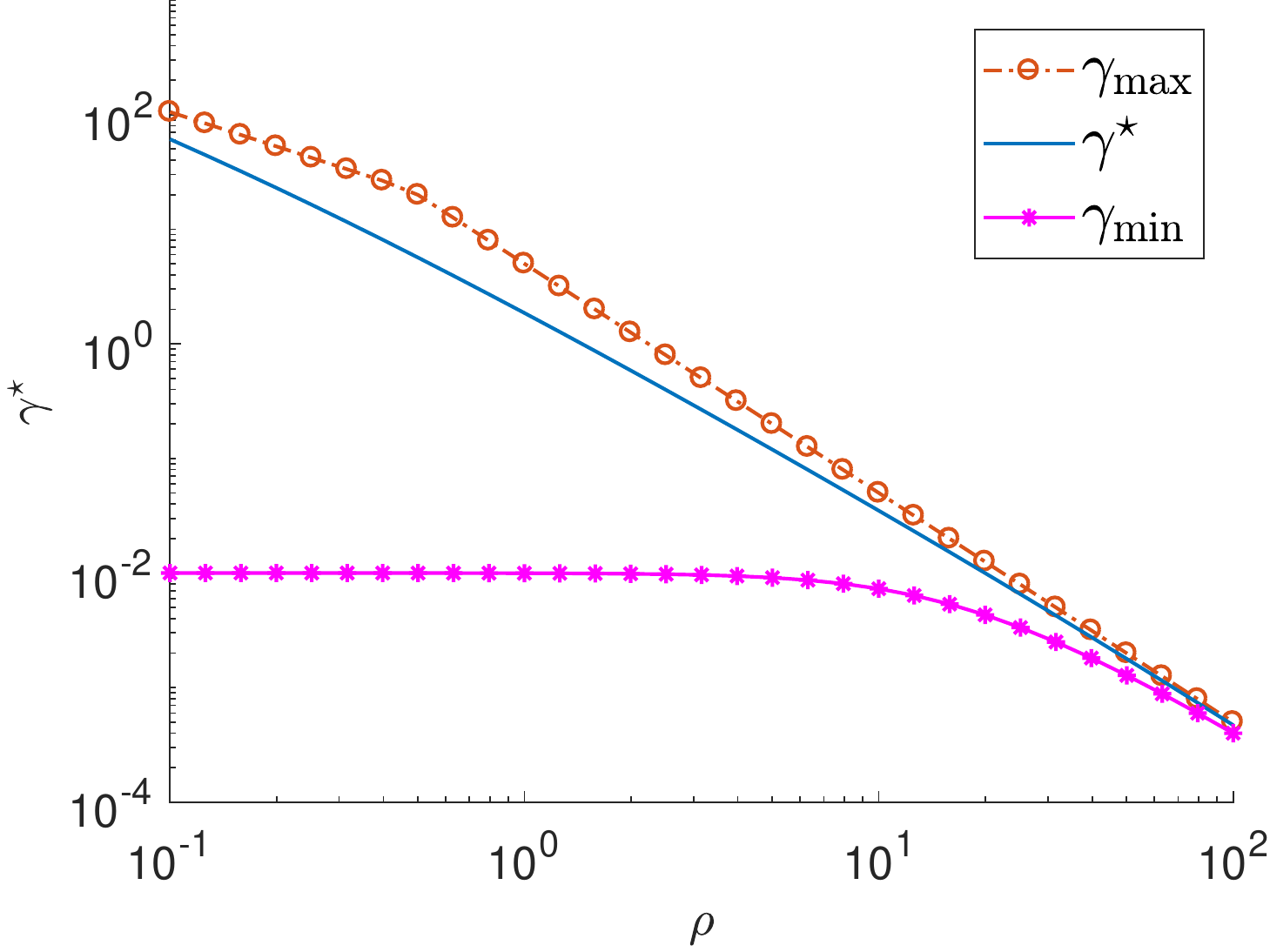}} \hspace{1mm}
	\subfigure[Eigenvalues (left axis) and condition number (round marker - right axis) of $\est$.]{\label{figure:eig_gamma_rho:2} 
		\includegraphics[width=0.4\columnwidth]{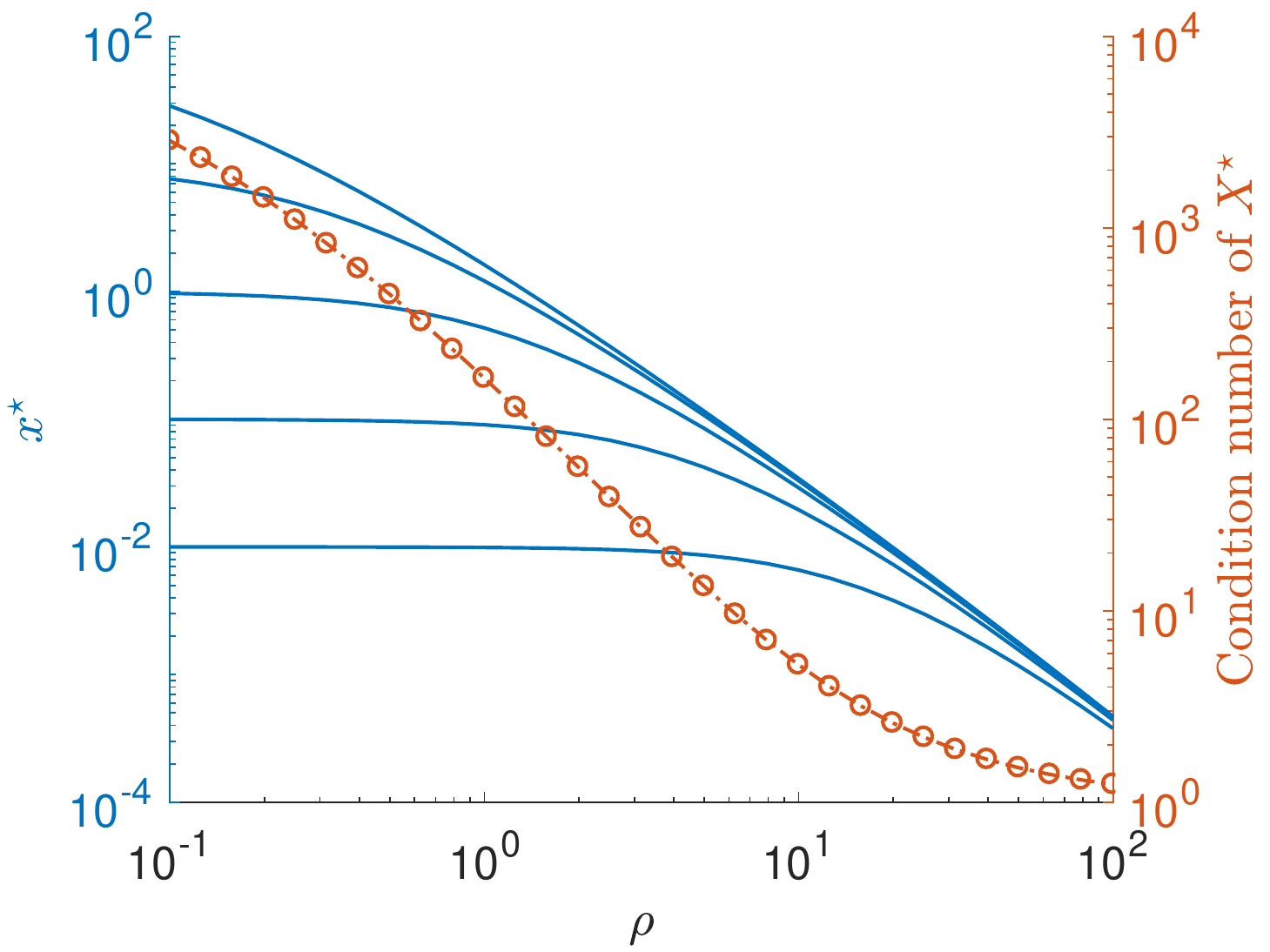}} \hspace{1mm}
	\caption{Dependence of the Lagrange multiplier $\estdual$ (left panel) as well as the eigenvalues $\estx_i$, $i\le 5$, and the condition number $\estx_5/\estx_1$ of the optimal estimator $\est$ (right panel) on $\rho$. 
	}
	\label{figure:eig_gamma_rho}
\end{figure*}

%

\section{Numerical Solution with Sparsity Information}
\label{sec:numerical}
We now investigate a more general setting where $\mc X$ may be a strict subset of $\PD^p$, which captures a prescribed conditional independence structure of $\xi$. Specifically, we assume that there exists $\mc E\subseteq \{1,\ldots,p\}^2$ such that the random variables $\xi_i$ and $\xi_j$ are conditionally independent given $\xi_{-\{i,j\}}$ for any pair $(i,j)\in\mc E$, where $\xi_{-\{i,j\}}$ represents the truncation of the random vector $\xi$ without the components~$\xi_i$ and~$\xi_j$. It is well known that if~$\xi$ follows a normal distribution with covariance matrix $S\succ 0$ and precision matrix $X=S^{-1}$, then $\xi_i$ and $\xi_j$ are conditionally independent given $\xi_{-(i,j)}$ if and only if $X_{ij}=0$. This reasoning forms the basis of the celebrated Gaussian graphical models, see, {\em e.g.}, \cite{ref:Lauritzen-1996}. Any prescribed conditional independence structure of $\xi$ can thus conveniently be captured by the feasible set
\[
	\X = \{ X \in \PD^p : X_{ij} = 0 \quad \forall (i, j) \in \mc E \}.
\]
We may assume without loss of generality that $\mc E$ inherits symmetry from $X$, that is, $(i,j)\in\mc E \implies (j,i)\in\mc E$. In Section~\ref{sec:analytical} we have seen that the robust maximum likelihood estimation problem~\eqref{eq:DROSimplified} admits an analytical solution when $\mc E=\emptyset$. In the general case, analytical tractability is lost. Indeed, if $\mc E\neq \emptyset$, then even the nominal estimation problem obtained by setting $\rho=0$ requires numerical solution~\cite{dahletal2005}. In this section we develop a Newton-type algorithm to solve~\eqref{eq:DROSimplified} in the presence of prior conditional independence information. For the sake of consistency, we will refer to the optimal solution of problem~\eqref{eq:DROSimplified} as the {\em Wasserstein shrinkage estimator} even in the presence of sparsity constraints.

\begin{remark}[Conditional independence information in $\Ambi$]
	We emphasize that our proposed estimation model accounts for the prescribed conditional independence structure only in the feasible set $\X$ but not in the ambiguity set $\Ambi$. Otherwise, the ambiguity set would have to be redefined as
	\[
		\Ambi = \left\{ \Q \in \mc N^p_0~:~\Wass(\Q, \wh{\mbb P}) \leq \rho , ~ (\EE^{\Q} [\xi \xi^{\top}]^{-1})_{ij} = 0 \quad \forall (i,j) \in \mc E\right\} .
	\]
	While conceptually attractive, this new ambiguity set is empty even for some $\rho > 0$ because the inverse sample covariance matrix $\covsa^{-1}$ violates the prescribed conditional independence relationships with probability~1. 
\end{remark}

Recall from Theorem~\ref{thm:refor} that the estimation problem~\eqref{eq:DROSimplified} is equivalent to the convex program~\eqref{eq:Reformulation2} and that the optimal value of~\eqref{eq:Reformulation2} depends continuously on $\covsa\in\PSD^p$. In the remainder of this section we may thus assume without much loss of generality that $\covsa\succ 0$. Otherwise, we can replace $\covsa$ with $\covsa+\eps I$ for some small~$\eps>0$ without significantly changing the estimation problem's solution. Inspired by~\cite{ref:Olsen-2012, ref:Hsieh-2014}, we now develop a sequential quadratic approximation algorithm for solving problem~\eqref{eq:Reformulation2} with sparsity information. Note that the set~$\mc X$ of feasible precision matrices typically fixes many entries to zero, thus reducing the effective problem dimension and making a second-order algorithm attractive even for large instances of~\eqref{eq:Reformulation2}.

The proposed algorithm starts at $X_0=I$ and at some $\dualvar_0>1$, which are trivially feasible in~\eqref{eq:Reformulation2}. In each iteration the algorithm moves from the current iterate $(X_t, \dualvar_t)$ along a feasible descent direction, which is constructed from a quadratic approximation of the objective function of problem~\eqref{eq:Reformulation2}. A judiciously chosen step size guarantees that the next iterate $(X_{t+1}, \dualvar_{t+1})$ remains feasible and has a better (lower) objective value; see Algorithm~\ref{algo:QuadraticApprox}. The construction of the descent direction relies on the following lemma.




\begin{lemma}[Fact~7.4.8 in \cite{ref:Bernstein-2009}]
\label{lem:kronecker}
For any $A,B\in\R^{p\times p}$ and $X\in\mathbb S^p$, we have \[\Tr{AXBX}=\vect(X)^\top (B\otimes A^\top) \vect(X).\]\end{lemma}

\begin{proposition}[Descent direction]
	\label{prop:DXDl:optimal}
	Fix $(X,\dualvar)\in \PD^p\times \R_{++}$ with $\dualvar I\succ X$, and define the orthogonal projection $P:\R^{p^2+1}\rightarrow \R^{p^2+1}$ through $(Pz)_k=0$ if $k={j(p-1)+i}$ for some $(i,j)\in\mc E$; $=\frac{1}{2}z_{j(p-1)+i}+\frac{1}{2}z_{i(p-1)+j}$ if $k={j(p-1)+i}$ for some $i,j\le p$ with $(i,j)\notin\mc E$; $=z_k$ if $k=p^2+1$. Moreover, define $G \Let I-\frac{X}{\dualvar}$,
	\begin{align*}
		H \Let \begin{bmatrix} X^{-1} \otimes X^{-1}+  \frac{2}{\dualvar} G^{-1}\covsa G^{-1} \otimes G^{-1} &
		 - \frac{1}{\dualvar^2} \vect(G^{-1} [X G^{-1}\covsa +\covsa G^{-1} X ]G^{-1} ) \\ - \frac{1}{\dualvar^2} \vect(G^{-1} [X G^{-1}\covsa +\covsa G^{-1} X] G^{-1})^\top &
		 \frac{2}{\dualvar^3} \Trace[G^{-1}X G^{-1}\covsa G^{-1} X] \end{bmatrix}  \in \mathbb S^{p^2+1}
	\end{align*}
	and
	\begin{align*}
		g \Let \begin{bmatrix} \vect(G^{-1} \covsa G^{-1}-X^{-1}) \\ \rho^2+\Trace [G^{-1}\covsa (I - \frac{1}{\dualvar}G^{-1} X) -\covsa] \end{bmatrix} \in\R^{p^2+1}.
	\end{align*}
	Then, the unique solution $(\DX\opt,\Dg\opt)\in \mathbb S^p\times\R$ of the linear system
	\begin{equation}
	\label{eq:descent-direction}
	PH \left( (\vect(\DX\opt)^\top, \Dg\opt )^\top +g\right)= 0 \quad \text{and} \quad (\DX\opt)_{ij}=0 \quad \forall (i, j) \in \mc E
	\end{equation}
	represents a feasible descent direction for the optimization problem~\eqref{eq:Reformulation2} at $(X,\gamma)$.
\end{proposition}

\begin{proof}
	We first expand the objective function of problem~\eqref{eq:Reformulation2} around $(X,\dualvar)\in \PD^p\times \R_{++}$ with $\dualvar I\succ X$. By the rules of matrix calculus, the second-order Taylor expansion of the negative log-determinant is given by
	\begin{align*}
		- \log \det& (X + \DX) = - \log \det (X) - \Tr {X^{-1} \DX} + \frac{1}{2} \Tr {X^{-1} \DX X^{-1} \DX } +\mc O(\|\DX\|^3) 
	\end{align*} 
	for $\DX\in\mathbb S^p$, see also~\cite[page 644]{ref:Boyd-04}. 
	Moreover, by using a geometric series expansion, we obtain
		\begin{align*}
			\left(I - \frac{X + \DX}{\dualvar + \Dg}\right)^{-1}
			&= \left(I - \frac{X + \DX}{\dualvar} \left( 1-\frac{\Dg}{\dualvar} + \frac{\Dg^2}{\dualvar^2} +\mc O(\|\Dg\|^3)\right) \right)^{-1} \\
			&= \left( I - \frac{X}{\dualvar} + \frac{X \Dg}{\dualvar^2} -\frac{X \Dg^2}{\dualvar^3} -\frac{\DX}{\dualvar} + \frac{\DX\Dg}{\dualvar^2}
			+\mc O(\| (\DX,\Dg)\|^3) \right)^{-1}
		\end{align*}
		for $\Dg\in\R$. Expanding the matrix inverse as a Neumann series and setting $G=I-\frac{X}{\dualvar}$, which is invertible because $\dualvar I\succ X$, the above expression can be reformulated as
		\begin{align*}
			& \textstyle G^{-\frac{1}{2}} \left( I +\frac{G^{-\frac{1}{2}} X G^{-\frac{1}{2}} \Dg}{\dualvar^2} - \frac{G^{-\frac{1}{2}} X G^{-\frac{1}{2}} \Dg^2}{\dualvar^3} 
			- \frac{G^{-\frac{1}{2}}\DX G^{-\frac{1}{2}}}{\dualvar}
			+ \frac{G^{-\frac{1}{2}}\DX G^{-\frac{1}{2}} \Dg}{\dualvar^2} +\mc O(\| (\DX,\Dg)\|^3) \right)^{-1} G^{-\frac{1}{2}}\\
			= \;&G^{-1} - \frac{G^{-1} X G^{-1} \Dg}{\dualvar^2} + \frac{G^{-1} X G^{-1} \Dg^2}{\dualvar^3} + \frac{G^{-1}\DX G^{-1}}{\dualvar} 
			- \frac{G^{-1}\DX G^{-1}\DX G^{-1} \Dg}{\dualvar^2} + \frac{G^{-1} X G^{-1} X G^{-1} \Dg^2}{\dualvar^4} \\
			& + \frac{G^{-1}\DX G^{-1} \DX G^{-1}}{\dualvar^2}
			- \frac{G^{-1} X G^{-1} \DX G^{-1}\Dg}{\dualvar^3} - \frac{G^{-1} \DX G^{-1} X G^{-1}\Dg}{\dualvar^3} +\mc O(\| (\DX,\Dg)\|^3)\,.
	\end{align*}
	Thus, the second-order Taylor expansion of the last term in the objective function of~\eqref{eq:Reformulation2} is given by
	\begin{align*}
		& (\dualvar + \Dg)^2 \Tr{\left((\dualvar + \Dg)I - (X + \DX) \right)^{-1}\covsa} = (\dualvar + \Dg) \Tr{\left(I - \frac{X + \DX}{\dualvar + \Dg}\right)^{-1}\covsa}\\
		= & \; 
		\dualvar \Tr{G^{-1}\covsa} + \Dg \Tr{G^{-1}\covsa (I - \frac{1}{\dualvar}G^{-1} X) } + \frac{\Dg^2}{\dualvar^3} \Tr{G^{-1} X G^{-1}\covsa G^{-1} X} \\
		& + \Tr{G^{-1} \covsa G^{-1} \DX }  - \frac{\Dg}{\dualvar^2} \Tr{ G^{-1}\covsa G^{-1} \DX G^{-1}X + G^{-1}\covsa G^{-1} X G^{-1} \DX}  \\
		& + \frac{1}{\dualvar} \Tr{G^{-1} \DX G^{-1}\covsa G^{-1} \DX} +\mc O(\| (\DX,\Dg)\|^3) \,,
	\end{align*}
	where the second equality follows from the Taylor expansion of the matrix inverse derived above. Using Lemma~\ref{lem:kronecker}, the objective function of~\eqref{eq:Reformulation2} is thus representable~as
	\begin{align*}
		- \log \det (X + \DX) &+ (\dualvar +\Dg) \left( \rho^2 - \Tr{\covsa} \right)	+ (\dualvar + \Dg)^2 \Tr{\left((\dualvar + \Dg)I - (X + \DX) \right)^{-1}\covsa} \\
		=~& c + g^\top ( \vect(\DX)^\top, \Dg )^\top + \frac{1}{2} ( \vect(\DX)^\top, \Dg ) H ( \vect(\DX)^\top, \Dg )^\top
		+\mc O(\| (\DX,\Dg)\|^3)
	\end{align*}
	for some $c\in\R$, where the gradient $g\in\R^p$ and the Hessian $H\in\mathbb S^p$ are defined as in the proposition statement. A feasible descent direction for problem~\eqref{eq:Reformulation2} is thus obtained by solving the auxiliary quadratic program
	\be \label{eq:quad:3}
	\begin{array}{cl}
		\Min{\DX, \Dg} & g^\top ( \vect(\DX)^\top, \Dg )^\top + \frac{1}{2} ( \vect(\DX)^\top, \Dg ) H ( \vect(\DX)^\top, \Dg )^\top \\
		\st & \DX\in\mathbb S^p,~ (\DX)_{ij}=0 \quad \forall (i, j) \in \mc E 
	\end{array}
	\ee
	Note that~\eqref{eq:quad:3} has a unique minimizer because $H$ is positive definite. Indeed, we have
        \begin{align*}
        & \frac{4}{\dualvar^4} \vect(G^{-1} X G^{-1} \covsa G^{-1})^\top
        \left(X^{-1} \otimes X^{-1} + \frac{2}{\dualvar} G^{-1} \covsa G^{-1} \otimes G^{-1} \right)^{-1} \vect(G^{-1} X G^{-1} \covsa G^{-1}) \\
        <& \frac{4}{\dualvar^4} \vect(G^{-1} X G^{-1} \covsa G^{-1})^\top
        \left(\frac{2}{\dualvar} G^{-1} \covsa G^{-1} \otimes G^{-1} \right)^{-1} \vect(G^{-1} X G^{-1} \covsa G^{-1}) \\
        =& \frac{2}{\dualvar^3} \vect(G^{-1} X G^{-1} \covsa G^{-1})^\top
        \left( G \covsa^{-1} G \otimes G \right) \vect(G^{-1} X G^{-1} \covsa G^{-1}) \label{eq:kron:inverse} \\
        =& \frac{2}{\dualvar^3} \Tr{G^{-1} X G^{-1} \covsa G^{-1} X},
        \end{align*}
	where the inequality holds because $X\otimes X$ is positive definite and $G^{-1} X G^{-1} \covsa G^{-1}\neq 0$, the first equality follows from \cite[Proposition~7.1.7]{ref:Bernstein-2009}, which asserts that $(A \otimes B)^{-1} = A^{-1} \otimes B^{-1}$ for any $A, B\in\PD^p$, and the second equality follows from Lemma~\ref{lem:kronecker}. The above derivation shows that the Schur complement of the positive definite block $X^{-1} \otimes X^{-1} + \frac{2}{\dualvar} G^{-1} \covsa G^{-1} \otimes G^{-1}$ in $H$ is a positive number, which in turn implies that the Hessian $H$ is positive definite. In the following, we denote the unique minimizer of~\eqref{eq:quad:3} by~$(\DX\opt, \Dg\opt)$. As $\DX=0$ and $\Dg=0$ is feasible in~\eqref{eq:quad:3}, it is clear that the objective value of $(\DX\opt, \Dg\opt)$ is nonpositive. In fact, as $H\succ 0$, the minimum of~\eqref{eq:quad:3} is negative unless $g=0$. Thus, $(\DX\opt, \Dg\opt)$ is a feasible descent direction.
	
	Note that $P$ defined in the proposition statement represents the orthogonal projection on the linear space 
	\[
		\mc Z = \left\{ z = (\vect(\DX)^\top, \Dg )^\top\in\R^{p^2+1}: \DX\in\mathbb S^p,\quad(\DX)_{ij}=0\quad\forall (i,j)\in\mc E \right\}.
	\]
	Indeed, it is easy to verify that $P^2=P=P^\top$ because the range and the null space of $P$ correspond to $\mc Z$ and its orthogonal complement, respectively. The quadratic program~\eqref{eq:quad:3} is thus equivalent to
	\[
		\min_{z\in\mc Z}\; \left\{g^\top z+\frac{1}{2}z^\top Hz \right\} = \min_{z\in\R^{p^2+1}} \left\{ g^\top z+\frac{1}{2}z^\top Hz: Pz=z \right\}.
	\]
	The minimizer $z\opt$ of the last reformulation and the optimal Lagrange multiplier $\mu\opt$ associated with its equality constraint correspond to the unique solution of the Karush-Kuhn-Tucker optimality conditions
	\[
		Hz\opt+g +(I-P)\mu\opt = 0, ~ (1-P)z\opt=0 \quad \iff \quad P(Hz\opt+g) = 0, ~ (1-P)z\opt=0,
	\]
	which are mainfestly equivalent to~\eqref{eq:descent-direction}. Thus, the claim follows.
\end{proof}

Given a descent direction $(\DX\opt, \Dg\opt)$ at a feasible point $(X,\dualvar)$, we use a variant of Armijo's rule \cite[Section~3.1]{ref:Nocedal-06} to choose a step size $\alpha>0$ that preserves feasibility of the next iterate $(X+\alpha \DX\opt, \dualvar+\alpha \Dg\opt)$ and ensures a sufficient decrease of the objective function. Specifically, for a prescribed line search parameter $\sigma\in(0,\frac{1}{2})$, we set the step size $\alpha$ to the largest number in $\{ \frac{1}{2^m}\}_{m\in\mathbb Z_+}$ satisfying the following two conditions:
\begin{itemize}
	\item[(C1)] Feasibility: $(\dualvar + \alpha \Dg\opt) I \succ X + \alpha \DX\opt \succ 0$;
	\item[(C2)] Sufficient decrease: $f(X + \alpha \DX\opt, \dualvar + \alpha \Dg\opt) \le f(X, \dualvar) + \sigma \alpha \delta$, where $\delta=g^\top ( \vect(\DX\opt)^\top, \Dg\opt )^\top < 0$, and $g$ is defined as in Propostion~\ref{prop:DXDl:optimal}.
\end{itemize}
Notice that the sparsity constraints are automatically satisfied at the next iterate thanks to the construction of the descent direction $(\DX\opt, \Dg\opt)$ in \eqref{eq:descent-direction}. Algorithm~\ref{algo:QuadraticApprox} repeats the procedure outlined above until $\|g\|$ drops below a given tolerance ($10^{-3}$) or until the iteration count exceeds a given threshold ($10^2$). Throughout the numerical experiments in Section~\ref{sec:num-res} we set $\sigma = 10^{-4}$, which is the value recommended in~\cite{ref:Nocedal-06}.

\begin{algorithm}[h]
	\caption{Sequential quadratic approximation algorithm}
	\KwData{Sample covariance matrix $\covsa \succ 0$, Wasserstein radius $\rho > 0$, line search parameter $\sigma\in(0,\frac{1}{2})$.}
	Initialize $X_0 = I$ and $\dualvar_0 >1$, and set $t \leftarrow 0$\;
	\While{stopping criterion is violated} {
		Find the descent direction $(\DX\opt, \Dg\opt)$ at $(X,\dualvar)=(X_t,\dualvar_t)$ by solving \eqref{eq:descent-direction};\\
		Find the largest step size $\alpha_t\in\{ \frac{1}{2^m}\}_{m\in\mathbb Z_+}$ satisfying (C1) and (C2);\\
		Set $X_{t+1} = X_t + \alpha_t \DX\opt$, $\dualvar_{t+1} = \dualvar_t + \alpha_t \Dg\opt$;\\
		Set $t \leftarrow t + 1$;
	}
	\label{algo:QuadraticApprox}
\end{algorithm}

\begin{remark}[Steepest descent algorithm] The computation of the descent direction in Proposition~\ref{prop:DXDl:optimal} requires second-order information. It is easy to verify that Proposition~\ref{prop:DXDl:optimal} remains valid if the Hessian $H$ is replaced with the identity matrix, in which case the sequential quadratic approximation algorithm reduces to the classical steepest descent algorithm \cite[Chapter~3]{ref:Nocedal-06}.
\end{remark}

The next proposition establishes that Algorithm~\ref{algo:QuadraticApprox} converges to the unique minimizer of problem~\eqref{eq:Reformulation2}. 

\begin{proposition}[Convergence]
	\label{prop:convergence}
	Assume that $\covsa\succ 0$, $\rho>0$ and $\sigma\in(0,\frac{1}{2})$. For any initial feasible solution $(X_0, \dualvar_0)$, the sequence~$\big\{(X_t, \dualvar_t)\big\}_{t \in \mbb Z_+}$ generated by Algorithm~\ref{algo:QuadraticApprox} converges to the unique minimizer $(X\opt, \dualvar\opt)$ of problem~\eqref{eq:Reformulation2}.
	Moreover, the sequence converges locally quadratically.
\end{proposition}

\begin{proof}
Denote by $f(X,\dualvar)$ the objective function of problem~\eqref{eq:Reformulation2}, and define
\[
	\mc C \Let \big\{ (X, \dualvar) \in \X \times \R_{+} : f(X, \dualvar) \leq f(X_0, \dualvar_0),~ 0 \prec X \prec \dualvar I\big\} 
\]
as the set of all feasible solutions that are at least as good as the initial solution $(X_0,\dualvar_0)$. The proof of Theorem~\ref{thm:refor} implies that $\underline{x} I \preceq X \preceq \overline{x} I$ and $\underline{x} \le \dualvar \le \overline{x}$ for all $(X, \dualvar) \in \mc C$, where the strictly positive constants $\underline{x}$ and $\overline{x}$ are defined as in \eqref{eq:bound:def}. Note that, as $\covsa$ is fixed in this proof, the dependence of $\underline{x}$ and $\overline{x}$ on $\covsa$ is notationally suppressed to avoid clutter. Thus, $\mc C$ is bounded. Moreover, as $\covsa\succ 0$, it is easy to verify $f(X,\dualvar)$ tends to infinity if the smallest eigenvalue of $X$ approaches 0 or if the largest eigenvalue of $X$ approaches $\gamma$. The continuity of $f(X,\dualvar)$ then implies that $\mc C$ is closed. In summary, we conclude that $\mc C$ is compact.


By the definition of $f(X,\dualvar)$ in~\eqref{eq:Reformulation2}, any $(X, \dualvar) \in \mc C$ satisfies
\begin{align*}
	0& \le f(X_0,\dualvar_0)+\log\det(X)-\dualvar \left(\rho^2-\Tr{\covsa}\right)-\dualvar \inner{( I - \dualvar^{-1} X)^{-1}}{\covsa} \\
	& \le f(X_0,\dualvar_0)+p\log(\overline x)+ \overline x \Tr{\covsa}-\underline x \, \eigval_{\min} \Tr{( I - \dualvar^{-1} X)^{-1}},
\end{align*}
where $\eigval_{\min}$ denotes the smallest eigenvalue of $\covsa$, which is positive by assumption. Thus, we have
\begin{align*}
	\Tr{( I - \dualvar^{-1} X)^{-1}} \le \frac{1}{\underline x \, \eigval_{\min}}\left(f(X_0,\dualvar_0)+p\log(\overline x)+ \overline x \Tr{\covsa}\right),
\end{align*}
which implies that the eigenvalues of $I-\frac{X}{\dualvar}$ are uniformly bounded away from 0 on $\mc C$. More formally, there exists $c_0>0$ with $I-\frac{X}{\dualvar}\succ c_0 I$ for all $(X,\dualvar)\in\mc C$. As the objective function $f(X,\dualvar)$ is smooth wherever it is defined, its gradient and Hessian constitute continuous functions on $\mc C$. Moreover, as $f(X,\dualvar)$ is strictly convex on the compact set~$\mc C$, the eigenvalues of its Hessian matrix are uniformly bounded away from 0. This implies that the inverse Hessian matrix and the descent direction $(\DX\opt,\Dg\opt)$ constructed in Proposition~\ref{prop:DXDl:optimal} are also continuous on~$\mc C$. Hence, there exist $c_1, c_2>0$ such that $\DX\opt \preceq c_1 I$ and $| \Dg\opt| \leq c_2$ uniformly on~$\mc C$. 

We conclude that any positive step size $\alpha<\underline{x} \,\min\left\{ c_1^{-1}, (c_1 + c_2)^{-1} c_0 \right\}$ satisfies the feasibility condition~(C1) uniformly on $\mc C$ because $X + \alpha  \DX\opt  \succ \big(\underline{x} - \alpha c_1\big) I \succeq 0$ and
\begin{align*}
	(\gamma + \alpha \Dg\opt) I & \succeq X + c_0 \underline{x} I+ \alpha \big(\DX\opt - \DX\opt + \Dg\opt I \big) \succeq X + c_0 \underline{x} I + \alpha\big( \DX\opt -(c_1+c_2)I \big) \succ X + \alpha  \DX\opt
\end{align*}
for all $(X,\dualvar)\in\mc C$. 
Moreover, by \cite[Lemma~5(b)]{ref:Tseng-2009} there exists $\overline\alpha>0$ such that any positive step size $\alpha \le \overline \alpha$ satisfies the descent condition~(C2) for all $(X,\dualvar)\in\mc C$. In summary, there exists $m\opt\in\mbb Z_+$ such that 
\[
	\alpha\opt=\frac{1}{2^{m\opt}}<\min\left\{ \overline \alpha, \underline{x} \,\min\left\{ c_1^{-1}, (c_1 + c_2)^{-1} c_0 \right\} \right\}
\]
satisfies both line search conditions~(C1) and~(C2) uniformly on $\mc C$. By induction, the iterates $\{ (X_t, \dualvar_t)\}_{t \in \mbb N}$ generated by Algorithm~\ref{algo:QuadraticApprox} have nonincreasing objective values and thus all belong to $\mc C$, while the step sizes $\{ \alpha_t\}_{t \in \mbb N}$ generated by Algorithm~\ref{algo:QuadraticApprox} are all larger or equal to $\alpha\opt$. Hence, the algorithm's global convergence is guaranteed by \cite[Theorem 1]{ref:Tseng-2009}, while the local quadratic convergence follows from \cite[Theorem~16]{ref:Hsieh-2014}.
\end{proof}

\begin{remark}[Refinements of Algorithm~\ref{algo:QuadraticApprox}]
For large values of $p$, computing and storing the exact Hessian matrix $H$ from Proposition~\ref{prop:DXDl:optimal} is prohibitive. In this case, $H$ can be approximated by a low-rank matrix as in the limited-memory Broyden-Fletcher-Goldfarb-Shanno (BFGS) method without sacrificing global convergence~\cite{ref:Tseng-2009}. Alternatively, one can resort to a coordinate descent method akin to the QUIC algorithm \cite{ref:Hsieh-2014}, in which case both the global and local convergence guarantees of Proposition~\ref{prop:convergence} remain valid.
\end{remark}

\section{Extremal Distributions}
\label{sec:wc-dist}
It is instructive to characterize the extremal distributions that attain the supremum in~\eqref{eq:g:def} for a given sample covariance matrix $\covsa$ and a fixed candidate estimator $X$. 

\begin{theorem}[Extremal distributions]
	\label{thm:worst:case}
	For any $\covsa, X \in \PD^p$, the supremum in~\eqref{eq:g:def} is attained by the normal distribution $\Q\opt= \mc N(0,S\opt)$ with covariance matrix
	\[
	S\opt = (\estdual)^2  (\estdual I - X)^{-1} \covsa (\estdual I - X)^{-1},
	\]
	where $\estdual$ is the unique solution with $\estdual I \succ X$ of the following algebraic equation
	\be
	\label{eq:g:FOC}
	\rho^2 - \Tr{\covsa} + 2 \estdual \Tr {(\estdual I - X)^{-1} \covsa}  - (\estdual)^2 \Tr {(\estdual I - X)^{-1} \covsa (\estdual I - X)^{-1}} = 0\,.
	\ee
\end{theorem}
\begin{proof}
	From Proposition~\ref{prop:g-refor} we know that the worst-case expectation problem~\eqref{eq:g:def} is equivalent to the semidefinite program~\eqref{eq:g:refor}. Note that the strictly convex objective function of~\eqref{eq:g:refor} is bounded below by
	\[
		\dualvar \left(\rho^2 - \Tr{\covsa} \right) + \lambda_{\rm min} \dualvar^2 \Tr{(\dualvar I - X)^{-1}},
	\]
	where $\lambda_{\rm min}$ denotes the smallest eigenvalue of $\covsa$. As $\lambda_{\rm min}$ is positive by assumption, the objective function of~\eqref{eq:g:refor} tends to infinity as $\dualvar$ approaches the largest eigenvalue of $X$, in which case $\dualvar I-X$ becomes singular. Thus, the unique optimal solution $\estdual$ of~\eqref{eq:g:refor} satisfies $\estdual I\succ X$ and solves the first-order optimality condition~\eqref{eq:g:FOC}.
	
	Now we are ready to prove that $\Q\opt$ is both feasible and optimal in~\eqref{eq:g:def}. By the formula for $S\opt$ in terms of $\estdual$, $\covsa$ and $S$ and by using Definition~\ref{def:Wass-for-S} and Proposition~\ref{prop:Wass}, it is easy to verify that~\eqref{eq:g:FOC} is equivalent to
	\[
	\Tr{S\opt} + \Tr{\covsa} - 2\Tr{ \sqrt{\covsa^\half S\opt \covsa^\half} } = \rho^2 \quad \iff \quad \V(S\opt, \covsa) =\Wass(\Q\opt, \wh{\mbb P}) = \rho,
	\]
	which confirms that $\Q\opt$ is feasible in~\eqref{eq:g:def}. Moreover, the objective value of $\Q\opt$ in~\eqref{eq:g:def} amounts to
	\begin{align*}
			\EE^{\Q\opt}[ \inner{\xi \xi^\top}{X}] &= \inner{S\opt}{X} = (\estdual)^2 \inner{(\estdual I - X)^{-1} \covsa (\estdual I - X)^{-1} }{X} \\
			&= (\estdual)^2 \inner{(\estdual I - X)^{-1} \covsa (\estdual I - X)^{-1} }{(X - \estdual I) + \estdual I} \\
			&= -(\estdual)^2 \Tr{(\estdual I - X)^{-1}\covsa} + (\estdual)^3 \Tr{(\estdual I - X)^{-1} \covsa (\estdual I - X)^{-1}} \\
			&=\estdual (\rho^2 - \Tr{\covsa}) + (\estdual)^2 \inner{(\estdual I - X)^{-1}}{\covsa} = g(\covsa,S\opt),
	\end{align*}
	where the penultimate equality exploits~\eqref{eq:g:FOC}, while the last equality follows from the optimality of $S\opt$ in~\eqref{eq:g:refor} and from Proposition~\ref{prop:g-refor}. Thus, $\Q\opt$ is optimal in~\eqref{eq:g:def}.
\end{proof}

In the absence of sparsity information (that is, if $\X=\PD^p$), the unique minimizer $X\opt$ of problem~\eqref{eq:DROSimplified} is available in closed form thanks to Theorem~\ref{thm:main:theorem}. In this case, the extremal distribution attaining the supremum in~\eqref{eq:g:def} at $X=X\opt$ can also be computed in closed form even if $\covsa$ is rank deficient.

\begin{corollary}[Extremal distribution for optimal estimator]
	\label{corollary:worst:case:est}
	Assume that $\rho > 0$, $\X = \PD^p$ and $\covsa\in\PSD^p$ admits the spectral decomposition $\covsa = \sum_{i=1}^p \covsaeig_i \covsaeigvect_i \covsaeigvect_i^\top$ with eigenvalues $\covsaeig_i$ and corresponding orthonormal eigenvectors~$\covsaeigvect_i$, $i\le p$.  If $(\est, \estdual)$ represents the unique solution of~\eqref{eq:Reformulation2} given in Theorem~\ref{thm:main:theorem}, then the supremum in~\eqref{eq:g:def} at $X=X\opt$ is attained by the normal distribution $\Q\opt= \mc N(0,S\opt)$ with covariance matrix 
	\[
		S\opt =\sum_{i=1}^p s\opt_i \covsaeigvect_i \covsaeigvect_i^\top, \quad \text{where} \quad s\opt_i = \begin{cases} (\estdual)^2 \eigval_i (\estdual - \estx_i)^{-2} & \text{if } \eigval_i > 0, \\
		(\estdual)^{-1} & \text{if } \eigval_i = 0.
	\end{cases}
	\]
\end{corollary}
\begin{proof}
If $\covsa\succ 0$, the claim follows immediately by substituting the formula for $\est$ from Theorem~\ref{thm:main:theorem} into the formula for $S\opt$ from Theorem~\ref{thm:worst:case}. If $\covsa\succeq 0$ is rank deficient, we consider the invertible sample covariance matrix $\covsa+\eps I\succ 0$ for some $\eps>0$, denote by $(\est(\eps),\estdual(\eps))$ the corresponding minimizer of problem~\eqref{eq:Reformulation2} as constructed in~\eqref{eq:analytical} and let $S\opt(\eps)$ be the covariance matrix of the extremal distribution of problem~\eqref{eq:g:def} at $X=X\opt(\eps)$. Using the same reasoning as in the proof of Theorem~\ref{thm:main:theorem}, one can show that $(\est(\eps),\estdual(\eps))$ is continuous in $\eps\in\R_+$ and converges to $(\est,\estdual)$ as $\eps$ tends to $0$. Similarly, $S\opt(\eps)$ is continuous in $\eps\in\R_+$ and converges to $S\opt$ as $\eps$ tends to $0$. To see this, note that the eigenvalues $s_i\opt(\eps)$, $i\le p$, of $S\opt(\eps)$ satisfy 
\begin{align*}
	\lim_{\eps\ra0^+} s_i\opt(\eps) & = \lim_{\eps\ra0^+} \frac{\estdual(\eps)^2 (\eigval_i+\eps)}{(\estdual(\eps) - \estx_i(\eps))^2} \\
	& = \lim_{\eps\ra0^+} \frac{4 (\lambda_i+\eps)}{\left( \sqrt{(\lambda_i+\eps)^2 \estdual(\eps)^2+4 (\lambda_i+\eps) \estdual(\eps)} - (\lambda_i+\eps) \estdual(\eps)\right)^2} = s_i\opt \quad \forall i\le p\,,
\end{align*}
where the first equality follows from the first part of the proof, the second equality exploits~\eqref{eq:x:value} and the third equality holds due to the definition of $s\opt_i$.

We are now armed to prove that $\Q\opt$ is both feasible and optimal in~\eqref{eq:g:def}. Indeed, using the continuity of $S\opt (\eps)$ and~$\V(S_1,S_2)$ in their respective arguments, we find
\[
	\Wass(\Q\opt, \wh{\mbb P}) = \V(S\opt, \covsa) = \lim_{\eps\ra0^+} \V(S\opt(\eps), \covsa+\eps I)=\rho\,,
\]
where the last equality follows from the construction of $S\opt(\eps)$ in the proof of Theorem~\ref{thm:worst:case}. Thus, $\Q\opt$ is feasible in~\eqref{eq:g:def}. Similarly, using the continuity of $S\opt(\eps)$ and $X\opt(\eps)$ in $\eps$, we have
\[
	\EE^{\Q\opt}[ \inner{\xi \xi^\top}{X\opt}] = \inner{S\opt}{X\opt} = \lim_{\eps\ra0^+} \inner{S\opt(\eps)}{X\opt(\eps)} = \lim_{\eps\ra0^+} g(\covsa+\eps I,S\opt(\eps)) = g(\covsa,S\opt)\,,
\]
where the last two equalities follow from the construction of $S\opt(\eps)$ in the proof of Theorem~\ref{thm:worst:case} and the continuity of $g(\covsa, X)$ established in Proposition~\ref{prop:g-refor}, respectively. Thus, $\Q\opt$ is optimal in~\eqref{eq:g:def}.
\end{proof}


\section{Numerical Experiments}
\label{sec:num-res}
To assess the statistical and computational properties of the proposed Wasserstein shrinkage estimator, we compare it against two state-of-the-art precision matrix estimators from the literature.

\begin{definition}[Linear shrinkage estimator]
\label{def:lin-shrinkage}
Denote by $\diag(\covsa)$ the diagonal matrix of all sample variances. Then, the linear shrinkage estimator with mixing parameter $\alpha\in[0,1]$ is defined as 
	\[
	\est =  \left[(1 - \alpha) \covsa + \alpha \diag(\covsa)\right]^{-1}.
	\]
\end{definition}

The linear shrinkage estimator uses the diagonal matrix of sample variances as the shrinkage target. Thus, the sample covariances are shrunk to zero, while the sample variances are preserved. We emphasize that the most prevalent shrinkage target is a scaled identity matrix \cite{ref:Ledoit-2004}. The benefits of using $\diag(\covsa)$ instead are discussed in \cite[\S~2.4]{ref:SchaeferStrimmer-2005}. Note that while $\covsa$ is never invertible for $n<p$, $\diag(\covsa)$ is almost surely invertible whenever the true covariance matrix is invertible and $n>1$. Thus, the linear shrinkage estimator is almost surely well-defined for all $\alpha>0$. Moreover, it can be efficiently computed in $\mc O(p^3)$ arithmetic operations.

\begin{definition}[$\ell_1$-Regularized maximum likelihood estimator]
	The $\ell_1$-regularized maximum likelihood estimator with penalty parameter $\beta \ge 0$ is defined as
	\[
	\est = \arg \Min{X \succeq 0}  \left\{- \log \det X + \inner{\covsa}{X} + \beta \sum_{i,j=1}^p |X_{ij}| \right\}.
	\]
\end{definition}

Adding an $\ell_1$-regularization term to the standard maximum likelihood estimation problem gives rise to sparse---and thus interpretable---estimators \cite{ref:Banerjee-2008, ref:Friedman-2008}. The resulting semidefinite program can be solved with general-purpose interior point solvers such as SDPT3 or with structure-exploiting methods such as the QUIC algorithm, which enjoys a quadratic convergence rate and requires $\mc O(p^3)$ arithmetic operations per iteration~\cite{ref:Hsieh-2014}. 
In the remainder of this section we test the Wasserstein shrinkage, linear shrinkage and $\ell_1$-regularized maximum likelihood estimators on synthetic and real datasets. All experiments are implemented in MATLAB, and the corresponding codes are included in the \textbf{W}asserstein \textbf{I}nverse Covariance \textbf{S}hrinkage \textbf{E}stimator (WISE) package available at
 \url{https://www.github.com/nvietanh/wise}.
 
 \begin{remark}[Bessel's correction]
So far we used $\mc N(\wh\mu, \covsa)$ as the nominal distribution, where the sample covariance matrix $\covsa$ was identified with the (biased) maximum likelihood estimator. In practice, it is sometimes useful to use $\covsa/\kappa$ as the nominal covariance matrix, where $\kappa\in(0,1)$ is a Bessel correction that removes the bias; see, e.g., Sections~\ref{sect:LDA} and~\ref{sect:MinVar} below. 
Under the premise that $\X$ is a cone, it is easy to see that if $(\est, \dualvar\opt)$ is optimal in~\eqref{eq:Reformulation1} for a prescribed Wasserstein radius $\rho$ and a scaled sample covariance matrix $\covsa/\kappa$, then $(\kappa \est, \kappa\dualvar\opt)$ is optimal in~\eqref{eq:Reformulation1} for a scaled Wasserstein radius $\sqrt{\kappa}\rho$ and the original sample covariance matrix~$\covsa$. Thus, up to scaling, using a Bessel correction is tantamount to shrinking $\rho$.
\end{remark}

\subsection{Experiments with Synthetic Data}

Consider a $(p=20)$-variate Gaussian random vector $\xi$ with zero mean. The (unknown) true covariance matrix~$\cov_0$ of $\xi$ is constructed as follows. We first choose a density parameter $d\in\{12.5\%, 50\%, 100\%\}$. Using the legacy MATLAB~5.0 uniform generator initialized with seed~0, we then generate a matrix $C \in \mbb R^{p \times p}$ with $\lfloor d\times p^2 \rfloor$ randomly selected nonzero elements, all of which represent independent Bernoulli random variables taking the values $+1$ or $-1$ with equal probabilities. Finally, we set $\cov_0= (C^\top C+10^{-3}I)^{-1}\succ 0$.

As usual, the quality of an estimator $\est$ for the precision matrix $\cov_0^{-1}$ is evaluated using Stein's loss function
\[
	L (\est, \cov_0) = -\log \det ( \est \cov_0) + \inner{\est}{\cov_0} - p, 
\]
which vanishes if $\est=\Sigma_0^{-1}$ and is strictly positive otherwise~\cite{ref:Stein-1961}. 


\begin{figure*} [t]
	\centering
	\subfigure[Wasserstein shrinkage]{\label{fig:dense:W} 
		\includegraphics[width=0.31\columnwidth]{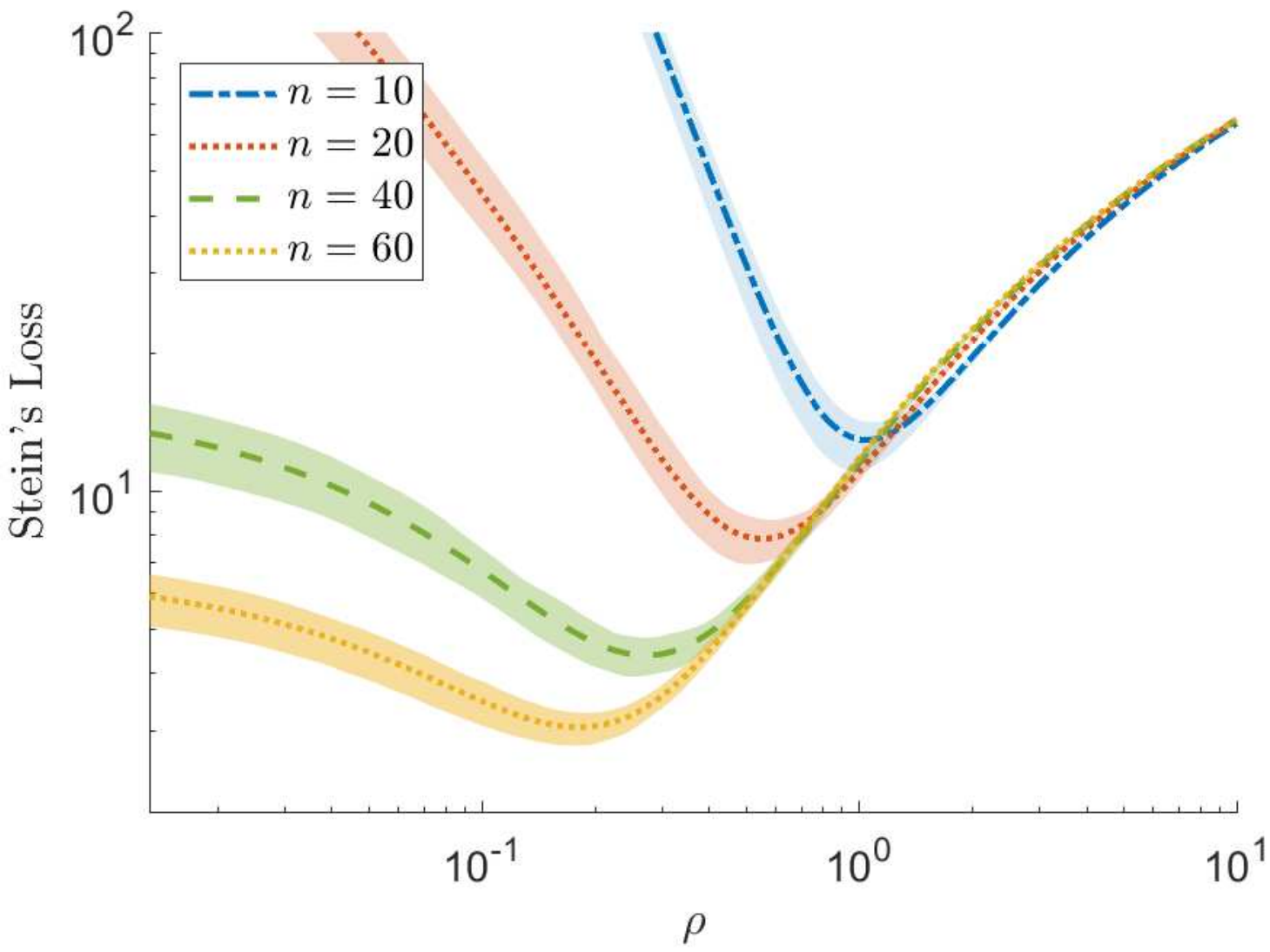}} \hspace{1mm}
	\subfigure[Linear shrinkage]{\label{fig:dense:S}
		\includegraphics[width=0.31\columnwidth]{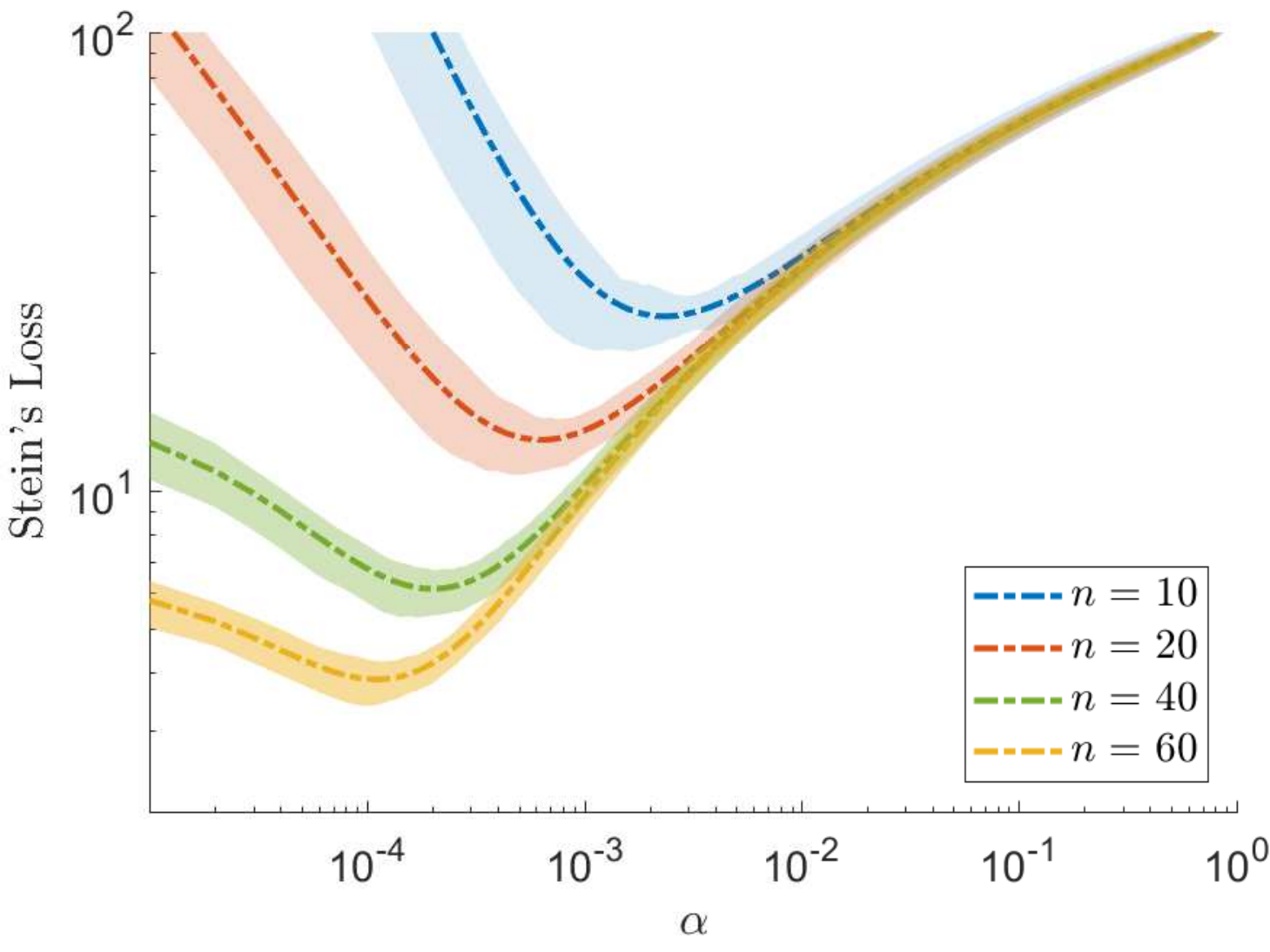}} \hspace{1mm}
	\subfigure[$\ell_1$-regularized ML]{\label{fig:dense:l1}
		\includegraphics[width=0.31\columnwidth]{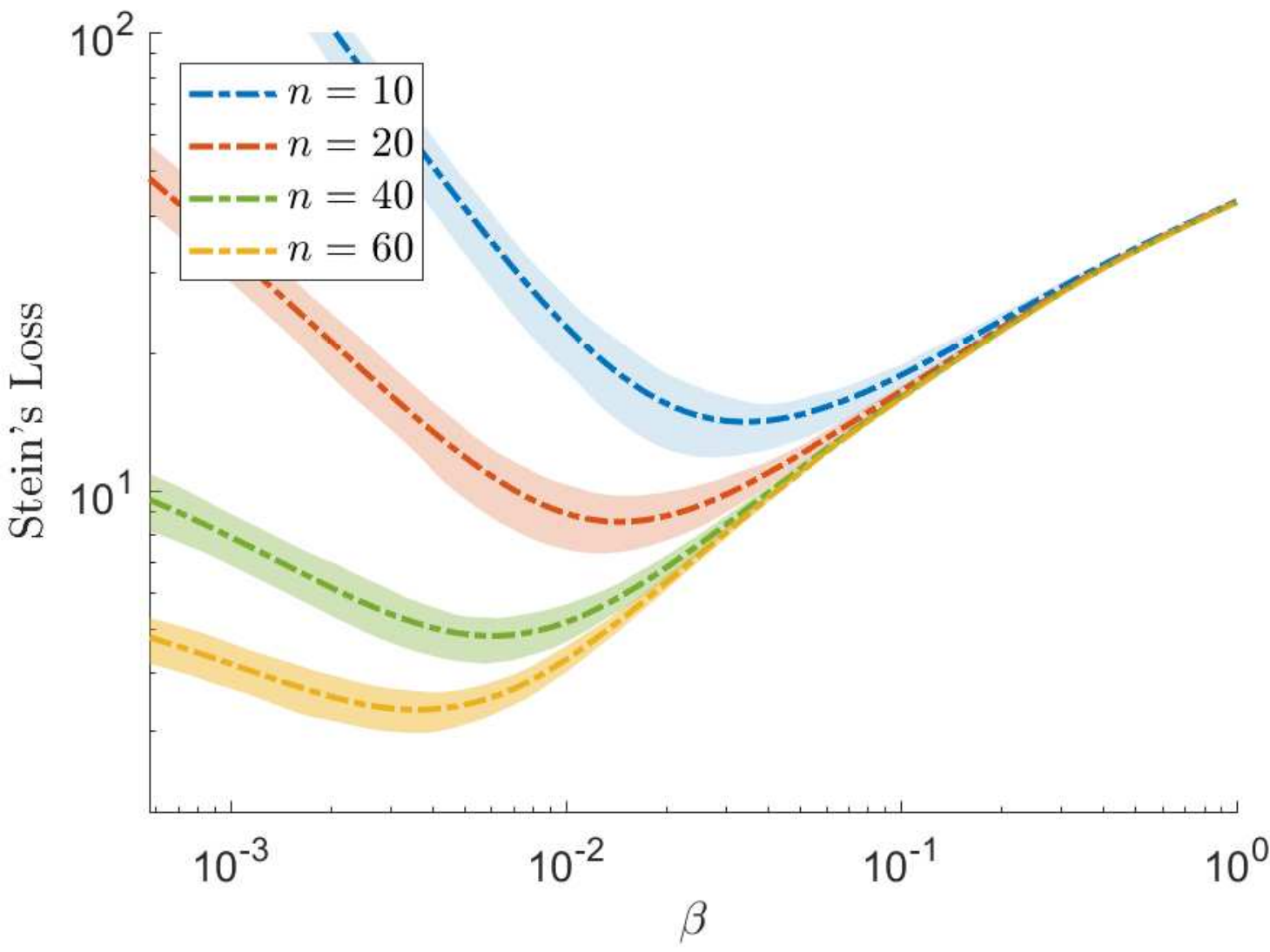}} \hspace{1mm}	
	\subfigure[Wasserstein shrinkage]{\label{fig:d200:W}
		\includegraphics[width=0.31\columnwidth]{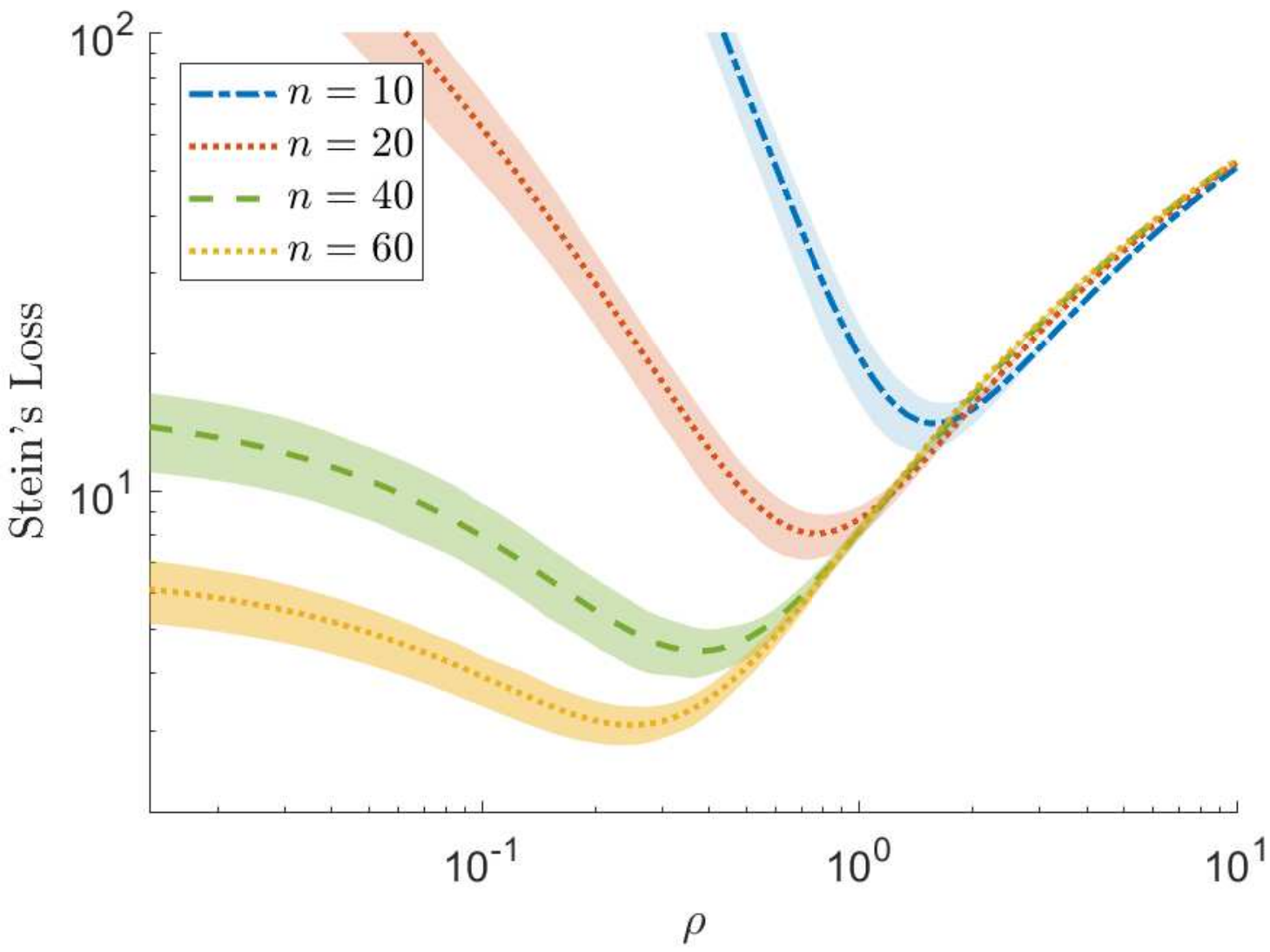}} \hspace{1mm}
	\subfigure[Linear shrinkage]{\label{fig:d200:S} 
		\includegraphics[width=0.31\columnwidth]{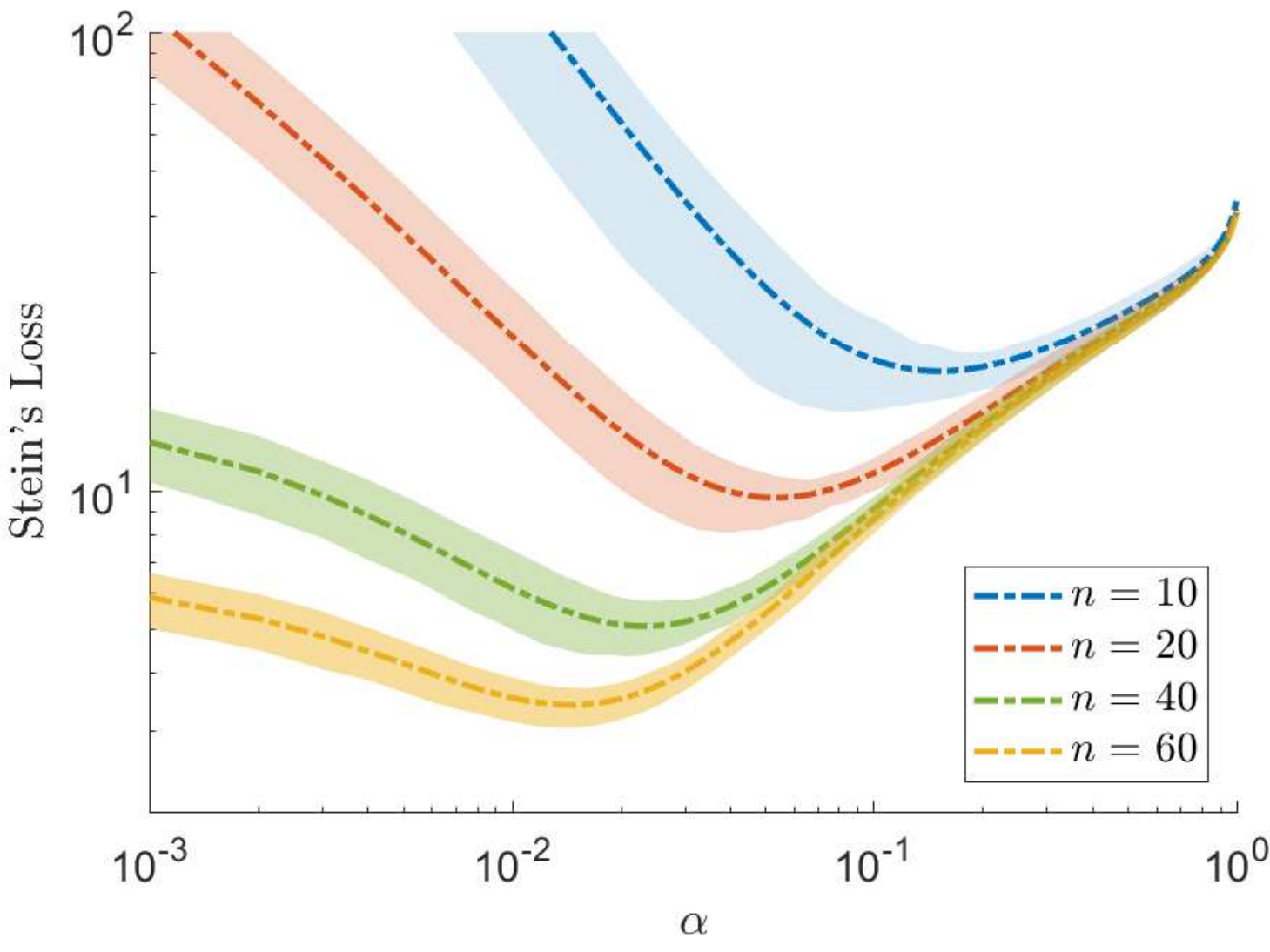}} \hspace{1mm}
	\subfigure[$\ell_1$-regularized ML]{\label{fig:d200:l1} 
		\includegraphics[width=0.31\columnwidth]{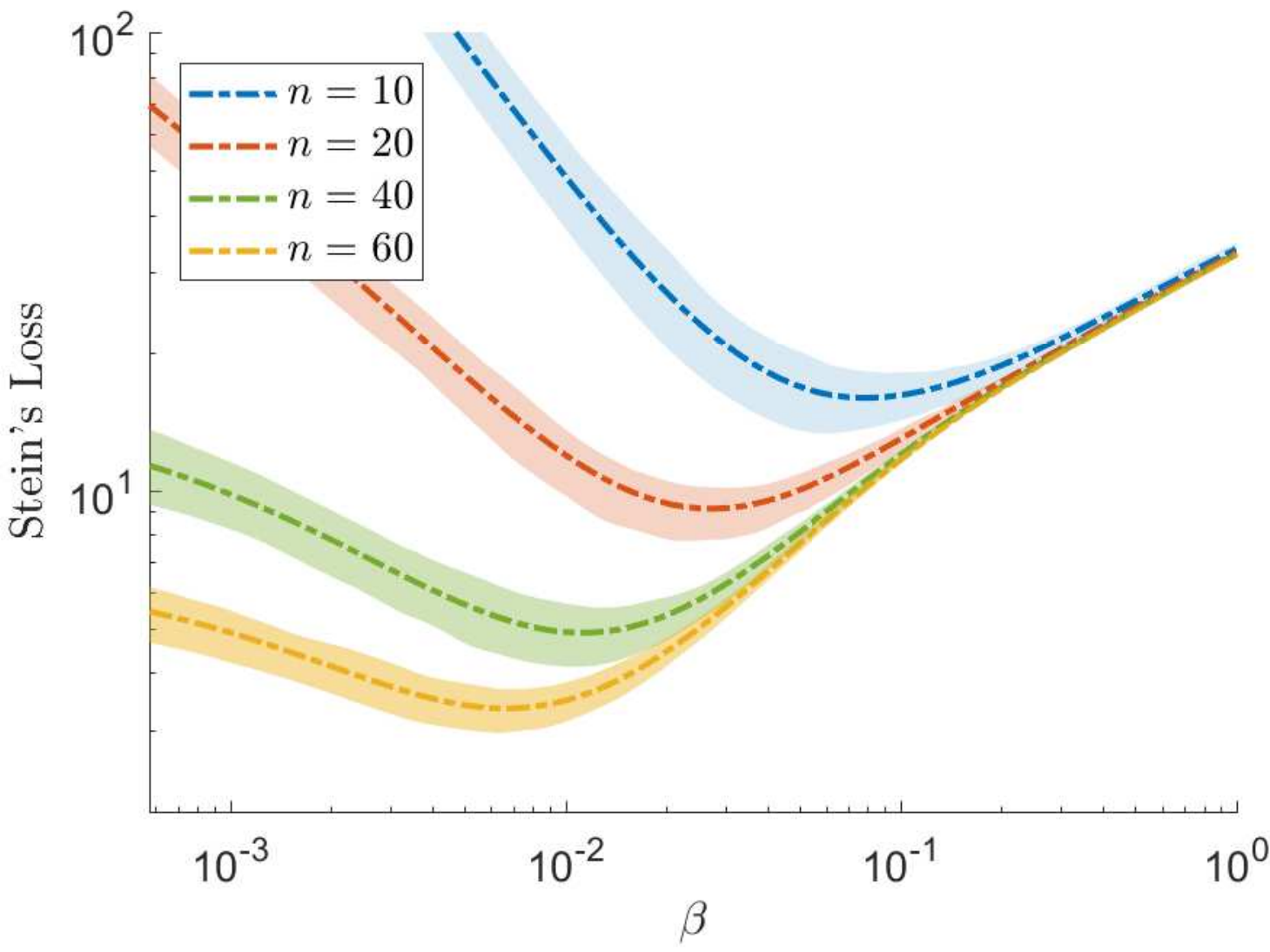}} \hspace{1mm}	
	\subfigure[Wasserstein shrinkage]{\label{fig:d50:W}
		\includegraphics[width=0.31\columnwidth]{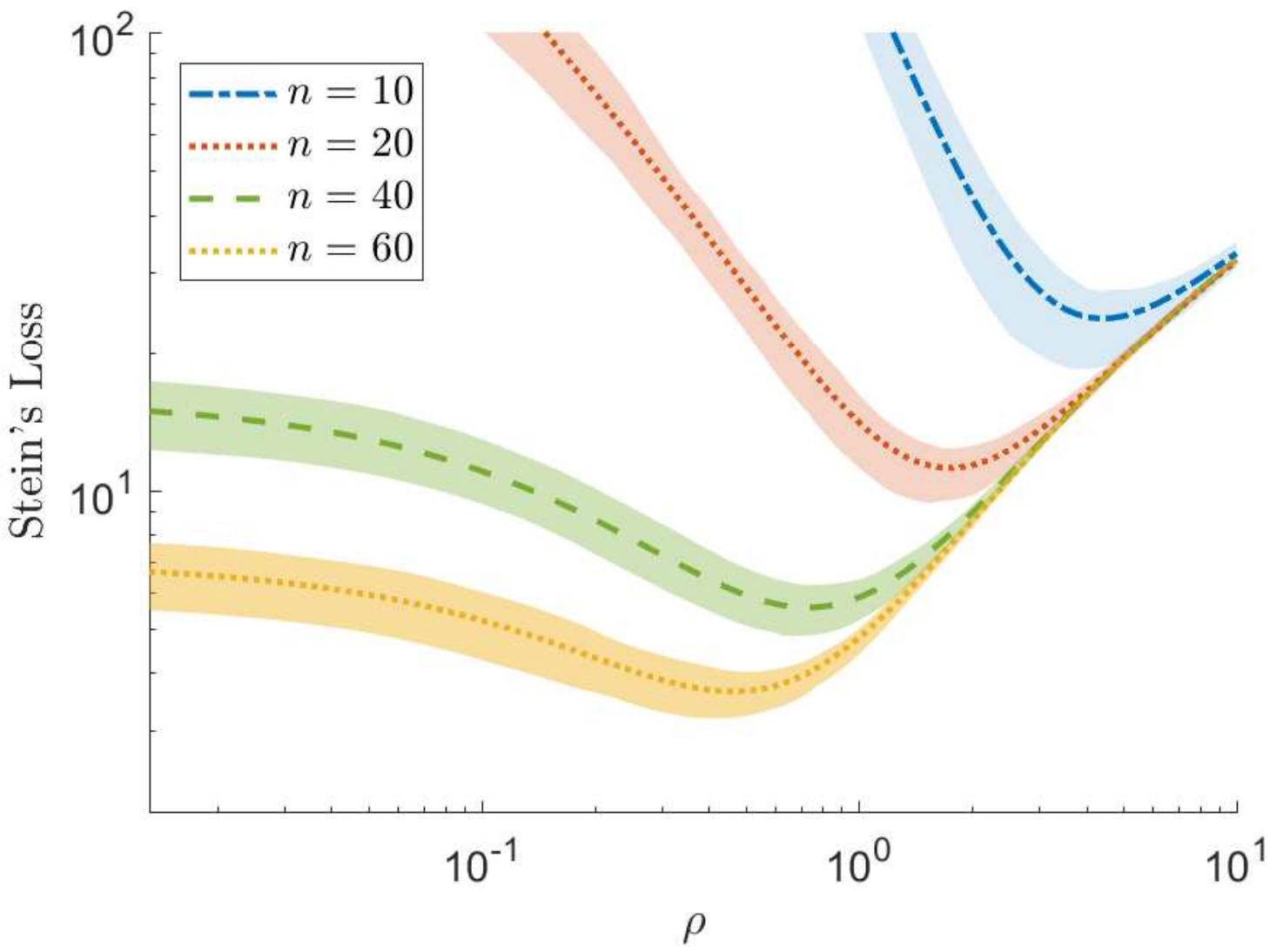}} \hspace{1mm}
	\subfigure[Linear shrinkage]{\label{fig:d50:S}
		\includegraphics[width=0.31\columnwidth]{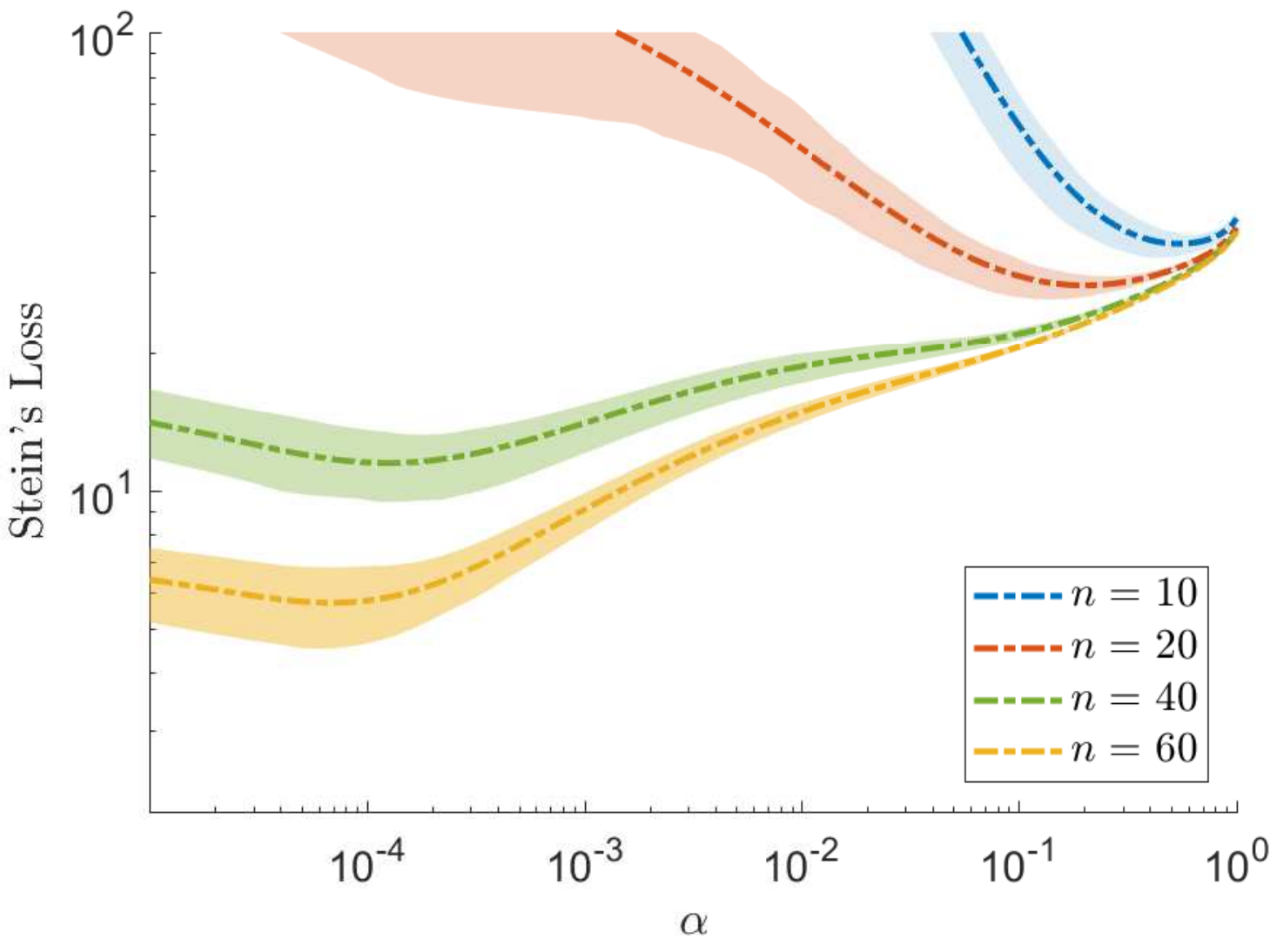}} \hspace{1mm}
	\subfigure[$\ell_1$-regularized ML]{\label{fig:d50:l1}
		\includegraphics[width=0.31\columnwidth]{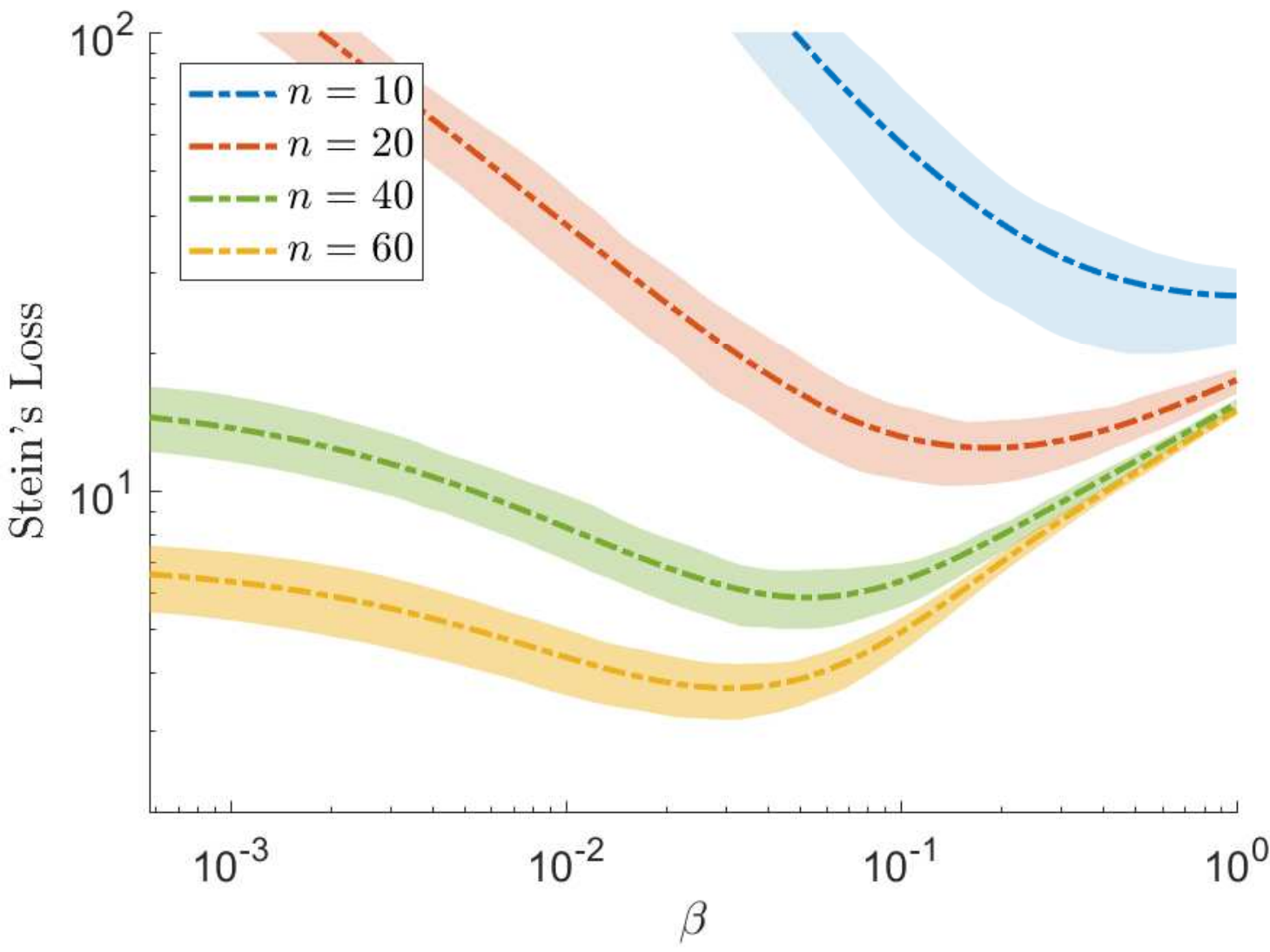}} 
	\caption{Stein's loss of the Wasserstein shrinkage, linear shrinkage and $\ell_1$-regularized maximum likelihood estimators as a function of their respective tuning parameters for $d=100\%$ (panels \ref{fig:dense:W}--\ref{fig:dense:l1}), $d=50\%$ (panels \ref{fig:d200:W}--\ref{fig:d200:l1}) and $d=12.5\%$ (panels \ref{fig:d50:W}--\ref{fig:d50:l1}).}
	\label{fig:perf:synthetic}
\end{figure*}

\begin{figure*} [t]
	\centering
	\subfigure[50\% sparsity information]{\label{fig:d200:W50} 
		\includegraphics[width=0.31\columnwidth]{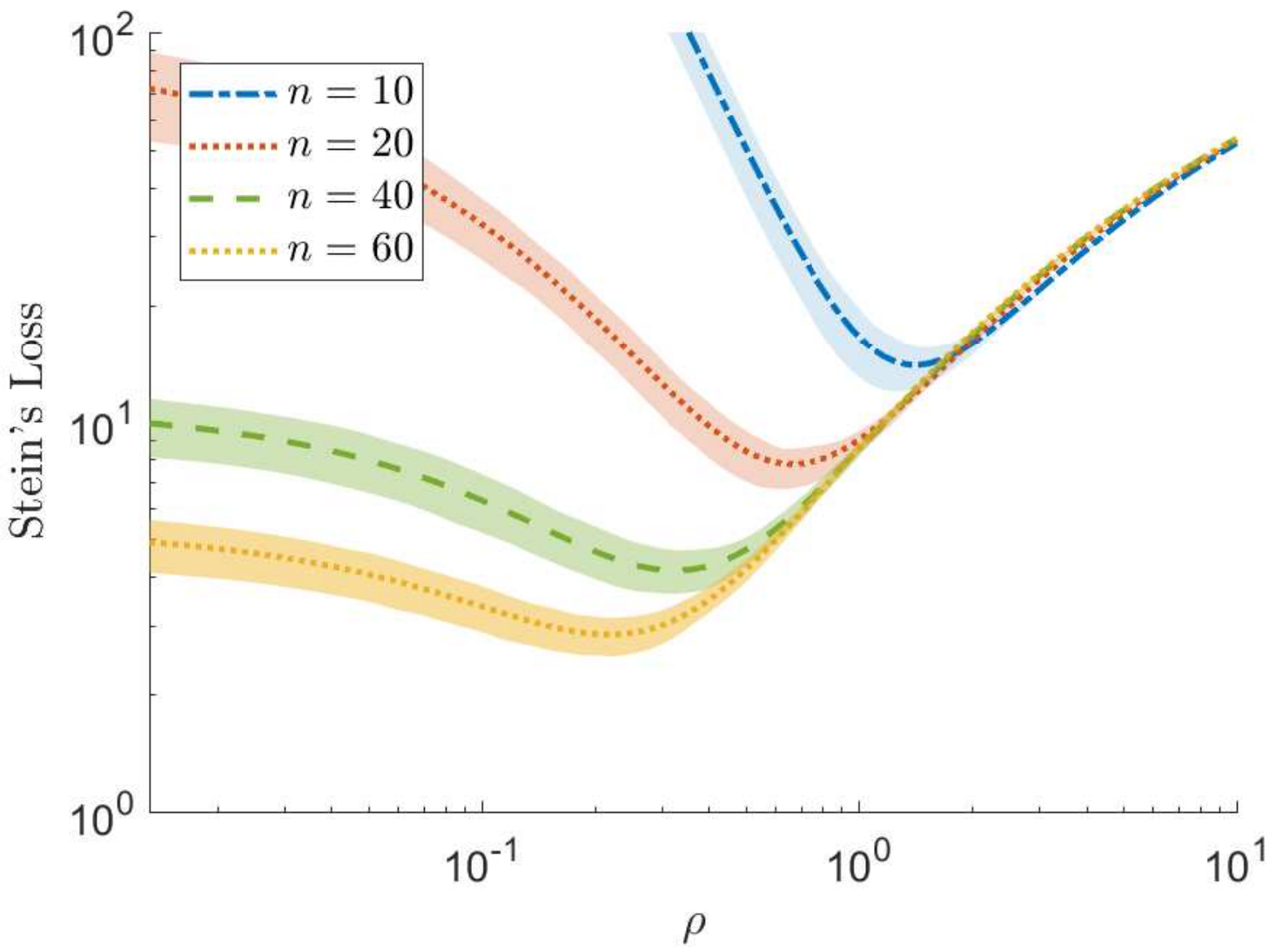}} \hspace{1mm}
	\subfigure[75\% sparsity information]{\label{fig:d200:W75}
		\includegraphics[width=0.31\columnwidth]{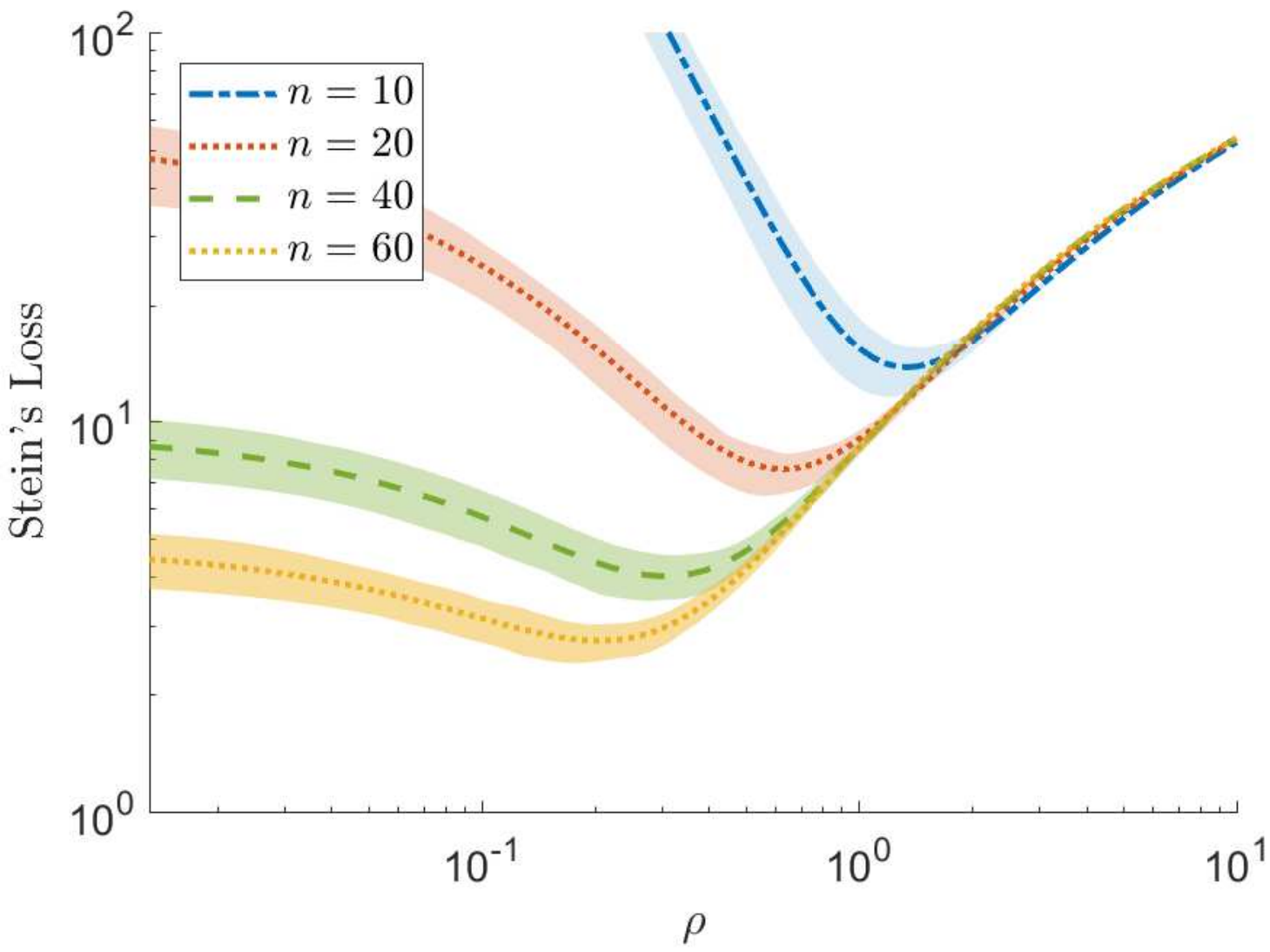}} \hspace{1mm}
	\subfigure[100\% sparsity information]{\label{fig:d200:W100}
		\includegraphics[width=0.31\columnwidth]{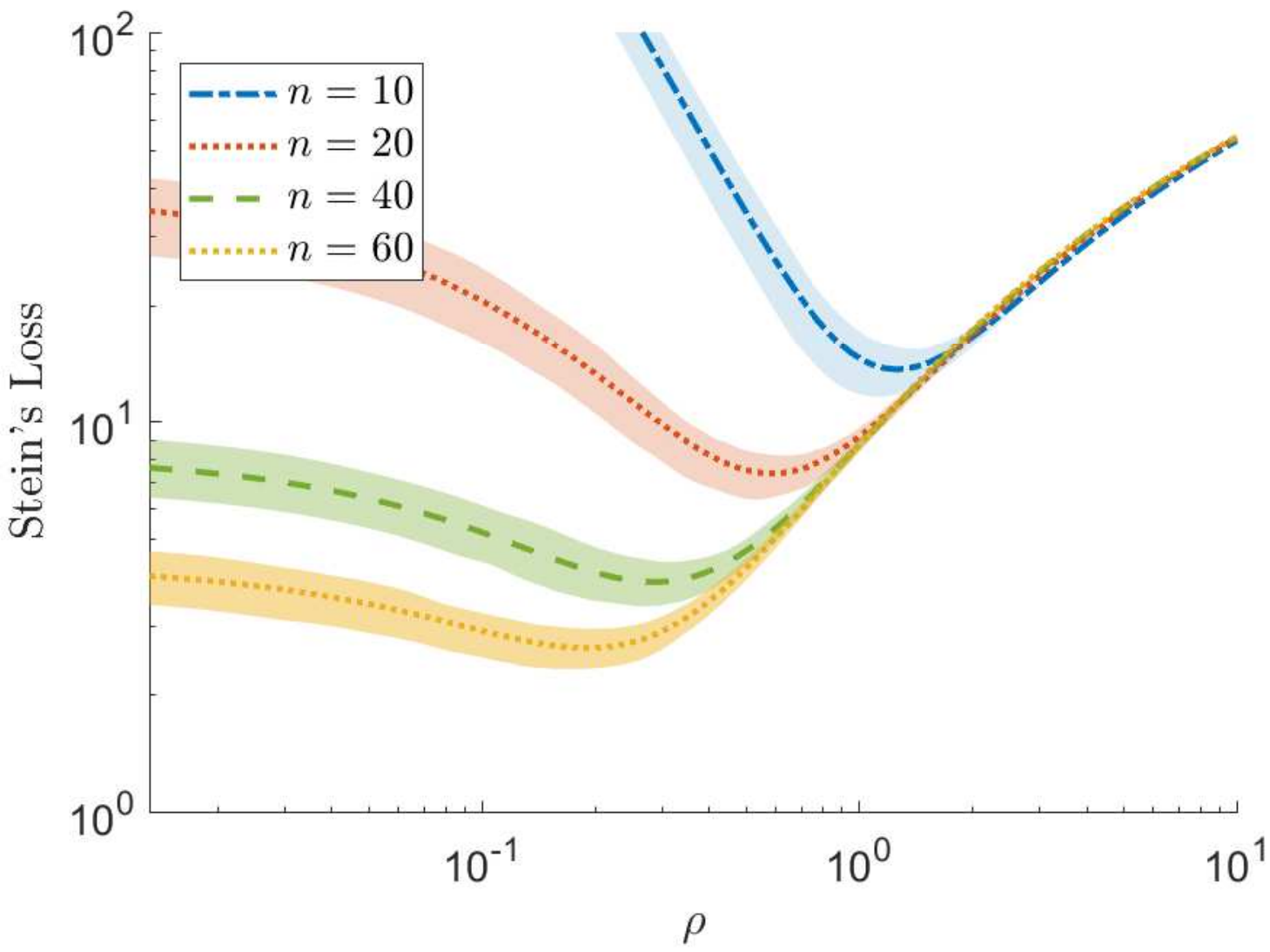}} \hspace{1mm}	
	\subfigure[50\% sparsity information]{\label{fig:d50:W50}
		\includegraphics[width=0.31\columnwidth]{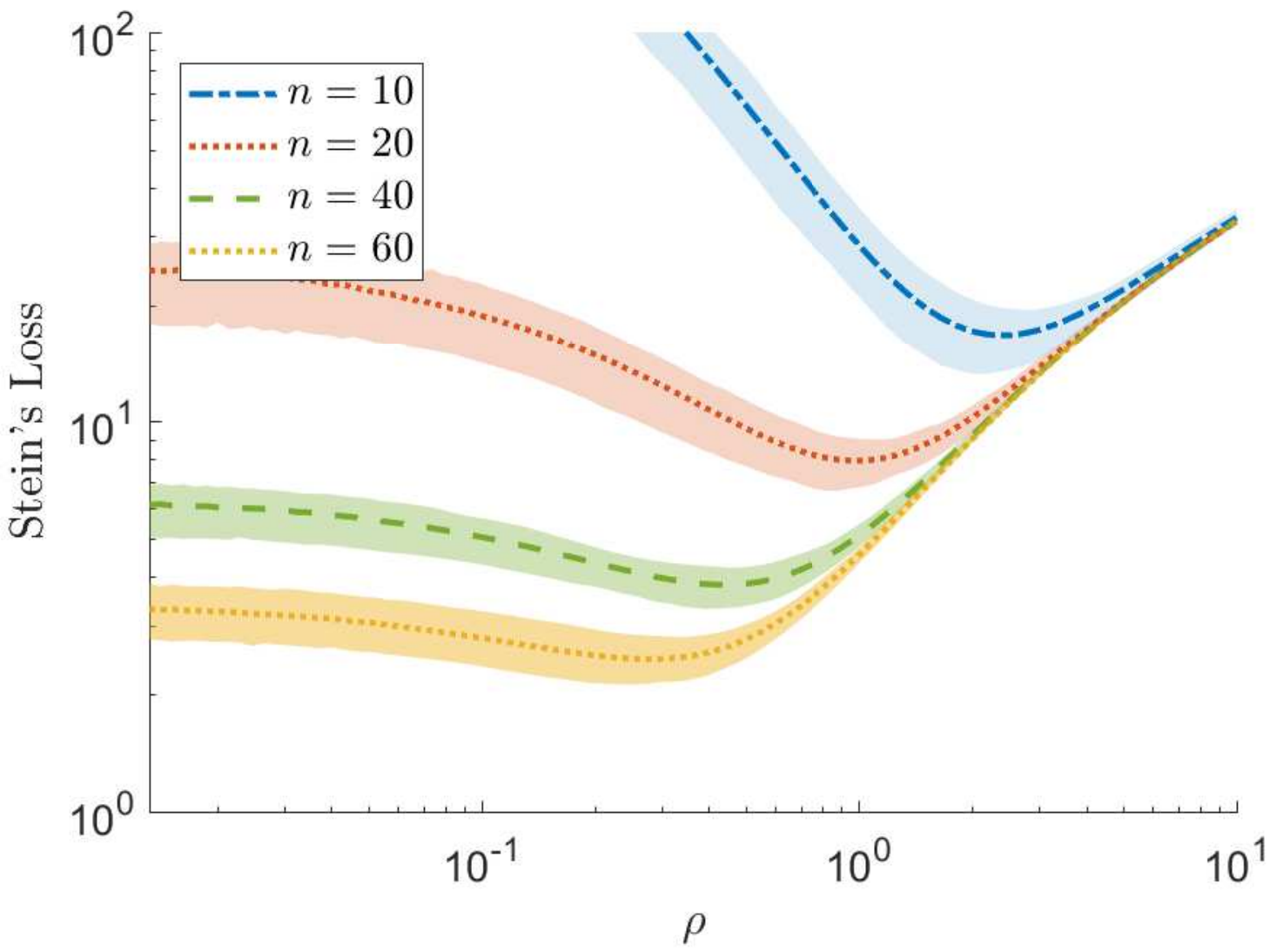}} \hspace{1mm}
	\subfigure[75\% sparsity information]{\label{fig:d50:W75} 
		\includegraphics[width=0.31\columnwidth]{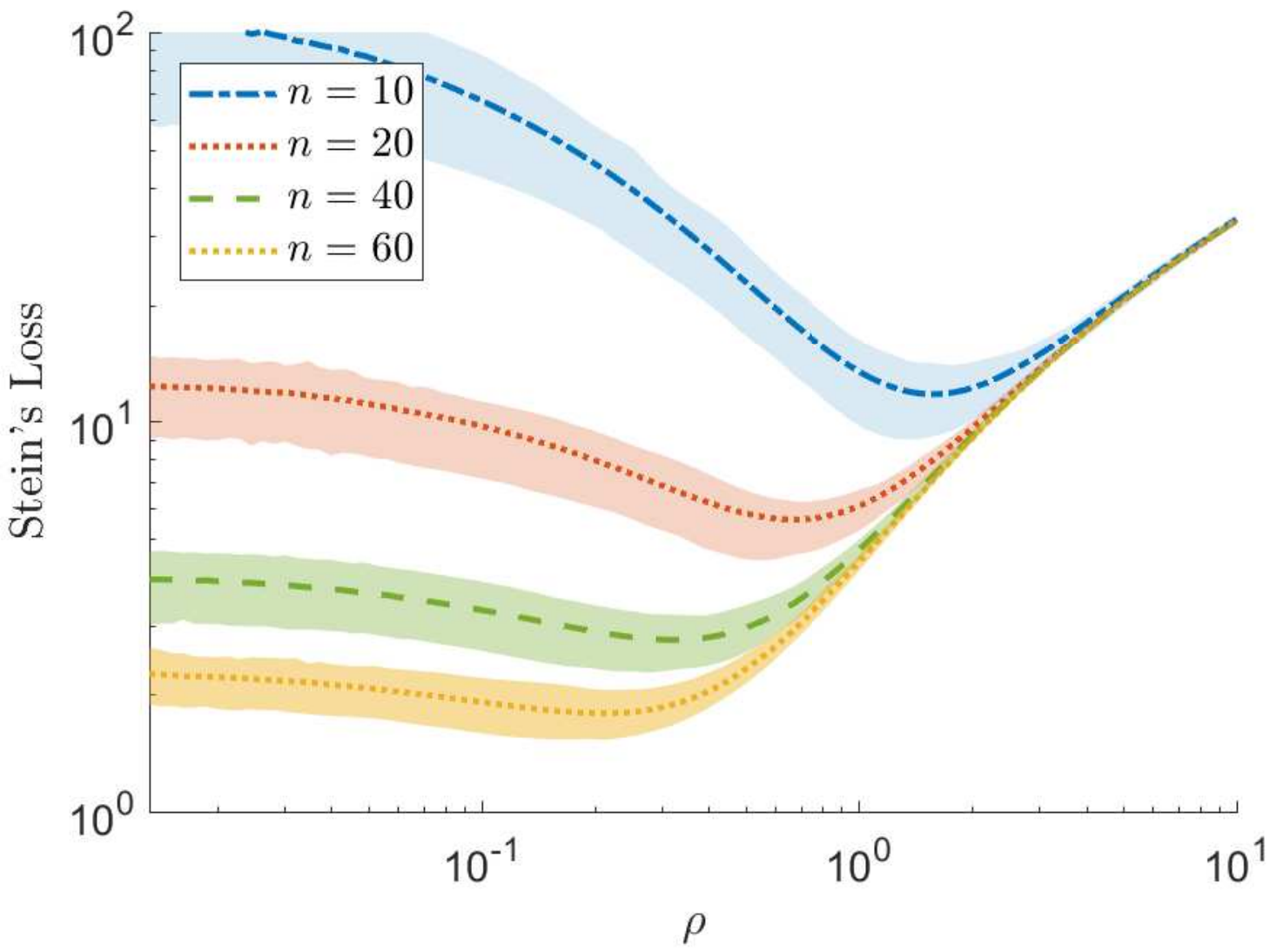}} \hspace{1mm}
	\subfigure[100\% sparsity information]{\label{fig:d50:W100} 
		\includegraphics[width=0.31\columnwidth]{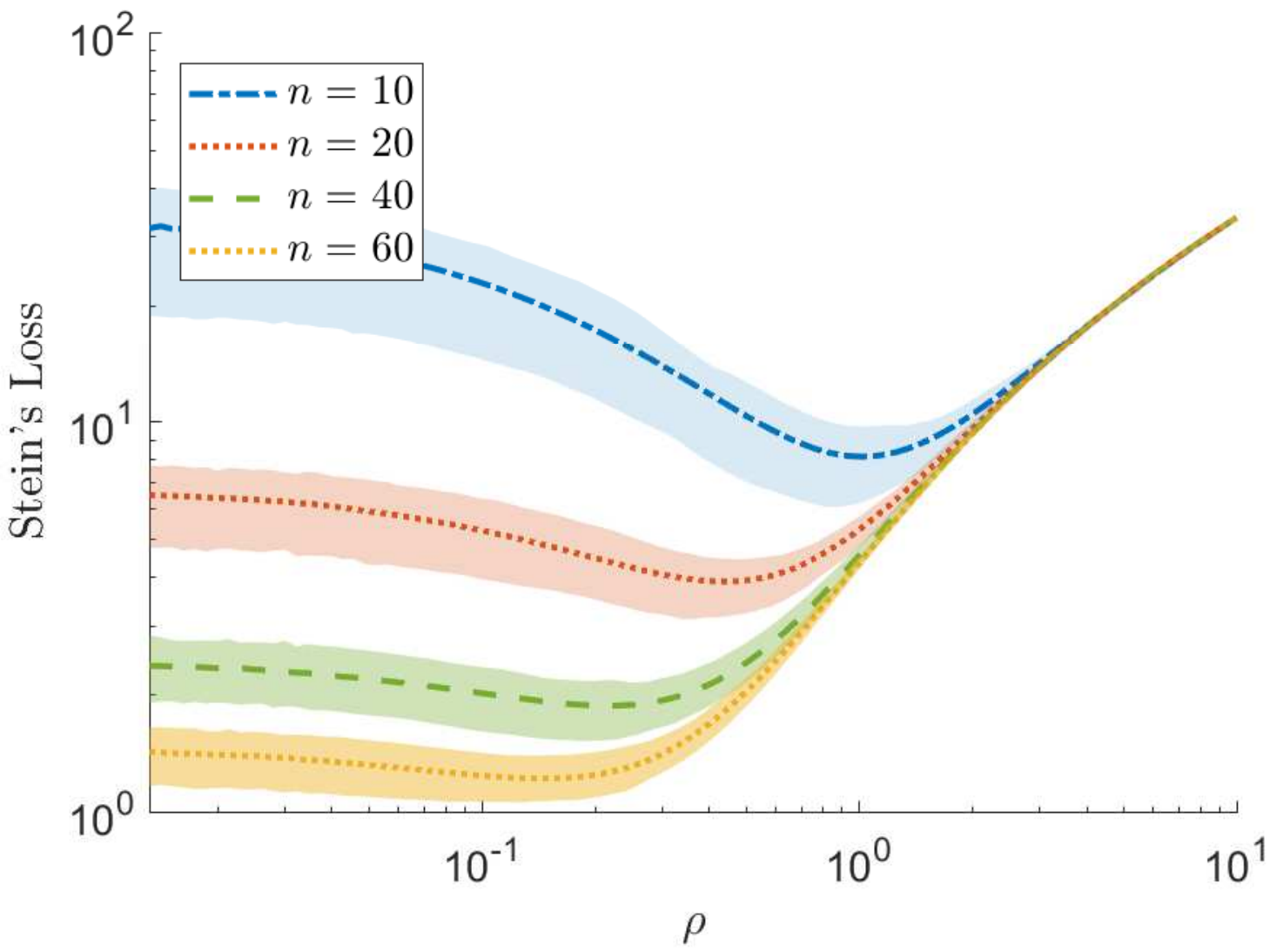}} \hspace{1mm}	
	\caption{Stein's loss of the Wasserstein shrinkage estimator with 50\%, 75\% or 100\% sparsity information as a function of the Wasserstein radius $\rho$ for $d=50\%$ (panels \ref{fig:d200:W50}--\ref{fig:d200:W100}) and $d=12.5\%$ (panels \ref{fig:d50:W50}--\ref{fig:d50:W100}).}
	\label{fig:perf:synthetic:structure}
\end{figure*}


All simulation experiments involve $100$ independent trials. In each trial, we first draw $n\in\{10, 20, 40, 60\}$ independent samples from $\mc N(0,\Sigma_0)$, which are used to compute the sample covariance matrix $\covsa$ and the corresponding precision matrix estimators. Figure~\ref{fig:perf:synthetic} shows Stein's loss of the Wasserstein shrinkage estimator without structure information for $\rho\in[10^{-2}, 10^1]$, the linear shrinkage estimator for $\alpha\in[10^{-5}, 10^0]$ and the $\ell_1$-regularized maximum likelihood estimator for $\beta\in[5\times 10^{-5}, 10^0]$. Lines represent averages, while shaded areas capture the tubes between the empirical 20\% and 80\% quantiles across all $100$ trials. Note that all three estimators approach $\covsa^{-1}$ when their respective tuning parameters tend to zero. As $\covsa$ is rank deficient for $n<p=20$, Stein's loss thus diverges for small tuning parameters when $n=10$.

The best Wasserstein shrinkage estimator in a given trial is defined as the one that minimizes Stein's loss over all $\rho\ge 0$. The best linear shrinkage and $\ell_1$-regularized maximum likelihood estimators are defined analogously. Figure~\ref{fig:perf:synthetic} reveals that the best Wasserstein shrinkage estimators dominate the best linear shrinkage and---to a lesser extent---the best $\ell_1$-regularized maximum likelihood estimators in terms of Stein's loss for all considered parameter settings. The dominance is more pronounced for small sample sizes. We emphasize that Stein's loss depends explicitly on the unknown true covariance matrix $\cov_0$. Thus, Figure~\ref{fig:perf:synthetic} is not available in practice, and the optimal tuning parameters $\rho\opt$, $\alpha\opt$ and $\beta\opt$ cannot be computed exactly. The performance of different precision matrix estimators with {\em estimated} tuning parameters will be studied in Section~\ref{sec:real-data}.



For $d=12.5\%$ and $d=50\%$, the true precision matrix $\cov^{-1}_0$ has many zeros, and prior knowledge of their positions could be used to improve estimator accuracy. To investigate this effect, we henceforth assume that the feasible set $\mc X$ correctly reflects a randomly selected portion of 50\%, 75\% or 100\% of all zeros of $\cov_0^{-1}$, while $\mc X$ contains no (neither correct nor incorrect) information about the remaining zeros. In this setting, we construct the Wasserstein shrinkage estimator by solving problem~\eqref{eq:DROSimplified} numerically. 

Figure~\ref{fig:perf:synthetic:structure} shows Stein's loss of the Wasserstein shrinkage estimator with prior information for $\rho\in[10^{-2}, 10^1]$. Lines represent averages, while shaded areas capture the tubes between the empirical 20\% and 80\% quantiles across $100$ trials. As expected, correct prior sparsity information improves estimator quality, and the more zeros are known, the better. Note that $\cov_0^{-1}$ contains $21.5\%$ zeros for $d=12.5\%$ and $68\%$ zeros for $d=50\%$.


In the last experiment, we investigate the Wasserstein radius $\rho\opt$ of the best Wasserstein shrinkage estimator without sparsity information. Figure~\ref{fig:learning:curve:Peyman:rho} visualizes the average of $\rho\opt$ across 100 independent trials as a function of the sample size~$n$. A standard regression analysis based on the data of Figure~\ref{fig:learning:curve:Peyman:rho} reveals that $\rho\opt$ converges to zero approximately as $n^{-\kappa}$ with $\kappa\approx 61\%$ for $d=12.5\%$, $\kappa\approx 66\%$ for $d= 50\%$ and $\kappa\approx 68\%$ for $d=100\%$. 

\begin{figure}
	\centering
	\includegraphics[width=0.4 \columnwidth]{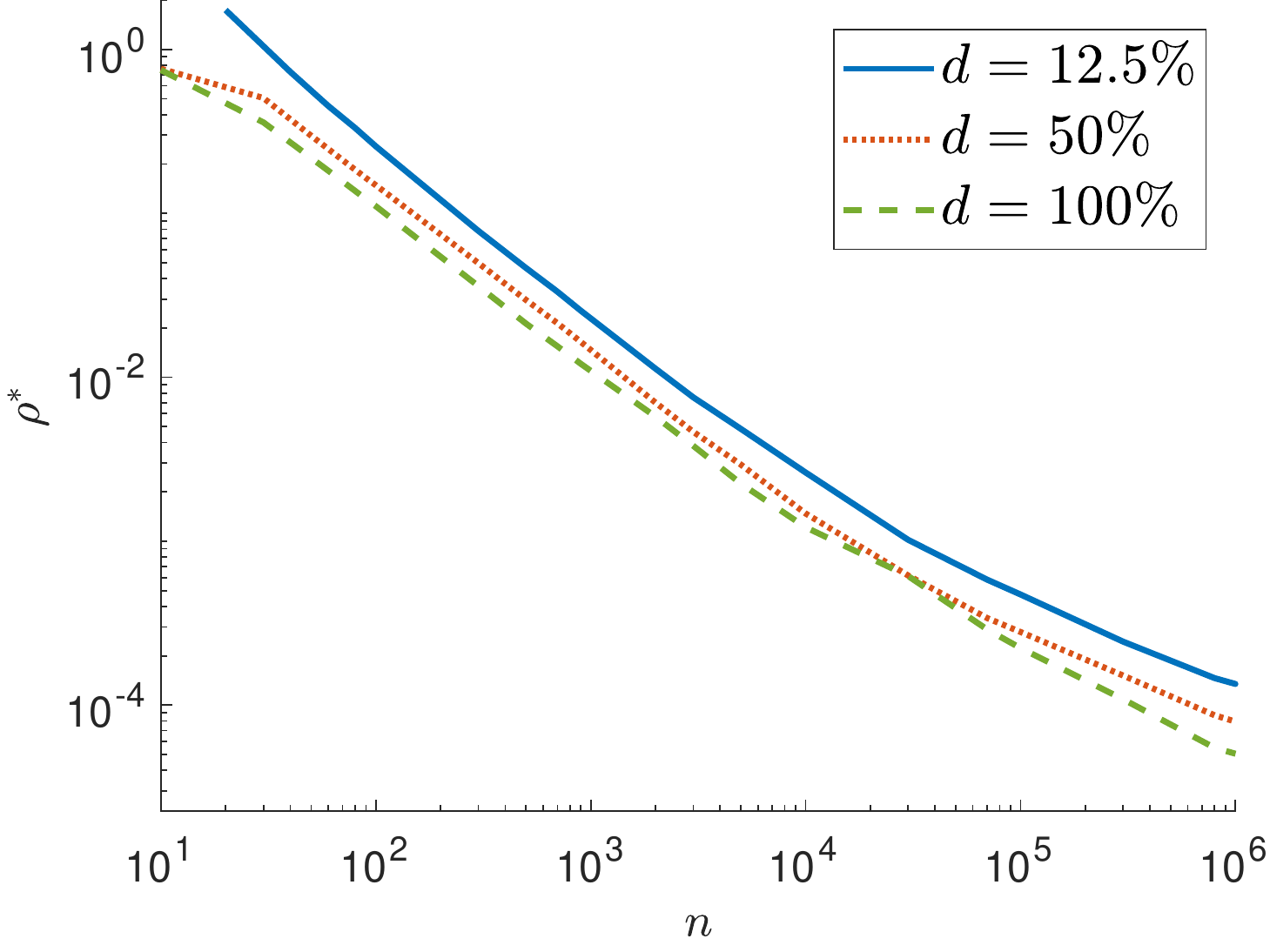}
	\caption{Dependence of the the best Wasserstein radius $\rho\opt$ on the sample size $n$.}
	\label{fig:learning:curve:Peyman:rho}
\end{figure}

\subsection{Experiments with Real Data}
\label{sec:real-data}
We now study the properties of the Wasserstein shrinkage estimator in the context of linear discriminant analysis, portfolio selection and the inference of solar irradiation patterns.

\subsubsection{Linear Discriminant Analysis}
\label{sect:LDA}
Linear discriminant analysis aims to predict the class $y\in\mc Y$, $|\mc Y|<\infty$, of a feature vector $z\in\R^p$ under the assumption that the conditional distribution of $z$ given $y$ is normal with a class-dependent mean $\mu_y\in\R^p$ and class-{\em in}dependent covariance matrix $\cov_0\in\PD^p$ \cite{ref:HastieTibshiraniFriedman-2001}. If all $\mu_y$ and $\cov_0$ are known, the maximum likelihood classifier $\mc C:\R^p\to \mc Y$ assigns $z$ to a class that maximizes the likelihood of observing $y$, that is,
\begin{equation}
\label{eq:ML-classifier}
\mc C(z) \in \arg \Min{y\in\mc Y} (z - \mu_y)^\top \cov_0^{-1} (z - \mu_y).
\end{equation}
In practice, however, the conditional moments are typically unknown and must be inferred from finitely many training samples $(\wh{z}_i, \wh{y}_i)$, $i\le n$. If we estimate $\mu_y$ by the sample average
\[
\wh{\mu}_y = \frac{1}{|\mc I_y|} \sum_{i\in \mc I_y} \wh{x}_i\,,
\]
where $\mc I_y=\{ i\in\{1,\ldots,n\}: \wh y_i=y\}$ records all samples in class $y$, then it is natural to define the residual feature vectors as $\wh \xi_i=\wh z_i-\wh \mu_{\wh y_i}$, $i\le n$. Accounting for Bessel's correction, the conditional distribution of~$\wh\xi_i$ given~$\wh y_i$ is normal with mean 0 and covariance matrix $(|\mc I_{\wh y_i}| -1)\, |\mc I_{\wh y_i}|^{-1} \cov_0$. The marginal distribution of $\wh\xi_i$ thus constitutes a mixture of $|\mc Y|$ normal distributions with mean 0, all of which share the same covariance matrix up to a scaling factor close to unity. As such, the residuals fail to be normally distributed. Moreover, due to their dependence on the sample means, the residuals are correlated. However, if each class accommodates many training samples, then the residuals can approximately be regarded as independent samples from $\mc N(0,\cov_0)$. 

Irrespective of these complications, the sample covariance matrix
\[
	\covsa = \frac{1}{n-|\mc Y|} \sum_{i=1}^n \wh{\xi}_i \wh{\xi}_i ^\top
\]
provides an unbiased estimator for $\cov_0$. Indeed, by the law of total expectation we have
\begin{align*}
	\EE^\mathbb P[\covsa] & = \frac{1}{n-|\mc Y|} \EE^\mathbb P\left[ \sum_{i=1}^n \EE^\mathbb P\left[\left.  \wh{\xi}_i \wh{\xi}_i ^\top \right| \wh y_i \right] \right]\\
	& = \frac{1}{n-|\mc Y|} \sum_{y\in\mc Y}\EE^\mathbb P\left[ \sum_{i\in \mc I_y} \frac{ |\mc I_{\wh y_i}| -1}{|\mc I_{\wh y_i}|} \Sigma \right]
	= \frac{1}{n-|\mc Y|} \sum_{y\in\mc Y}(|\mc I_y| -1) \cov_0 \;=\;\cov_0\,,
\end{align*}
where $\mathbb P$ stands for the unknown true joint distribution of the residuals and class labels. 
In a data-driven setting, the ideal maximum likelihood classifier \eqref{eq:ML-classifier} is replaced with
\be
\label{eq:LDA}
\wh{\mc C}(\xi) = \arg \Min{y\in \mc Y} (\xi - \wh{\mu}_y)^\top \est (\xi - \wh{\mu}_y)\,,
\ee
which depends on the raw data through the sample averages $\wh \mu_y$, $y\in\mc Y$, and some precision matrix estimator~$\est$. The possible choices for $\est$ include the Wasserstein shrinkage estimator without prior information, the linear shrinkage estimator and the $\ell_1$-regularized maximum likelihood estimator, all of which depend on the data merely through $\covsa$. Note that the na\"ive precision matrix estimator~$\covsa^{-1}$ exists only for $n>p$ and is therefore disregarded. All estimators depend on a scalar parameter (the Wasserstein radius $\rho$, the mixing parameter $\alpha$ or the penalty parameter $\beta$) that can be used to tune the performance of the classifier~\eqref{eq:LDA}.

We test the classifier~\eqref{eq:LDA} equipped with different estimators $\est$ on two preprocessed datasets from~\cite{ref:Dettling-2004}:

\begin{enumerate}
	\item The ``{\em colon cancer}'' dataset contains 62 gene expression profiles, each of which involves~2{,}000 features and is classified either as normal tissue (NT) or tumor-affected tissue (TT). The data is split into a training dataset of 29 observations (9 in class NT and 20 in class TT) and a test dataset of~33 observations (13 in class NT and 20 in class TT).
	\item The ``{\em leukemia}'' dataset contains 72 gene expression profiles, each of which involves  3{,}571 features and is classified either as acute lymphocytic leukemia (ALL) or acute myeloid leukemia (AML). The data is split into a training dataset of 38 observations (27 in class ALL and 11 in class AML) and a test dataset of 34 observations (20 in class ALL and 14 in class AML).
\end{enumerate}

Classification is based solely on the first $p\in\{ 20, 40, 80, 100\}$ features of each gene expression profile. We use leave-one-out cross validation on the training data to tune the precision matrix estimator~$\est$ with the goal to maximize the correct classification rate of the classifier~\eqref{eq:LDA}. To keep the computational overhead manageable, we optimize the tuning parameters over the finite search grids
\[
	\rho \in \{10^{\frac{j}{20}-1}: j=0,\ldots,60 \},\quad \alpha \in \{10^{\frac{j}{20}-3}: j=0,\ldots,60 \}\quad \text{and} \quad \beta \in \{10^{\frac{j}{20}-3}: j=0,\ldots,60\}. 
\]
We highlight that, in case of the $\ell_1$-regularized maximum likelihood estimator, cross validation becomes computationally prohibitive for $p>80$ even if the state-of-the-art QUIC routine is used~\cite{ref:Hsieh-2014} to solve the underlying semidefinite programs. In contrast, the Wasserstein and linear shrinkage estimators can be computed and tuned quickly even for $p\gg 100$. Once the optimal tuning parameters are found, we fix them and recalculate $\est$ on the basis of the entire training dataset. Finally, we substitute the resulting precision matrix estimator into the classifier~\eqref{eq:LDA} and evaluate its correct classification rate on the test dataset. The test results are reported in Table~\ref{table:LDA}. We observe that the Wasserstein shrinkage estimator frequently outperforms the linear shrinkage and $\ell_1$-regularized maximum likelihood estimators, especially for higher values of $p$.

\begin{table}[htbp]
	\centering
	\caption{Correct classification rate of the classifier~\eqref{eq:LDA} instantiated with different precision matrix estimators. The best result in each experiment is highlighted in bold.}
	\begin{tabular}{|l|r|r|r|r|r|r|r|r|}
		\hline
		& \multicolumn{4}{c|}{Colon cancer dataset}    & \multicolumn{4}{c|}{Leukemia dataset} \bigstrut\\
		\hline
		Estimator & $p=20$  & $p=40$  & $p=80$  & $p = 100$ & $p=20$  & $p=40$  & $p=80$  & $p = 100$ \bigstrut\\
		\hline
		Wasserstein shrinkage & \textbf{72.73} & 75.76 & \textbf{78.79} & \textbf{75.76} & \textbf{73.53} & 67.65 & \textbf{91.18} & \textbf{91.18} \bigstrut\\
		\hline
		Linear shrinkage & 57.58 & 72.73 & 72.73 & 72.73 & 70.59 & \textbf{70.59} & 82.35 & 82.35 \bigstrut\\
		\hline
		$\ell_1$-regularized ML & \textbf{72.73} & \textbf{78.79} & \textbf{78.79} & 72.73 & 70.59 & 64.71 & 82.35 & 82.35 \bigstrut\\
		\hline
	\end{tabular}%
	\label{table:LDA}
\end{table}%

\subsubsection{Minimum Variance Portfolio Selection}
\label{sect:MinVar}
Consider the minimum variance portfolio selection problem without short sale constraints \cite{ref:Jagannathan-2003} 
\[
\begin{array}{cl}
	\Min{w\in\R^p} & w^\top \cov_0 w \\
	\st & \mathbbm{1}^\top w = 1\,,
\end{array}
\]
where the portfolio vector $w\in\R^p$ captures the percentage weights of initial capital allocated to $p$ different assets with random returns, $\mathbbm{1}\in\R^p$ stands for the vector of ones, and $\cov_0\in\PD^p$ denotes the covariance matrix of the asset returns. The objective represents the variance of the portfolio return, which is strictly convex in $w$ thanks to the positive definiteness of $\cov_0$. The unique optimal solution of this portfolio selection problem is given by $w\opt = \cov_0^{-1}\mathbbm{1}/\mathbbm{1}^\top \cov_0^{-1} \mathbbm{1}$. In practice, the unknown true precision matrix $\cov_0^{-1}$ must be replaced with an estimator $\est$, which gives rise to the estimated minimum variance portfolio $\wh{w}\opt = \est \mathbbm{1}/\mathbbm{1}^\top \est \mathbbm{1}$.

A vast body of literature in finance focuses on finding accurate precision matrix estimators for portfolio construction, see, {\em e.g.}, \cite{ref:Miguel-2009, ref:Ledoit-2004:honey,ref:Torri-2017}. In the following we compare the minimum variance portfolios based on the Wasserstein shrinkage estimator without structural information, the linear shrinkage estimator and $\ell_1$-regularized maximum likelihood estimator on two preprocessed datasets from the Fama-French online data library:\footnote{See \url{http://mba.tuck.dartmouth.edu/pages/faculty/ken.french/data_library.html} (accessed January 2018)} the ``{\em 48 industry portfolios}'' dataset (FF48) and the ``{\em 100 portfolios formed on size and book-to-market}'' dataset (FF100). Recall that the estimators depend on the data only through the sample covariance matrix $\covsa$, which is computed from the residual returns relative to the sample means and thus needs to account for Bessel's correction. The datasets both consist of monthly returns for the period from January 1996 to December 2016. The first 120 observations from January~1996 to December~2005 serve as the training dataset. The optimal tuning parameters that minimize the portfolio variance are estimated via leave-one-out cross validation on the training dataset using the finite search grids
\[
	\rho \in \{10^{\frac{j}{100}-2}: j=0,\ldots,200 \},\quad \alpha \in \{10^{\frac{j}{100}-2}: j=0,\ldots,200 \}\quad \text{and} \quad \beta \in \{10^{\frac{j}{50}-4}: j=0,\ldots,200 \}. 
\]
The out-of-sample performance of the minimum variance portfolio corresponding to a particular precision matrix estimator is then evaluated using the rolling horizon method over the period from January 2006 to December 2016, where the sample covariance matrix needed as an input for the precision matrix is re-estimated every three months based on the most recent 120 observations (10 years), while the tuning parameter is kept fixed. The resulting out-of-sample mean, standard deviation and Sharpe ratio of the portfolio return are reported in Table~\ref{table:portfolio}. While the $\ell_1$-regularized maximum likelihood estimator yields the portfolio with the lowest standard deviation for both datasets, the Wasserstein shrinkage estimator always generates the highest mean and, maybe surprisingly, the highest Sharpe ratio.

\begin{table}[htbp]
	\centering
	\caption{Standard deviation, mean and Sharpe ratio of the minimum variance portfolio based on different estimators. The best result in each experiment is highlighted in bold.}
	\begin{tabular}{|l|r|r|r|r|r|r|}
		\hline
		& \multicolumn{3}{c|}{FF48 dataset} & \multicolumn{3}{c|}{FF100 dataset} \bigstrut\\
		\hline
		Estimator & \multicolumn{1}{l|}{std} & \multicolumn{1}{l|}{mean} & \multicolumn{1}{l|}{Sharpe} & \multicolumn{1}{l|}{std} & \multicolumn{1}{l|}{mean} & \multicolumn{1}{l|}{Sharpe} \bigstrut\\
		\hline
		Wasserstein shrinkage & 3.146 & \textbf{0.701} & \textbf{0.223} & 3.518 & \textbf{1.079} & \textbf{0.307}  \bigstrut\\
		\hline
		Linear shrinkage & 3.152 & 0.688 & 0.218 & 3.484 & 0.965 & 0.277  \bigstrut\\
		\hline
		$\ell_1$-regularized ML & \textbf{3.077} & 0.668 & 0.217 & \textbf{3.423} & 1.010 & 0.295 \bigstrut\\
		\hline
	\end{tabular}%
	\label{table:portfolio}%
\end{table}%

\subsubsection{Inference of Solar Irradiation Patterns}
\label{sec:solar}
In the last experiment we aim to estimate the spatial distribution of solar irradiation in Switzerland using the ``{\em surface incoming shortwave radiation}'' (SIS) data provided by MeteoSwiss.\footnote{See \url{http://www.meteoschweiz.admin.ch/data/products/2014/raeumliche-daten-globalstrahlung.html} (accessed January 2018)} The SIS data captures the horizontal solar irradiation intensities in $\text{W/m}^2$ for pixels of size 1.6km by 2.3km based on the effective cloud albedo, which is derived from satellite imagery. The dataset spans 13 years from 2004 to 2016, with a total number of 4{,}749 daily observations. We deseasonalize the time series of each pixel as follows. First, we divide the original time series by a shifted sinusoid with a yearly period, whose baseline level, phase and amplitude are estimated via ordinary least squares regression. Next, we subtract unity. The resulting deseasonalized time series is viewed as the sample path of a zero mean Gaussian noise process. This approach relies on the assumption that the mean and the standard deviation of the original time series share the same seasonality pattern. It remains to estimate the joint distribution of the pixel-wise Gaussian white noise processes, which is fully determined by the precision matrix of the deseasonalized data. We estimate the precision matrix using the Wasserstein shrinkage, linear shrinkage and $\ell_1$-regularized maximum likelihood estimators. As each pixel represents a geographical location and as the solar irradiation intensities at two distant pixels are likely to be conditionally independent given the intensities at all other pixels, it is reasonable to assume that the precision matrix is sparse; see also~\cite{ref:Das-2017, ref:Wiesel-2010}. Specifically, we assume here that the solar irradiation intensities at two pixels indexed by $(i,j)$ and $(i',j')$ are conditionally independent and that the corresponding entry of the precision matrix vanishes whenever $|i-i'|+|j-j'|> 3$. This sparsity information can be used to enhance the basic Wasserstein shrinkage estimator. 

\begin{figure}
	\centering
	\includegraphics[width=8cm]{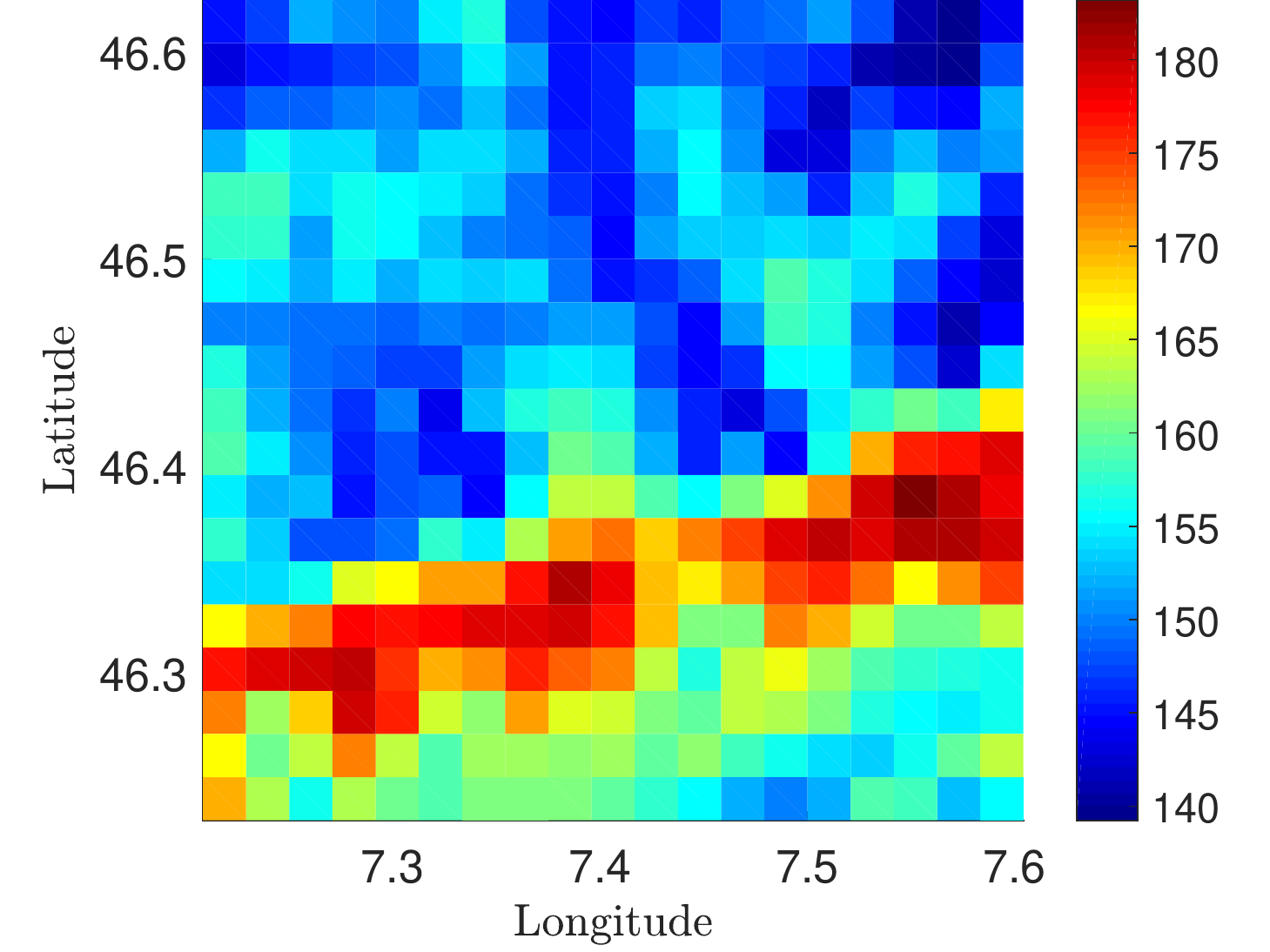}
	\caption{Average solar irradiation intensities ($\text{W/m}^2$) for the Diablerets region in Switzerland.}
	\label{figure:real:solar:mean}
\end{figure}

Consider now the Diablerets region of Switzerland, which is described by a spatial matrix of 20$\times$20 pixels. Thus, the corresponding precision matrix has dimension 400$\times$400. The average daily solar irradiation intensities within the region of interest are visualized in Figure~\ref{figure:real:solar:mean}. We note that the sunshine exposure is highly variable due to the heterogeneous geographical terrain characterized by a high mountain range in the south intertwined with deep valleys in the north. In order to assess the quality of a specific precision matrix estimator, we use $K$-fold cross validation with $K = 13$. 
The $k$-th fold comprises all observations of year~$k$ and is used to construct the estimator $\est_k$. The data of the remaining 12 years, without year~$k$, are used to compute the empirical covariance matrix $\covsa_{- k}$. The estimation error of $\est_k$ is then measured via Stein's loss
\[
L (\est_k, \covsa_{- k}) = -\log \det (\est_k \, \covsa_{- k}) + \inner{\est_k }{\covsa_{- k}} - p.
\]
We emphasize that here, in contrast to the experiment with synthetic data, $ \covsa_{- k}$ is used as a proxy for the unknown true covariance matrix $\cov$. Figure~\ref{fig:perf:solar} shows Stein's loss of the Wasserstein shrinkage estimator with and without structure information for $\rho\in[10^{-2}, 10^0]$, the linear shrinkage estimator for $\alpha\in[10^{-3}, 2\times10^{-2}]$ and the $\ell_1$-regularized maximum likelihood estimator for $\beta\in[10^{-5}, 10^{-3}]$. Lines represent averages, while shaded areas capture the tubes between the best- and worst-case loss realizations across all $K$ folds. 

The Wasserstein shrinkage estimator with structure information reduces the minimum average loss by 13.5\% relative to the state-of-the-art $\ell_1$-regularized maximum likelihood estimator. 
Moreover, the average runtimes for computing the different estimators amount to 51.84s for the Wasserstein shrinkage estimator with structural information (Algorithm~\ref{algo:QuadraticApprox}), 0.08s for the Wasserstein shrinkage estimator without structural information (analytical formula and bisection algorithm), 0.01s for the linear shrinkage estimator (analytical formula) and 1493.61s for the $\ell_1$-regularized maximum likelihood estimator (QUIC algorithm \cite{ref:Hsieh-2014}).

\begin{figure*} [t]
	\centering
	\subfigure[Wasserstein shrinkage]{\label{fig:solar:W} 
		\includegraphics[width=0.31\columnwidth]{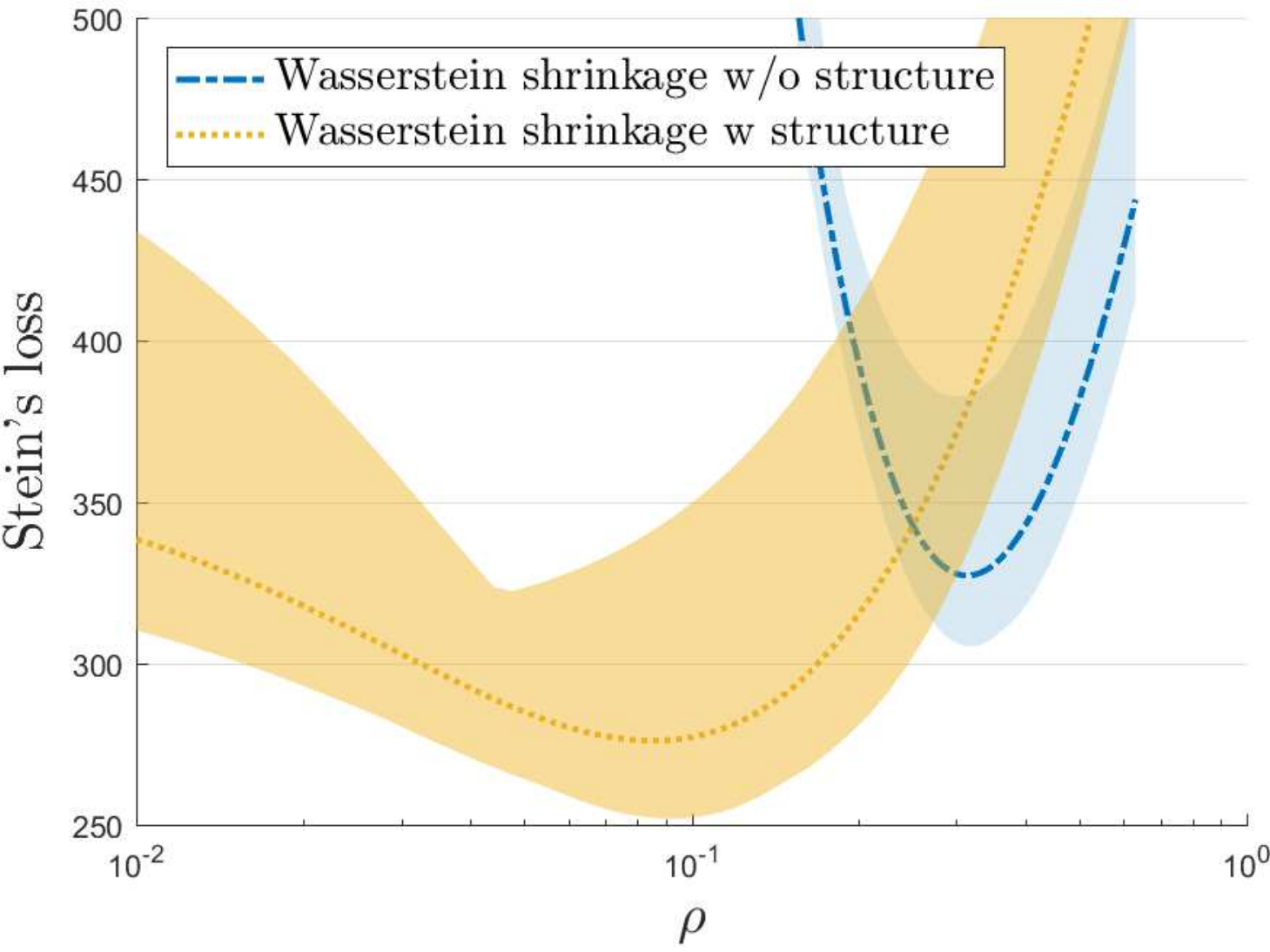}} \hspace{1mm}
	\subfigure[Linear shrinkage]{\label{fig:solar:S}
		\includegraphics[width=0.31\columnwidth]{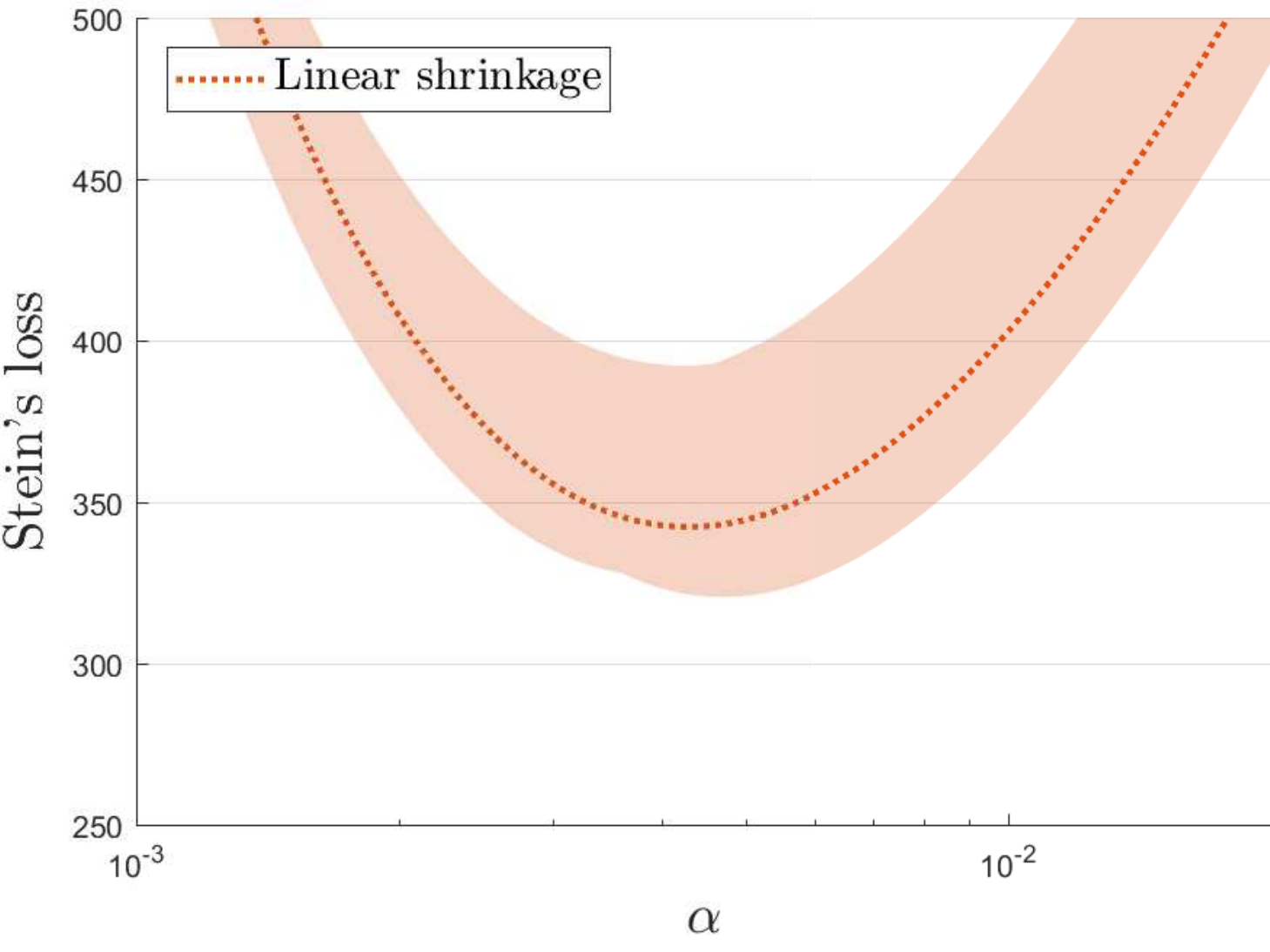}} \hspace{1mm}
	\subfigure[$\ell_1$-regularized ML]{\label{fig:solar:l1}
		\includegraphics[width=0.31\columnwidth]{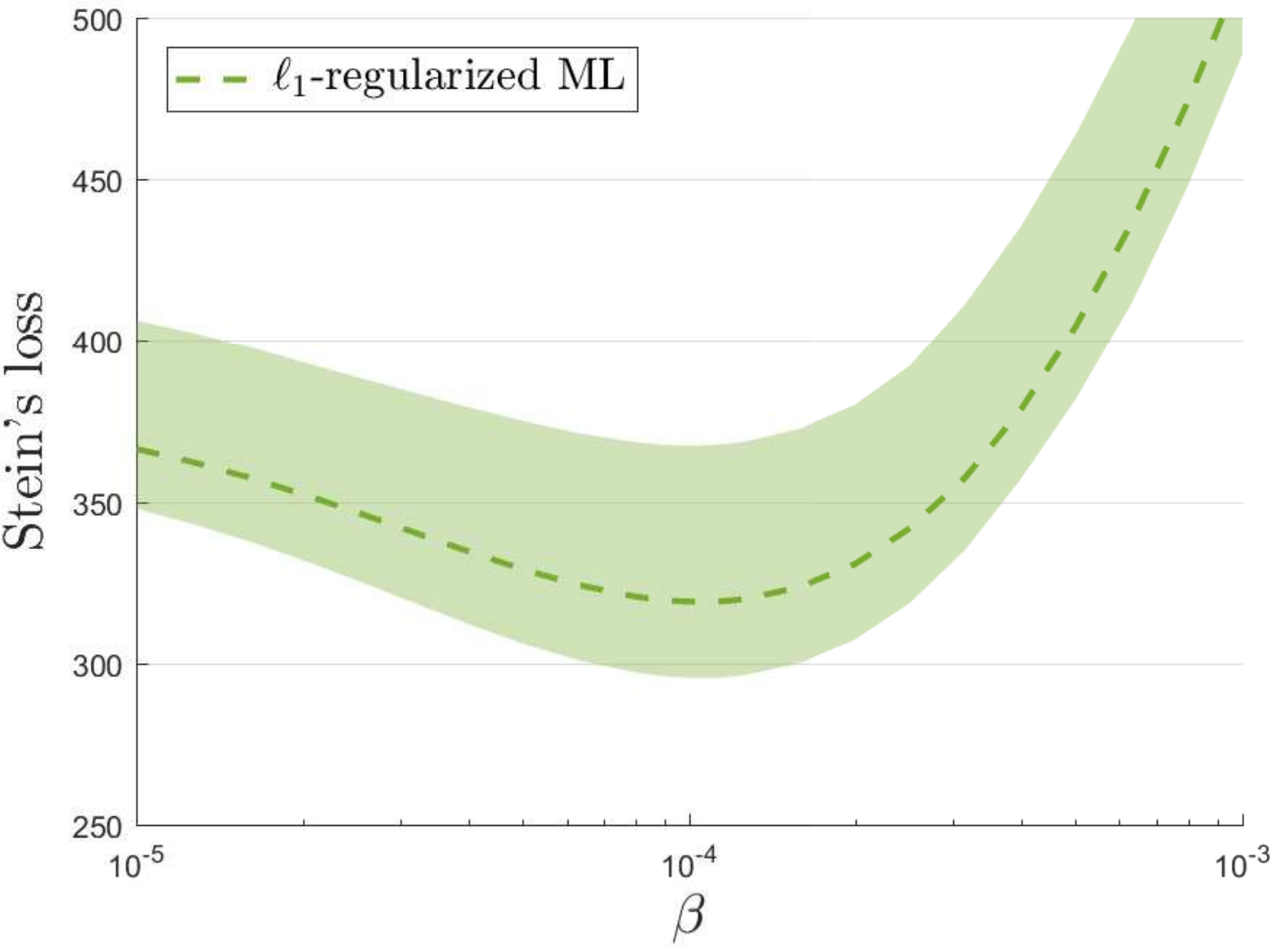}} \hspace{1mm}	
	\caption{Stein's loss of the Wasserstein shrinkage, linear shrinkage and $\ell_1$-regularized maximum likelihood estimators as a function of their respective tuning parameters.}
	\label{fig:perf:solar}
\end{figure*}

%
%
%

\paragraph{\bf Acknowledgments}
We gratefully acknoweldge financial support from the Swiss National Science Foundation under grants BSCGI0\_157733 and P2EZP2\_165264.

\end{document}